\documentclass[11pt, reqno]{amsart}

\usepackage[T1]{fontenc}
\usepackage[latin1]{inputenc}

\usepackage{amsthm, amsmath, amssymb,booktabs,xcolor,graphicx}
\usepackage[margin=1in]{geometry}
\usepackage[shortlabels]{enumitem}
\usepackage{scalerel,stackengine}
\usepackage[textsize=scriptsize,colorinlistoftodos]{todonotes}
\usepackage{booktabs}
\usepackage[hidelinks]{hyperref}
\usepackage[noabbrev,capitalize,nameinlink]{cleveref}
\crefformat{equation}{(#2#1#3)}
\crefrangeformat{equation}{(#3#1#4) to~(#5#2#6)}
\crefmultiformat{equation}{(#2#1#3)}
{ and~(#2#1#3)}{, (#2#1#3)}{ and~(#2#1#3)}

\usepackage{tikz}
\usetikzlibrary{decorations.markings}
\usetikzlibrary{arrows}
\usetikzlibrary{calc}

\usepackage[final]{showkeys}

\colorlet{refkey}{orange!20}
\colorlet{labelkey}{blue!30}

\newtheorem{theorem}{Theorem}[section]
\newtheorem{proposition}[theorem]{Proposition}
\newtheorem{prop}[theorem]{Proposition}
\newtheorem{lemma}[theorem]{Lemma}
\newtheorem{claim}[theorem]{Claim}
\newtheorem{cor}[theorem]{Corollary}
\newtheorem{conj}[theorem]{Conjecture}
\newtheorem*{question*}{Question}

\theoremstyle{definition}
\newtheorem{defn}[theorem]{Definition}
\newtheorem{problem}[theorem]{Problem}

\newtheorem*{definition*}{Definition}
\newtheorem{example}[theorem]{Example}

\newtheorem{alg}[theorem]{Algorithm}

\theoremstyle{remark}
\newtheorem{rem}[theorem]{Remark}

\numberwithin{equation}{section}

\newcommand{\abs}[1]{\left\lvert#1\right\rvert}

\newcommand{\floor}[1]{\left\lfloor #1 \right\rfloor}
\newcommand{\ceil}[1]{\left\lceil #1 \right\rceil}
\newcommand{\paren}[1]{\left( #1 \right)}
\newcommand{\sqb}[1]{\left[ #1 \right]}
\newcommand{\set}[1]{\left\{ #1 \right\}}

\newcommand{\codeg}{\text{codeg}}

\renewcommand{\lg}{\text{lg}}
\newcommand{\sm}{\text{sm}}

\stackMath
\newcommand\widecheck[1]{
\savestack{\tmpbox}{\stretchto{
  \scaleto{
    \scalerel*[\widthof{\ensuremath{#1}}]{\kern-.6pt\bigwedge\kern-.6pt}
    {\rule[-\textheight/2]{1ex}{\textheight}}
  }{\textheight}
}{0.5ex}}
\stackon[1pt]{#1}{\scalebox{-1}{\tmpbox}}
}
\DeclareMathOperator{\Tr}{Tr}

\DeclareMathOperator{\SAT}{-SAT}
\DeclareMathOperator{\type}{type}
\DeclareMathOperator{\wt}{wt}

\newcommand{\ol}{\overline}

\newcommand{\wc}{\widecheck}

\newcommand{\EE}{\mathbb{E}}

\newcommand{\RR}{\mathbb{R}}

\newcommand{\NN}{\mathbb{N}}

\newcommand{\cB}{\mathcal{B}}

\newcommand{\cF}{\mathcal{F}}
\newcommand{\mcE}{\mathcal{E}}
\newcommand{\mcF}{\mathcal{F}}
\newcommand{\mcG}{\mathcal{G}}

\newcommand{\cH}{\mathcal{H}}
\newcommand{\cI}{\mathcal{I}}
\newcommand{\mcH}{\mathcal{H}}

\newcommand{\cG}{\mathcal{G}}
\newcommand{\cP}{\mathcal{P}}

\newcommand{\cS}{\mathcal{S}}

\title{Enumerating $k$-SAT functions}

\author{Dingding Dong}
\author{Nitya Mani}
\author{Yufei Zhao}

\thanks{Mani was supported by the NSF Graduate Research Fellowship Program and a Hertz Graduate Fellowship.}
\thanks{Zhao was supported by NSF award DMS-1764176, NSF CAREER award DMS-2044606, a Sloan Research Fellowship, and the MIT Solomon Buchsbaum Fund.}

\address{Dong: Department of Mathematics, Harvard University, Cambridge, MA 02138, USA}
\email{ddong@math.harvard.edu}

\address{Mani \& Zhao: Department of Mathematics, Massachusetts Institute of Technology, Cambridge, MA 02139, USA}
\email{\{nmani,yufeiz\}@mit.edu}

\begin{document}

\begin{abstract}
    How many $k$-SAT functions on $n$ boolean variables are there? 
    What does a typical such function look like?
    Bollob\'as, Brightwell, and Leader conjectured that, for each fixed $k \ge 2$, the number of $k$-SAT functions on $n$ variables is $(1+o(1))2^{\binom{n}{k} + n}$, or equivalently: a $1-o(1)$ fraction of all $k$-SAT functions are unate, i.e., monotone after negating some variables. They proved a weaker version of the conjecture for $k=2$. The conjecture was confirmed for $k=2$ by Allen and $k=3$ by Ilinca and Kahn.
    
    We show that the problem of enumerating $k$-SAT functions is equivalent to a Tur\'an density problem for partially directed hypergraphs.
    Our proof uses the hypergraph container method.
    Furthermore, we confirm the Bollob\'as--Brightwell--Leader conjecture for $k=4$ by solving the corresponding Tur\'an density problem. 
    Our solution applies a recent result of F\"uredi and Maleki on the minimum triangular edge density in a graph of given edge density.
    In an appendix (by Nitya Mani and Edward Yu), we further confirm the $k=5$ case of the conjecture via a brute force computer search.
\end{abstract}

\maketitle

\section{Introduction}

\subsection{Background}
We study the following basic question on boolean functions:
\begin{quote}
How many $k$-SAT functions on $n$ boolean variables are there? 
What does a typical such function look like?
\end{quote}
This question was first studied by Bollob\'as, Brightwell, and Leader~\cite{BBL03}. 
We focus on the regime where $k$ is fixed and $n \to \infty$.
We will consider $k$-SAT functions in their disjunctive normal form (DNF). It would be an equivalent problem to enumerate $k$-SAT functions in their conjunctive normal form (CNF) since the negation of a DNF is a CNF and vice-versa. 
For our purpose, a \emph{$k$-SAT function} on $n$ boolean variables is a function $f \colon \{0,1\}^n \to \{0,1\}$ of the form
\[
f(x_1, \dots, x_n) = C_1 \lor C_2 \lor  \cdots \lor C_m,
\]
where each $C_i$ has the form $z_1 \land \cdots \land z_k$ with $z_1, \dots, z_k \in \{x_1, \ol {x_1}, \dots, x_n, \ol {x_n}\}$.
Here we call $x_1, \dots, x_n$ the \emph{variables}.
Each of $x_i$ and $\ol{x_i}$ is a called a \emph{literal} (\emph{positive literal} and \emph{negative literal}, respectively). Each \emph{clause} $C_i$ is a conjunction (``and'') of $k$ literals.
We further restrict that every clause uses $k$ distinct variables (e.g., both $x_1 \land x_2$ and $x_1 \land x_2 \land \ol{x_2}$ are invalid 3-SAT clauses).
This restriction does not lose any generality (the first example can be replaced by $(x_1 \land x_2 \land x_3) \lor (x_1 \land x_2 \land \ol{x_3})$ and the second example is a clause that is never satisfied and so can be deleted).
To simplify notation, we will drop the ``and'' symbol $\land$ when writing a clause.
A \emph{formula} is a set of clauses. 
For instance, the $2$-SAT formula 
$(x_1 \land x_2) \lor  (\ol{x_1}\land x_3) \lor (x_3 \land x_4)$ is written as 
$\{x_1x_2, \ol{x_1} x_3, x_3 x_4\}$. 
Every $k$-SAT function has a $k$-SAT formula, but different $k$-SAT formulae may correspond to the same $k$-SAT function.

Given the importance of $k$-SAT functions, it is a natural question to try to understand how rich this family of functions is. 
While the total number of functions $\{0,1\}^n \to \{0,1\}$ is $2^{2^n}$, the number of $k$-SAT functions is significantly smaller. 
As an easy upper bound, since there are $2^k \binom{n}{k}$ possible clauses, the number of $k$-SAT formulae is $2^{2^k \binom{n}{k}}$.
So the number of $k$-SAT functions is at most $2^{2^k \binom{n}{k}}$,
which is significantly smaller than $2^{2^n}$ for a fixed $k$ and large $n$.
The actual number of $k$-SAT functions turns out to be considerably smaller than even this upper bound.

A $k$-SAT formula is \emph{monotone} if it only uses positive literals.
A $k$-SAT function is monotone if it has a monotone $k$-SAT formula.
There are $\binom{n}{k}$ possible monotone clauses, and every monotone $k$-SAT formula produces a unique monotone $k$-SAT function, and so there are $2^{\binom{n}{k}}$ monotone $k$-SAT functions.

A $k$-SAT function or formula is \emph{unate} if it is monotone after replacing some variables with their negations (e.g., $\{\ol{x_1}x_2, \ol{x_1}x_2,x_2x_3\}$ is unate but $\{\ol{x_1}x_2, x_1x_2, x_2x_3\}$ is not).
The number of unate $k$-SAT formulae that use all $n$ variables is at least
\[
2^n \left(2^{\binom{n}{k}} - n 2^{\binom{n-1}{k}} \right) = (1+o(1)) 2^{n + \binom{n}{k}},
\]
for fixed $k$ as $n \to \infty$. Indeed, for each variable $x_i$, there are $2^n$ choices as to whether to use it as a positive literal or a negative literal and there are at least $2^{\binom{n}{k}} - n 2^{\binom{n-1}{k}}$ monotone formulae that use all $n$ variables. 
All unate $k$-SAT formulae represent distinct functions, and thus the number of unate $k$-SAT functions on $n$ variables is at least $(1+o(1)) 2^{n + \binom{n}{k}}$.

Bollob\'as, Brightwell, and Leader~\cite{BBL03} conjectured the following.

\begin{conj} \label{conj:BBL}
Fix $k \ge 2$. 
The number of $k$-SAT functions on $n$ boolean variables is $(1+o(1)) 2^{n + \binom{n}{k}}$.
Equivalently: a $1-o(1)$ fraction of all $k$-SAT functions on $n$ variables are unate.
\end{conj}

Bollob\'as, Brightwell, and Leader also proposed a weaker version of this conjecture, namely that the number of $k$-SAT functions on $n$ boolean variables is $2^{(1+o(1))\binom{n}{k}}$, and established this weaker conjecture for $k=2$.
\cref{conj:BBL} was proved for $k=2$ by Allen \cite{All07} and for $k=3$ by Ilinca and Kahn~\cite{IK12}. The proofs in \cite{BBL03,All07} for $k=2$ used graph regularity (\cite{IK09} gave an alternate regularity-free proof for $k=2$), whereas the proof for $k=3$ \cite{IK12}  used hypergraph regularity. Bollob\'as and Brightwell~\cite{BB03} further conjectured that even if $k = k(n)$ is allowed to increase with $n$, as long as $k \le (1/2-c) n$ for some constant $c > 0$, the number of $k$-SAT functions on $n$-variables is $2^{(1+o(1))\binom{n}{k}}$. In that paper \cite{BB03}, they proved bounds on the number of $k$-SAT functions on $n$ variables in the regime $k \ge n/2$, where a completely different asymptotic behavior arises.

The goal of our paper here is two-fold. First, we reduce \cref{conj:BBL} for each $k$ to a specific extremal problem about partially directed hypergraphs, analogous to classical hypergraph Tur\'an density problems. 
Hypergraph Tur\'an density problems have been intensely studied, although only solved in a relatively small number of cases (see survey by Keevash~\cite{Kee11}).
Our reduction is essentially lossless, as we show that the $k$-SAT enumeration problem (for fixed $k$) is equivalent to the corresponding Tur\'an density problem. 
Whereas previous approaches to enumerating $k\SAT$ functions looked at specific forbidden structures (e.g., odd-blue-triangle-free graphs~\cite{BBL03,All07,IK09}), our approach is more systematic and identifies all relevant obstructions.
Our reduction uses the hypergraph container method, as opposed to the graph and hypergraph regularity methods used in earlier works.

The second goal of our paper is to prove \cref{conj:BBL} for $k=4$ by solving the corresponding hypergraph Tur\'an density problem, \cref{conj:pi} below.
We also give an easier proof of the corresponding $k=3$ problem compared to the methods in \cite{IK12}.
For every fixed $k$, in principle one might be able to confirm the conjecture (after our reduction) via a finite computation, but the size of the computation grows extremely quickly (the approach in \cite{IK12} for $k=3$ essentially amounts to checking 5-vertex cases by hand).
To solve the $k=4$ case of the problem, we apply a recent result of F\"uredi and Melaki~\cite{FM17} on the minimum number of triangular edges in a graph with a given number of vertices and edges.
In an appendix, we further confirm the $k=5$ case via a brute force computer search.
To prove \cref{conj:BBL} for additional values of $k$ (or ideally for all $k$), it remains to solve a Tur\'an density type problem, \cref{conj:pi} below.

\subsection*{Acknowledgments} Zhao first learned of this problem as a graduate student from Jeff Kahn and would like to thank him for the encouragement to work on this problem.

\subsection{An extremal open problem} \label{sec:extremal-open}
We state a tantalizing conjecture that would imply \cref{conj:BBL}. 

A \emph{partially directed graph} (also known as a \textit{mixed graph}) is formed by taking a graph and orienting a subset of its edges (to orient an edge means to choose one of two directions for the edge), so every edge is either directed or undirected. An example is illustrated below.

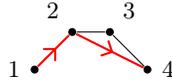
\begin{center}
    \begin{tikzpicture}[
    scale=.5,
    font=\footnotesize,
    v/.style = {circle, fill, inner sep = 0pt, minimum size = 3pt},
    e/.style = {thick, red, postaction={decorate,decoration={
        markings,
        mark=at position .5*\pgfdecoratedpathlength+1.5pt 
        with {\arrow{angle 90}}
      }}}]
    \node[v,label=left:1] (1) at (0,0) {};
    \node[v,label=above left:2] (2) at (1,1) {};
    \node[v,label=above right:3] (3) at (2,1) {};
    \node[v,label=right:4] (4) at (3,0) {};
    \draw (2)--(3)--(4);
    \draw[e] (2)--(4);
    \draw[e] (1)--(2);
    \end{tikzpicture}
\end{center}

Given a pair of partially directed graphs $\vec H$ and $\vec G$, we say that $\vec H$ is a \emph{subgraph} of $\vec G$ if one can obtain $\vec H$ from $\vec G$ by a combination of (1) removing vertices, (2) removing edges, and (3) removing the orientation of some edges. 

The following partially directed graph plays a special role:
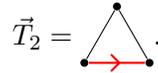
\[
\vec T_2 = 
\begin{tikzpicture}[
    scale=.5,
    baseline={([yshift=-.8ex]current bounding box.center)},
    font=\footnotesize,
    v/.style = {circle, fill, inner sep = 0pt, minimum size = 3pt},
    e/.style = {thick, red, postaction={decorate,decoration={
        markings,
        mark=at position .5*\pgfdecoratedpathlength+1.5pt 
        with {\arrow{angle 90}}
      }}}]
    \node[v] (1) at (90:1) {};
    \node[v] (2) at (210:1) {};
    \node[v] (3) at (-30:1) {};
    \draw (3)--(1)--(2);
    \draw[e] (2)--(3);
    \end{tikzpicture}.
\]
Below, the left graph contains $\vec T_2$ as a subgraph, and the right does not contain $\vec T_2$ as a subgraph.

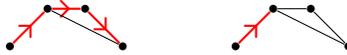
\begin{center}
    \begin{tikzpicture}[
    scale=.5,
    font=\footnotesize,
    v/.style = {circle, fill, inner sep = 0pt, minimum size = 3pt},
    e/.style = {thick, red, postaction={decorate,decoration={
        markings,
        mark=at position .5*\pgfdecoratedpathlength+1.5pt 
        with {\arrow{angle 90}}
      }}}]

    \begin{scope}
        \node[v] (1) at (0,0) {};
        \node[v] (2) at (1,1) {};
        \node[v] (3) at (2,1) {};
        \node[v] (4) at (3,0) {};
        \draw[e] (2)--(3);
        \draw[e] (3)--(4);
        \draw (2)--(4);
        \draw[e] (1)--(2);
    \end{scope}
    
    \begin{scope}[shift={(6,0)}]
        \node[v] (1) at (0,0) {};
        \node[v] (2) at (1,1) {};
        \node[v] (3) at (2,1) {};
        \node[v] (4) at (3,0) {};
        \draw (2)--(3)--(4);
        \draw (2)--(4);
        \draw[e] (1)--(2);
    \end{scope}

    \end{tikzpicture}
\end{center}

The following statement implies \cref{conj:BBL} for $k=2$ (see \cref{thm:strong-count} for the full statement of the implication).

\begin{theorem}
For all sufficiently large $n$, 
every $n$-vertex partially directed graph with $\alpha \binom{n}{2}$ undirected edges and $\beta \binom{n}{2}$ directed edges and not containing $\vec T_2$ as a subgraph satisfies
\[
\alpha + (\log_2 3)\beta \le 1.
\]
\end{theorem}
 
In fact, we know the optimal constant in front of $\beta$: it is $2-o(1)$ (see \cref{prop:cf2}).
By orientating all the edges of $K_{\floor{n/2}, \ceil{n/2}}$ from one part to the other part, we obtain a partially directed graph with $\floor{n^2/4}$ directed edges and no $\vec T_2$. This construction shows that one cannot do better than $\beta = 2-o(1)$.

Now let us generalize the problem to hypergraphs.

A \emph{partially directed 3-graph} (\emph{3-PDG}) is formed by taking a 3-graph and orienting a some subset of edges. Here to \emph{orient} an edge means to pick some vertex in the edge (we say that this is a \emph{directed edge} that is \emph{directed towards} or \emph{pointed at} the chosen vertex). We notate a directed edge by putting a $\vee$ on top of the pointed vertex, e.g., $12\wc3$. 

This following 3-PDG plays the role of $\vec T_2$ from earlier:
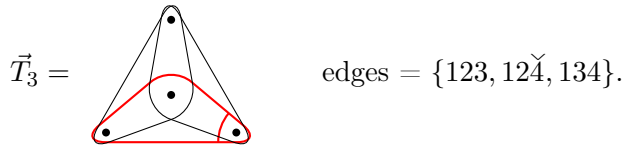
\[
\vec T_3
=
    \begin{tikzpicture}[
    font=\footnotesize,
    baseline={([yshift=-.8ex]current bounding box.center)},
    v/.style = {circle, fill, inner sep = 0pt, minimum size = 3pt},
    e/.style = {thick, red, postaction={decorate,decoration={
        markings,
        mark=at position .5*\pgfdecoratedpathlength+1.5pt 
        with {\arrow{angle 90}}
      }}}]
	\node[v] (0) at (0:0) {};
    \node[v] (1) at (90:1) {};
    \node[v] (2) at (210:1) {};
    \node[v] (3) at (-30:1) {};
	
	\draw[red, thick, rounded corners=10pt] (90:.4) -- +(-40:1.6) coordinate (a) -- +(-140:1.6)--cycle;
	\draw[red,  thick] ($ (a) + (-180:.6) $) arc (180:140:.6);

	\draw[rounded corners=10pt,rotate=120] (90:.4) -- +(-40:1.6) coordinate (a) -- +(-140:1.6)--cycle;
	\draw[rounded corners=10pt,rotate=-120] (90:.4) -- +(-40:1.6) coordinate (a) -- +(-140:1.6)--cycle;
    \end{tikzpicture}
    \qquad 
    \text{edges = } 
    \{ 123, 12\wc 4, 134 \}.
\]

We define \emph{subgraph} the same as earlier, i.e., obtainable by deleting vertices and edges as well as removing the orientation of some edges. 

We prove the following result, which in turn (via \cref{thm:strong-count}) implies \cref{conj:BBL} for $k=3$.

\begin{theorem}
For all sufficiently large $n$, 
every $n$-vertex 3-PDG with $\alpha \binom{n}{3}$ undirected edges and $\beta \binom{n}{3}$ directed edges and not containing $\vec T_3$ as a subgraph satisfies
\[
\alpha + (\log_2 3)\beta \le 1.
\]
\end{theorem}

We can extend the above definitions to \emph{partially directed $k$-graphs} (\emph{$k$-PDG}).
For each $k \ge 3$, we define $\vec T_k$, the \textit{partially directed $k$-triangle}, as the $k$-PDG obtained by starting with $\vec T_2$ and then adding $k-2$ common vertices to all three edges, e.g.,
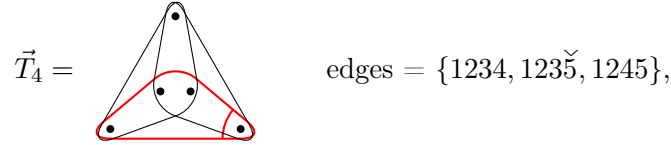
\[
\vec T_4
=
    \begin{tikzpicture}[
    font=\footnotesize,
    baseline={([yshift=-.8ex]current bounding box.center)},
    v/.style = {circle, fill, inner sep = 0pt, minimum size = 3pt},
    e/.style = {thick, red, postaction={decorate,decoration={
        markings,
        mark=at position .5*\pgfdecoratedpathlength+1.5pt 
        with {\arrow{angle 90}}
      }}}]
	\node[v] (0) at (0:.2) {};
    \node[v] (0') at (180:.2) {};
    
    \node[v] (1) at (90:1) {};
    \node[v] (2) at (210:1) {};
    \node[v] (3) at (-30:1) {};
	
	\draw[red, thick, rounded corners=10pt] (90:.4) -- +(-40:1.6) coordinate (a) -- +(-140:1.6)--cycle;
	\draw[red,  thick] ($ (a) + (-180:.6) $) arc (180:140:.6);

	\draw[rounded corners=10pt,rotate=120] (90:.4) -- +(-40:1.6) coordinate (a) -- +(-140:1.6)--cycle;
	\draw[rounded corners=10pt,rotate=-120] (90:.4) -- +(-40:1.6) coordinate (a) -- +(-140:1.6)--cycle;
    \end{tikzpicture}
    \qquad 
    \text{edges = } 
    \{ 1234, 123\wc 5, 1245 \},
\]
and
\[
\vec T_5
=
    
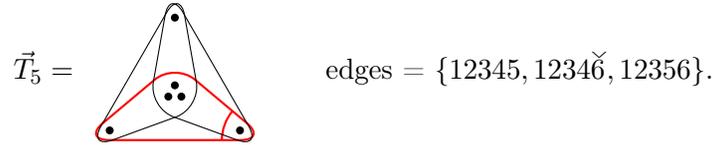
\begin{tikzpicture}[
    font=\footnotesize,
    baseline={([yshift=-.8ex]current bounding box.center)},
    v/.style = {circle, fill, inner sep = 0pt, minimum size = 3pt},
    e/.style = {thick, red, postaction={decorate,decoration={
        markings,
        mark=at position .5*\pgfdecoratedpathlength+1.5pt 
        with {\arrow{angle 90}}
      }}}]
	\node[v] at (90:.1) {};
    \node[v] at (210:.1) {};
    \node[v] at (-30:.1) {};
    
    \node[v] (1) at (90:1) {};
    \node[v] (2) at (210:1) {};
    \node[v] (3) at (-30:1) {};
	
	\draw[red, thick, rounded corners=10pt] (90:.4) -- +(-40:1.6) coordinate (a) -- +(-140:1.6)--cycle;
	\draw[red,  thick] ($ (a) + (-180:.6) $) arc (180:140:.6);

	\draw[rounded corners=10pt,rotate=120] (90:.4) -- +(-40:1.6) coordinate (a) -- +(-140:1.6)--cycle;
	\draw[rounded corners=10pt,rotate=-120] (90:.4) -- +(-40:1.6) coordinate (a) -- +(-140:1.6)--cycle;
    \end{tikzpicture}
    \qquad 
    \text{edges = } 
    \{ 12345, 1234\wc 6, 12356 \}.
\]

For each fixed $k$, the following conjecture implies \cref{conj:BBL} (via \cref{thm:strong-count}).

\begin{conj} \label{conj:T_k}
Fix $k \ge 2$. 
For all sufficiently large $n$, 
every $n$-vertex $k$-PDG with $\alpha \binom{n}{k}$ undirected edges and $\beta \binom{n}{k}$ directed edges and not containing $\vec T_k$ as a subgraph satisfies
\[
\alpha + (\log_2 3)\beta \le 1.
\]
\end{conj}

We also show that by enlarging the set of forbidden subgraphs from $\vec T_k$ to some special finite set $\cF_k$, the extremal claim $\alpha + (\log_2 3)\beta \le 1$ becomes essentially equivalent to \cref{conj:BBL}. This set $\cF_k$ is constructed in \cref{d:forbidden} below.

For each $k$, if the statement is true for some $n$, then it is true for all larger $n$ by a standard averaging argument.
So in principal, the conjecture could be verified via a finite computation.
However, this computation grows extremely quickly with $k$.
\cref{s:turan} contains our proof of the conjecture for $k \le 4$ via extremal graph theoretic techniques. 
Appendix~\ref{a:comp} contains a brute force solution to $k=5$.
The brute force method is unlikely to extend further, and so the extremal graph theoretic techniques may be valuable for future work.

\subsection{Strategy}
Now let us describe our strategy for enumerating $k$-SAT functions.
Our setup is closest to the work of Ilinca and Kahn~\cite{IK12} where they enumerated $3$-SAT functions.

\begin{defn}[Minimal formula]
A formula $G = \{ C_1, C_2, \dots , C_l\}$ on variables $\{x_1, \ldots x_n\}$ is \emph{minimal} if deleting any clause from $G$ changes the resulting boolean function. 
Equivalently, $G$ is minimal if for each clause $C_i \in G$, there is an assignment $\mathsf w \in \{0, 1\}^n$ that satisfies $C_i$ but no other $C_j \in G$, $j \neq i$. We call such $\mathsf w$ (not necessarily unique) a \emph{witness} to $C_i$.\footnote{In \cite{IK12}, the term ``non-redundant'' is used for what we are calling ``minimal.''}
\end{defn}

\begin{example}
The $2$-SAT formula $\{wx,wy,x\ol{z},\ol{y}z\}$ is not minimal since it is impossible to satisfy only $wx$ and no other clause. 
Indeed, if we attempt to only satisfy $wx$, we must assign $w=1$ and $x=1$; 
to avoid satisfying $wy$ and $x\ol{z}$, we must assign $y=0$ and $z=1$, which would then lead to the final clause $\ol{y}z$ being satisfied.
\end{example}

Every $k$-SAT function can be expressed as a minimal formula, but possibly in more than one way. We upper bound the number of minimal $k$-SAT formulae, which in turn upper bounds the number of $k$-SAT functions.

To bound the number of minimal $k$-SAT formulae, we identify a fixed set $\cB$ of non-minimal formulae, and then find an upper bound to the number of (not necessarily minimal) formulae not containing anything from $\cB$ as a subformulae.
We will show that, by setting $\cB$ to be all non-minimal formulae on a constant number of variables, the number of $\cB$-free formulae is asymptotically the same as the number of minimal formulae. 

\begin{defn}[Subformula]
Given a formula $G$, a \emph{subformula} of $G$ is a subset of clauses of $G$. 
We say that two formulae are \emph{isomorphic} if one can be obtained from the other by relabeling variables. 
Given another formula $F$, a \emph{copy} of $F$ in $G$ is a subformula of $G$ that is isomorphic to $F$.
Given a set $\cB$ of $k$-SAT formulae, we say that $G$ is \emph{$\cB$-free} if $G$ has no copy of $B$ for every $B \in \cB$.
\end{defn}

The problem of counting $\cB$-free formulae is analogous to counting $F$-free graphs on $n$ vertices.
A classic result by Erd\H{o}s, Kleitman, and Rothschild~\cite{EKR76} shows that almost all triangle-free graphs are bipartite. 
Erd\H{o}s, Frankl, and R\"odl~\cite{EFR86} generalized this result and showed that for a fixed graph $F$, the number of $n$-vertex $F$-free graph is $2^{\operatorname{ex}(n, F) + o(n^2)}$. 
This result was initially proved using the graph regularity method.
It can also be proved as quick application of the more recently developed hypergraph container method~\cite{BMS15,ST15} (see survey \cite{BMS18}).
Earlier works on $k$-SAT enumeration are closer in spirit to the regularity approach to counting $F$-free graphs. 
Here we instead use the container method, which is simpler to carry out.

Here is a quick sketch of how to use the container method to enumerate $n$-vertex triangle-free graphs. 
By an application of the hypergraph container theorem, there is a collection $\cG$ of $n$-vertex graphs (``containers'')
each with $(1/4 + o(1))n^2$ edges, 
such that $\abs{\cG} = 2^{o(n^2)}$
and every $n$-vertex triangle-free graph is a subgraph of some container $G \in \cG$.
It then follows that the number of $n$-vertex triangle-free graphs is at most $\sum_{G \in \cG} 2^{e(G)} = 2^{(1/4 + o(1)) n^2}$.

As in the case of counting $F$-free graphs, the enumeration of $\cB$-free formulae reduces to a certain extremal problem. 
Here the situation diverges from the graph theory setting.

\begin{defn}[Simple formula] \label{def:simple}
A $k$-SAT formula is \emph{simple} if no two clauses use exactly the same set of $k$ variables.
\end{defn}

\begin{example}
The 3-SAT formula $\{wxy, wx\ol{z},xyz\}$ is simple, whereas $\{xyz,xy\ol{z}\}$ is not.
\end{example}

What happens when we apply a similar approach to enumerate $\cB$-free formulae?
Every minimal $k$-SAT formula on $n$ variables can be made simple by removing $o(n^k)$ clauses (\cref{prop:nearly-simple}). 
So let us focus on enumerating simple $\cB$-free formulae.
The container theorem gives us a collection $\cG$ of $n$-variable ``container'' formulae with $\abs{\cG} = 2^{o(n^k)}$
such that every $n$-variable $\cB$-free formula is a subformula of some $G \in \cG$. 
Furthermore, every $G \in \cG$ is ``not too large'' in a certain sense.
For a container $G \in \cG$, suppose $\alpha \binom{n}{k}$ $k$-subsets of variables support exactly one clause, and $\beta \binom{n}{k}$ $k$-subsets of variables support exactly two clauses.
It will always be the case that $o(n^k)$ $k$-subsets of variables support more than two clauses.
Then the number of simple subformulae of $G$ is $2^{\alpha\binom{n}{k}} 3^{\beta\binom{n}{k}} 2^{o(n^k)}$, and we would like to guarantee that $\alpha + (\log_2 3) \beta \le 1 +o(1)$. 

Even if we wish to only enumerate simple formulae, we need to solve an extremal problem for non-simple formulae. 
This is not an inadequacy of our approach, but actually an essential feature of the problem (see a related discussion at the end of \cite[Section 2]{BBL03}).
This is where our situation differs from enumerating $F$-free graphs, which reduces to an extremal problem on $F$-free graphs instead of some more complicated object.

The extremal problem for formulae is rather unwieldly.
In earlier approaches, the idea is to isolate some specific patterns and analyze them by hand.
Here we take a different approach which is both simpler and more systematic.
We show that the extremal problem for formulae actually reduces to a Tur\'an density type extremal problem on partially directed hypergraphs.
We explain what these objects are shortly in the next subsection. 
We will define a concrete finite set $\cF_k$ of small partially directed $k$-graphs that turn out to completely capture the asymptotic enumeration of $k$-SAT functions.
This is a self-contained Tur\'an density problem and can be studied independently of the $k$-SAT enumeration problem.
A solution to this Tur\'an density problem would yield a resolution to \cref{conj:BBL}.
If the Tur\'an density conjecture were false, then there would be way more $k$-SAT functions than in \cref{conj:BBL}.
Unfortunately, we were not able to resolve the Tur\'an density conjecture in full, although we solve it for $k=4$, which is new.
Readers who are interested in working on the Tur\'an problem only need to read the first two sections of this paper, as the rest of the paper concerns the reduction from the $k$-SAT enumeration problem to this Tur\'an problem.

The container approach as described can only deduce an upper bound of the form $2^{(1+o(1))\binom{n}{k}}$. 
Obtaining the tighter asymptotic $(1+o(1))2^{n + \binom{n}{k}}$ of \cref{conj:BBL} requires additional stability arguments.
By analogy, the container method sketched earlier shows that there are $2^{(1/4+o(1))n^2}$ triangle-free graphs on $n$ vertices.
To obtain the Erd\H{o}s--Kleitman--Rothschild result in its full strength, namely that $1-o(1)$ fraction of all triangle-free graphs are bipartite, one needs to apply additional stability arguments on top of the container method, as done in \cite{BBCLMS17} (and later strengthened in \cite{BS20}).

Recall that  \cref{conj:BBL} is equivalent to saying that a $1-o(1)$ fraction of all $k$-SAT functions are unate.
Far-from-unate formulae can be effectively handled by the container argument.
The final piece to show is that that almost all nearly unate formulae are unate. 
We do this by extending the argument from \cite[Section 8]{IK12}, which only dealt with the $3$-SAT case; 
our $k$-SAT argument has to overcome further technical challenges.

\subsection{Partially directed hypergraphs}

Now we provide some precise definitions.
While there are many possible notions of a directed hypergraph, the relevant notion for us is the one where a directed edge is an edge with along with ``pointed'' vertex on the edge.
A partially directed hypergraph is formed from a hypergraph by directing some of its edges and leaving others intact.

\begin{defn}[$k$-PDG]
A \emph{partially directed $k$-graph}, abbreviated \emph{$k$-PDG}, is given by a set $V$ of $n$ vertices, and, for each unordered $k$-element subset $\{v_1, \dots, v_k\}$, one of the following possibilities:
\begin{enumerate}
    \item[(a)] No edge with vertices $\{v_1, \dots, v_k\}$, or
    \item[(b)] An \emph{undirected edge} with vertices  $\{v_1, \dots, v_k\}$, or
    \item[(c)] A \emph{directed edge} using vertices  $\{v_1, \dots, v_k\}$ along with a choice of some $v_i \in \{v_1, \dots, v_k\}$;  we say that the edge is \emph{directed} (or \emph{pointed}) towards $v_i$. We notate such a directed edge by $v_1 \cdots \wc v_i \cdots v_k$.
\end{enumerate}
In particular, no two edges (whether directed or undirected) of a $k$-PDG can use the exact same set of $k$ vertices.
\end{defn}

\begin{example}
This is the edge set of a $3$-PDG: $\{123,12\wc4, 134,235,\wc245\}$.

Non-examples of edge sets of $3$-PDGs include $\{123,12\wc 3\}$ (two edges using the same triple of vertices), $\{1\wc2\wc3\}$ (not a valid edge), and $\{123,12\}$ (not 3-uniform).
\end{example}

Given a $k$-PDG $\vec{H}$, let $e_u(\vec H)$ be the number of undirected edges in $\vec H$ and $e_d(\vec H)$ be the number of directed edges in $\vec H$. Denote their edge densities by
\[
\alpha(\vec{H}) = \frac{e_u(\vec H)}{\binom{n}{k}}
\quad \text{and} \quad 
\beta(\vec{H}) = \frac{e_d(\vec H)}{\binom{n}{k}}.
\]

\begin{defn}[Subgraph of $k$-PDG] \label{def:subgraph}
We say that a $k$-PDG $\vec{F}$ is a \emph{subgraph} of another $k$-PDG $\vec{H}$ if we can obtain a $k$-PDG isomorphic to $\vec{F}$ starting from $\vec{H}$ by some sequence of actions of the following types: (a) deleting an edge, (b) deleting a vertex, and (c) forgetting the direction of an edge.
\end{defn}

\begin{defn}[$\cF$-free $k$-PDG]
Given a set $\cF$ of $k$-PDGs, we say that a $k$-PDG $\vec{H}$ is \emph{$\cF$-free} if $\vec{H}$ does not contain any element of $\cF$ as a subgraph.
\end{defn}

We usually denote $k$-PDGs by simply writing their edge sets.

\begin{example}
$\{123,12\wc4\}$ is a subgraph of $\{12\wc3,12\wc 4\}$ by (c) in \cref{def:subgraph}. 
However, $\{1\wc23,12\wc4\}$ is not a subgraph of $\{12\wc3,12\wc 4\}$.
Another example is that by removing the second edge from $\{123,1\wc24,12\wc5\}$, we obtain a subgraph isomorphic to $\{123,12\wc4\}$.
\end{example}

\begin{defn}[Tur\'an density for $k$-PDGs]\label{d:densitypdg}
Given a set $\cF$ of $k$-PDGs, let $\pi(\cF, \theta)$ be the smallest real number such that every $n$-vertex $\cF$-free $k$-PDG $\vec{H}$ satisfies
\[
\alpha(\vec{H}) +\theta\cdot \beta(\vec{H})  \le \pi(\cF, \theta) + o_{n\to\infty}(1).
\]
\end{defn}

We will consider the following finite family $\mcF_k$ of forbidden structures, that will completely certify \textit{minimality} of the $k$-SAT formula associated to a particular $k$-PDG. 

As it turns out, our proof of \cref{conj:BBL} for $k=2,3,4$, and $5$, the Tur\'an density problem only needs to forbid $\vec T_k$, which recall is the generalized triangle with one directed edge introduced in the previous subsection; $T_k$ is also the first listed element (a) of $\cF_k$ in \cref{ex:F2,ex:F3,ex:F4}.

\begin{defn}\label{d:forbidden}
Let $k\ge 2$ be an integer. 
Let the family $\mcF_k$ be the set of $k$-PDGs that can be built from the following procedure, and furthermore are non-redundant in the sense of not containing as a subgraph a different $k$-PDG arising from the same procedure.
\begin{enumerate}
    \item \emph{Initialization.} Start with a single edge $e_0$ that may be directed or undirected.
    If $e_0$ is undirected, then let $S$ be its set of $k$ vertices.
    If $e_0$ is directed, then let $S$ be the set of $k-1$ undirected vertices of $e_0$. Note that $S$ is a set of unpointed vertices and it will maintain this property later.
    \item \emph{Extensions.} Add $j$ directed edges $e_1, \dots, e_j$, with $0 \le j \le k$, where each new edge $e_i$ is directed towards some new vertex $w_i$ and its $k-1$ remaining vertices are chosen from $S$. Here $w_1, \dots, w_j$ are all distinct and outside $e_0$. Let $T$ be the set of all $k+j$ vertices seen so far.
    \item \emph{Closure.} Either 
    \begin{enumerate}
        \item[(a)] add a new undirected edge contained in $T$ (note that this can only happen if we had $j > 0$ in the previous step, as every $k$ vertices are allowed to support at most one edge), or
        \item[(b)] add an edge directed towards some new vertex $w_0 \notin T$ with its remaining $k-1$ vertices all contained in $T$ and at least one of these $k-1$ vertices outside $S$. 
    \end{enumerate}
    Call this final edge $e_{j+1}$.
    \end{enumerate}
\end{defn}

\begin{rem}
It is perhaps unclear why the above forbidden structures are of interest. To give a brief glimpse, consider associating $\vec T_4 \in \mcF_4$ with the following $4$-SAT formula: $$\phi_{\vec T_4} = (x_1 \land x_2 \land x_3 \land x_4) \lor (x_1 \land x_2 \land x_3 \land \overline{x_5}) \lor (x_1 \land x_2 \land x_4 \land x_5).$$
Observe that $\phi_{\vec T_4}$ is non-minimal, as satisfying the first clause by assigning $x_1 = x_2 = x_3 = x_4 = 1$ necessitates that one of the latter two clauses in $\phi_{\vec T_4}$ is satisfied (which one depends on our assignment of $x_5$). Further, any $4$-SAT formula that contains a pattern like $\phi_{\vec T_4}$ is also non-minimal. Thus, one can think of $\vec T_4$ as a a representation of a ``short certificate for non-minimality'' (noting that non-minimal formula may still lack a $\vec T_k$-type pattern).
\end{rem}

\begin{rem}
Let us emphasize again that some of the subgraphs generated by this procedure may include another as a subgraph, in which case we do not include the supergraph in $\cF_k$. For example, both $F = \{123,124,13\wc4\}$ and $F' = \{12\wc 3,124,13\wc4\}$ can be generated by the procedure, but $F$ is a subgraph of $F'$ and so we do not include $F'$ in $\cF_3$;  on the other hand, $F \in \cF_3$ since no subgraph of $F$ other than itself can be generated by the procedure.
\end{rem}

\begin{rem}[Maximum size of an element of $\cF_k$] \label{rem:Fk-size}
Every $k$-PDG in $\mcF_k$ has at most $2k$ vertices and at most $k+2$ edges.
Indeed, by the end of the extension step, we can have up to $2k$ vertices and $k+1$ edges. 
If the closure edge adds a new vertex, then one of the edges added in the extension much must have been redundant (see previous remark).
\end{rem}

\begin{example} \label{ex:F2}
Here are all the elements of $\cF_2$:
\begin{enumerate}[(a)]
    \item $12,1\wc3,23$;
    \item $1\wc 2, 2\wc 3$.
\end{enumerate}
\end{example}

\begin{example}\label{ex:F3}
Here are all the elements of $\cF_3$. For each $3$-PDG listed, the edges are printed in the order of their generation in \cref{d:forbidden} (the initialization edge appears first and the closure edge appears last).
\begin{enumerate}[(a)]
    \item $123, 12\wc 4, 134$
    \item $12\wc3,13\wc4$
    \item $123, 12\wc 4, 12\wc 5, 345$
    \item $123,12\wc4,13\wc5,145$
    \item $123,12\wc4,13\wc5,245$
    \item $123,12\wc4,12\wc5,12\wc6,456$
    \item $123,12\wc4,12\wc5,13\wc6,456$
    \item $123,12\wc4,13\wc5,23\wc6,456$
    \item $123,12\wc4,34\wc5$
    \item $123,12\wc4,13\wc5,45\wc6$
\end{enumerate}
\end{example}

\begin{example}\label{ex:F4}
Here are two of the elements of $\cF_4$, although there are many more not listed.
\begin{enumerate}[(a)]
    \item $1234,123\wc 5, 1245$
    \item $123\wc 4, 124\wc 5$
\end{enumerate}
\end{example}

\subsection{Results} 
We will show how the enumeration of $k$-SAT functions is determined by the Tur\'an density problem $\pi(\cF_k, \log_2 3)$.
Note that for every $k$ and $\theta$ we always have
\[
\pi(\cF_k, \theta) \ge 1
\]
since the complete undirected $k$-PDG is $\cF_k$-free. (This is related to the fact that every monotone formula is minimal.) 
Also, note the trivial inequality $\pi(\cF_k, \theta) \le \pi(\cF_k,\theta')$ if $\theta < \theta'$.

\begin{theorem} \label{thm:weak-count}
    Fix $k \ge 2$. 
    The number of $k$-SAT functions on $n$ variables is 
    \[ 2^{(\pi(\cF_k, \log_2 3)+o(1))\binom{n}{k}}.
    \]
\end{theorem}

In particular, for each $k \ge 2$, if $\pi(\cF_k,\log_23) = 1$ (conjectured to be true, and we can prove it for $k \le 4$), then the number of $k$-SAT functions is $2^{(1+o(1))\binom{n}{k}}$, which corresponds to a weaker version of \cref{conj:BBL}. To get the full \cref{conj:BBL}, we just need a tiny bit more on the Tur\'an density problem.

\begin{theorem} \label{thm:strong-count}
    Fix $k \ge 2$. If $\pi(\cF_k, \theta) = 1$ for some $\theta > \log_2 3$, then the number of $k$-SAT functions on $n$ variables is $(1+o(1))2^{n+\binom{n}{k}}$ (i.e., \cref{conj:BBL} is true for this $k$).
\end{theorem}

Thus, to prove \cref{conj:BBL}, we simply need to show the following.

\begin{conj} \label{conj:pi}
For every $k \ge 2$, there is some constant $\theta > \log_2 3$ such that $\pi(\cF_k, \theta) = 1$.
\end{conj}

We prove this conjecture for $k = 2, 3,4$, thereby yielding the following result.

\begin{theorem} \label{eq:234}
For $k=2,3,4$, $\pi(\cF_k, \theta) = 1$ for some $\theta > \log_2 3$.
Consequently, the number of $k$-SAT functions on $n$ variables is $(1+o(1)) 2^{n+\binom{n}{k}}$.
\end{theorem}

\begin{rem}
If one could prove $\pi(\cF_k, \log_2 3) = 1$ then almost certainly the same proof verbatim would also yield $\pi(\cF_k, \theta) =1$ for some $\theta > \log_23$.
Indeed, the number $\log_2 3$ plays no special role for the Tur\'an density problem and is only meaningful because of its applications for $k$-SAT enumeration. 
So our results show that the $k$-SAT enumeration problem is essentially equivalent to determining $\pi(\cF_k, \log_23)$.
\end{rem}

\begin{rem}\label{r:finite-check}
For any fixed $k$, \cref{conj:pi}, if true, is verifiable via a finite computation, namely by checking over all $n_k$-vertex $\cF_k$-free $k$-PDGs for some sufficiently large but finite $n_k$ (depending on $\theta$).
In \cite{IK12}, an extremal problem of similar nature for $k=3$ was checked ``by hand'' over $5$-vertex graphs.
We do not take the brute-force computational approach for Theorem~\ref{eq:234}.
Instead, we solve the $k=4$ problem using extremal graph theoretic arguments, with the help of a recent result of F\"uredi and Maleki~\cite{FM17}. 
We also give a much simpler argument for $k=3$, which does not involve any case checking.

In Appendix~\ref{a:comp}, we include a description of a C++ implementation of an algorithm that does a brute force check to find that for all $7$-vertex $4$-PDGs, $\pi(\vec T_k, 7/4) = 1$. This, in concert with Lemma~\ref{lem:vertex-averaging}, verifies Conjecture~\ref{conj:pi} for $k=5$.
\end{rem}

For larger values of $k$, previously \cite[Theorem 5.3]{BBL03} showed that the number of $k$-SAT functions on $n$ variables is at most $2^{\sqrt{(k+1)\pi}\binom{n}{k}}$ for all $n \ge 2k$.
We obtain the following improved bound, shy of proving of \cref{conj:pi}.

\begin{theorem}\label{thm:generalub}
For all fixed $k \ge 5$ we have
\[
\pi(\cF_k, \log_23) \le \frac{\log_23}{2.727/k + 1}.
\]
Consequently, the number of $k$-SAT functions on $n$ variables is
\[
2^{(\pi(\cF_k, \log_23)+o(1))\binom{n}{k}} \le 2^{\paren{\frac{\log_23}{2.727/k + 1} + o(1)}\binom{n}{k}}.
\]
\end{theorem}

\subsection{Outline}

The rest of the paper can be divided into three parts.

\smallskip

\emph{Part I. A Tur\'an density problem.}

In~\cref{s:turan}, we study the Tur\'an density problem and prove \cref{conj:pi} for $k=2,3,4$. The readers who are interested in working on \cref{conj:BBL} only need to solve the remaining Tur\'an density problem and thus do not need to read beyond this section.

\smallskip

\emph{Part II. Exponential asymptotics.}

The goal of this part is to prove \cref{thm:weak-count}.
In~\cref{s:hyptools}, we recall some tools for hypergraphs, including the hypergraph container theorem. 
In \cref{s:firstcount}, we prove an upper bound on the number of $k$-SAT functions. 
In \cref{sec:lb}, we prove a matching lower bound.

\smallskip

\emph{Part III. A stability argument.} 

In the last part of this article, we prove \cref{thm:strong-count}.
In~\cref{s:stability}, by a more careful analysis of the containers, we reduce the problem to showing there are a negligibly many non-unate but close-to-unate minimal formulae.
This final claim is then established in \cref{s:istar} by extending the arguments in \cite[Section 8]{IK12}.

\section{A Tur\'an density problem}\label{s:turan}

The general hypergraph Tur\'an problem seeks to answer the following extremal question: how many edges can an $n$-vertex $k$-uniform hypergraph have without containing a copy of forbidden $k$-uniform hypergraph (``$k$-graph'') $F$? 
Many methods for tackling such problems are described in~\cite{Kee11}.
We consider a variant of the classical hypergraph Tur\'an problem and work with $k$-PDGs as defined in the introduction.
The main objective here is to show \cref{conj:pi}, that for each $k\ge 2$, $\pi(\cF_k,\theta) = 1$ for some $\theta > \log_23$; we prove this conjecture for $k=2,3,4$ here (and also $k = 5$ in the appanedix), and leave all larger values of $k$ as an open problem.
Towards the conjecture, it suffices to prove 
$\pi(\vec T_k,\theta) = 1$ for some $\theta > \log_23$
where $\vec T_k \in \cF_k$ is the special $k$-PDG defined in \cref{sec:extremal-open} (since $\cF_k$-free implies $\vec T_k$-free).

\subsection{Results}

We prove the following recursive bound relating the Tur\'an problem associated to $k$-SAT with that of $(k-1)$-SAT.

\begin{lemma} \label{lem:vertex-averaging} For $k \ge 3$, we have that
\[
\pi\paren{\vec T_k, \frac{(k-1)\theta + 1}{k}} \le \pi( \vec T_{k-1}, \theta).
\]
\end{lemma}

\begin{rem}\label{rem:vertex-averaging-tk}
The same proof also shows the inequality
\[
\pi\paren{\cF_k, \frac{(k-1)\theta + 1}{k}} \le \pi(\cF_{k-1}, \theta).
\]
\end{rem}

Recall that since the complete undirected $k$-PDG on $n$ vertices is $\vec T_k$-free, $\pi(\vec T_k, \theta) \ge 1$ for all $\theta > 0$. 
The following result solves the Tur\'an problem for enumerating $2$-SAT and $3$-SAT functions ($5/3 > \log_23=1.58496\dots$). 
We will see its short proof momentarily. 
The second claim follows from the first by \cref{rem:vertex-averaging-tk}.

\begin{prop}\label{prop:cf2}
$\pi(\vec T_2, 2) = 1$.  Therefore, $\pi(\vec T_3, 5/3) = 1$.
\end{prop}

The result for $4$-SAT follows from the following improved bound on $\pi(\vec T_3, \theta)$.
The second claim again follows from the first by \cref{lem:vertex-averaging}.

\begin{prop}\label{prop:cf3}
$\pi(\vec T_3, \phi) = 1$ with $\phi = 1.909 > (4\log_23-1)/3$. Consequently, $\pi(\vec T_4, \theta) = 1$ for $\theta = (3\phi+1)/4 > \log_23$.
\end{prop}

\begin{rem}
To prove~\cref{conj:pi} for $k=5$, 
by applying~\cref{lem:vertex-averaging} twice, 
it suffices to show $\pi(\vec T_3, \phi) = 1$ for some $\phi > 1.975 >(5\log_23-2)/3$.
\end{rem}

\cref{prop:cf2} and \cref{prop:cf3} together lead to \cref{eq:234}. To obtain an upper bound on $\pi(\vec T_k, \phi)$ for larger $k$, 
iterating~\cref{lem:vertex-averaging} gives
\[
\pi\paren{\vec T_k, \frac{\ell}{k}(\phi - 1) + 1} \le \pi(\vec T_\ell, \phi)
\qquad \text{for } k > \ell \ge 2.
\]
Using $\ell = 3$ and $\phi = 1.909$, we get
$\pi(\vec T_k, 3(\phi-1)/k + 1) = 1$ for all $k \ge 4$. 
Since $3(\phi-1)/k + 1 < \log_23$ for all $k \ge 5$, we have
\[
\pi(\vec T_k, \log_23) \le \frac{\log_23}{3(\phi - 1)/k + 1}\cdot  \pi\paren{\vec T_k, \frac{3}{k}(\phi - 1){k} + 1}
\le \frac{\log_23}{2.727/k + 1} 
\qquad \text{for all } k \ge 5.
\]
This gives~\cref{thm:generalub}.
Our work leads to the following natural Tur\'an problem for $k$-PDGs.

\begin{problem}
For each $k \ge 2$, what is the largest $\theta$ such that $\pi(\cF_k, \theta) = 1$?
\end{problem}

The following $k$-PDG is $\cF_k$-free: 
partitioning the $n$ vertices as $A \cup B$ with $\abs{A} = n/k$ and $\abs{B} = (k-1)n/k$, and add all directed edges with $k-1$ vertices in $B$ and directed towards a vertex in $A$.
This $k$-PDG has $\abs{A}\binom{\abs{B}}{k-1} = (n/k) \binom{(1-1/k)n}{k-1} = (1+o(1))(1-1/k)^{k-1}\binom{n}{k}$ edges. So $\pi(\cF_k, \theta) > 1$ whenever $\theta >  (1-1/k)^{-k+1}$.
We conjecture that this construction is optimal in the following sense.

\begin{conj} \label{conj:pi-opt}
$\pi(\vec T_k, (1-1/k)^{-k+1}) = 1$ for every $k \ge 2$.
\end{conj}

\cref{prop:cf2} confirms this conjecture for $k=2$. 
The conjectured optimal construction is reminiscent of the solution to the hypergraph Tur\'an problem corresponding to 3-graphs with independent  neighborhoods by F\"uredi, Pikhurko, and Simonovits~\cite{FPS05}. 
Our techniques are partly inspired by their work.



The following conjecture is a restatement of \cref{conj:T_k}, which would imply \cref{conj:pi} and hence \cref{conj:BBL}.
In this section, we prove it for $k \le 4$, and in the appendix we prove it for $k \le 5$.

\begin{conj}
For every $k \ge 2$, $\pi(\vec T_k, \theta) = 1$ for some $\theta > \log_23$.
\end{conj}

\subsection{An averaging argument} 
We first prove~\cref{lem:vertex-averaging} via an averaging argument. 

\begin{defn}
Given a $k$-PDG $\vec{H} = (V, E)$ and a vertex $v \in V$, the \emph{link of $v$}, denoted $\vec{H}_v$, is the $(k-1)$-PDG with vertex set $V \setminus \{v\}$ and with edges obtained by taking all edges of $\vec H$ containing $v$ and then deleting the vertex $v$ from each edge. We preserve all the direction data, except that every edge directed towards $v$ in $\vec H$ becomes an undirected edge in $\vec H_v$.
\end{defn}

\begin{lemma}
If $\vec{H}$ is $\vec T_k$-free, then $\vec{H}_v$ is $\vec T_{k-1}$-free for all $v \in V(\vec{H})$.
\end{lemma}

\begin{proof} 
If, for some vertex $v$, $\vec H_v$ contains $\vec T_{k-1}$ as a subgraph, then adding $v$ to all the edges shows that $\vec H$ contains $\vec T_k$ as a subgraph.
\end{proof}


In the proofs below, we drop dependencies on $\vec H$ from notation when there is no confusion:
\[
\alpha = \alpha(\vec H), \quad 
\alpha_v = \alpha(\vec H_v), \quad 
\beta = \beta(\vec H), \quad 
\beta_v = \beta(\vec H_v).
\]

\begin{proof}[Proof of~\cref{lem:vertex-averaging}] 
Let $\vec H$ be a $\vec T_k$-free $k$-PDG.
By linearity of expectation over a uniform random vertex $v \in V$, $\EE_v \alpha_v = \alpha + \beta/k$ and 
$\EE_v \beta_v = (k-1)\beta/k$.
For every $v \in V$, $\vec{H}_v$ is $\vec T_{k-1}$-free, so $\alpha_v + \theta\cdot \beta_v \le \pi(\vec T_{k-1}, \theta)+o(1)$. Thus, we have that
\[
\alpha + \frac{\theta(k-1) + 1}{k}\cdot \beta
\le 
\EE_v[\alpha_v + \theta\cdot \beta_v]
\le \pi(\vec T_{k-1}, \theta)+o(1).
\]
This implies $\pi\left(\vec T_k, \frac{(k-1)\theta + 1}{k}\right) \le \pi(\vec T_{k-1}, \theta)$.
\end{proof}

\subsection{Tur\'an density for 2-PDGs}
We leverage the following lemma of F\"uredi to show \cref{prop:cf2}, that partially directed graphs $\vec{H}$ satisfy $\pi(\vec T_2, 2) = 1$.

\begin{lemma}[{\cite[Lemma 2.1]{Fur92}}] \label{lem:furedi}
Given a graph $G = (V, E)$, let $G^2$ be the graph on vertex set $V$ with $xy \in E(G^2)$ if and only if there is some $z \in V$ with $xz, yz \in E$. Then, for any graph $G$,
$$e(G^2) \ge e(G) - \lfloor n/2 \rfloor.$$
\end{lemma}

In a graph $G$, we say that edge $xy \in E(G)$ is a \textit{triangular edge} if there exists some $z \in V(G)$ so that $xz,yz \in E(G)$.
We then have the following corollary, showing that $\pi(\vec T_2, 2) = 1$, which completes the proof of \cref{prop:cf2}.

\begin{cor}\label{cor:dtr}
If a 2-PDG $\vec{H}$ does not contain $\vec T_2$ as a subgraph, then $\alpha + 2\beta \le 1 + o(1)$.
\end{cor}

\begin{proof} 
    Let $G$ be the underlying simple graph of $\vec{H}$ (the graph obtained from $\vec H$ by replacing directed edges with simple edges).
    Then $e(G) = (\alpha + \beta){n \choose 2}$. 
    Only non-triangular edges in $G$ can arise as directed edges in $\vec{H}$. We further note that all triangular edges in $G$ are in $G^2$. Thus by~\cref{lem:furedi},
    $${n\choose 2} \ge e(G^2) + \beta{n \choose 2} \ge e(G) + \beta{n \choose 2} - O(n)
    = (\alpha + 2\beta){n \choose 2} - O(n),$$
    which implies that $\alpha + 2 \beta \le 1 + o(1)$.  
\end{proof}

\subsection{An improvement for 3-PDGs}

Here we prove \cref{prop:cf3} that $\pi(\vec T_3, 1.909) = 1$.
We begin with a more careful decomposition of the links, followed by  an averaging-inductive argument.

\begin{defn}
For a $3$-PDG $\vec H = (V,E)$ and a vertex $v\in V$, we define four associated link (di)graphs. Let $A_v$ be a graph with edges arising from undirected edges in $\vec H$ containing $v$, i.e., 
\[
A_v = \left(V \backslash \{v\},\{w_1w_2:vw_1w_2\in E\}\right).
\]
Let $a_v := e(A_v)/{n-1\choose 2}$. 
Let $B_v$ be a graph arising from edges in $\vec H$ containing but not directed towards $v$,  i.e., 
\[
B_v = \left(V\backslash \{v\}, \{w_1w_2:vw_1\wc w_2\in E\}\right).
\]
Let $b_v:= e(B_v)/{n-1\choose 2}$. 
Let $C_v$ be a graph arising from directed edges in $\vec H$ pointed at $v$, i.e.,
\[
C_v=\left(V\backslash \{v\},\{w_1w_2:\wc v w_1w_2\in E\}\right).
\]
Let $c_v:= e(C_v)/{n-1\choose 2}$.
So $B_v$ is comprised of the directed edges of the link $\vec H_v$, whereas $A_v \cup C_v$ is comprised of the subgraph consisting of undirected edges of $\vec H_v$.

Finally, we write $H_v$ for the undirected graph formed by forgetting the directions in $\vec H_v$, i.e., the union of $A_v$, $B_v$, and $C_v$.
\end{defn}

\begin{prop}\label{prop:linkgraphturan}
Let $\phi = 1.909$.
For every $\epsilon > 0$, there is some $n_0(\epsilon)$ such that for every $\vec T_3$-free $3$-PDG $\vec H$ on $n \ge n_0(\epsilon)$ vertices, every vertex $v$ satisfies
\[
a_v+\frac{3\phi-1}{2}\cdot b_v+c_v\leq 1 + \epsilon \qquad \text{or} \qquad
a_v+\phi \cdot (b_v + c_v) \le 1 + \epsilon.
\]
\end{prop}

\begin{proof}[Proof of~\cref{prop:cf3} using \cref{prop:linkgraphturan}]
It suffices to show that, for every fixed $\epsilon > 0$, there is some $C = C(\epsilon)$ so that every $n$-vertex $\vec T_3$-free $3$-PDG $\vec H$ satisfies
\[
(\alpha + \phi\cdot \beta)\binom{n}{k}
= e_u + \phi \cdot e_d \le (1+\epsilon)\binom{n}{k} + C
\]
By choosing a large constant $C$, we may assume that the inequality holds for all $n < n_0(\epsilon)$.
So assume $n \ge n_0(\epsilon)$.
By \cref{prop:linkgraphturan}, every vertex $v$ of $H$ satisfies 
\[
a_v+\frac{3\phi-1}{2}\cdot b_v+c_v\leq 1 + \epsilon \qquad \text{or} \qquad
a_v+\phi \cdot (b_v + c_v) \le 1 + \epsilon.
\]
By linearity of expectation, we have
\[
\EE_v a_v = \alpha, \quad 
\EE_v b_v = \frac{2\beta}3 , \quad \text{and}\quad 
\EE_v c_v = \frac{\beta}3.
\]
Thus, if the first inequality in the statement is true for all $v$, then averaging over all $v$ yields $\alpha + \phi \cdot \beta \le 1+\epsilon$, as desired.

Now suppose that there is some $v$ with $a_v + \phi \cdot (b_v + c_v)\le 1+\epsilon$.
The $(n-1)$-vertex $3$-PDG $\vec H - v$ is also $\vec T_3$-free, so by the induction hypothesis,
\[
e_u(\vec H-v) + \phi \cdot e_d(\vec H-v) \le (1+\epsilon)\binom{n-1}{k} + C.
\]
By separately considering edges in $\vec H$ containing $v$ and not containing $v$, we obtain
\[
e_u(\vec H) = a_v(\vec H) \binom{n}{k-1} + e_u(\vec H-v)
\]
and
\[
e_d(\vec H) = (b_v(\vec H) + c_v(\vec H)) \binom{n}{k-1} + e_d(\vec H-v).
\]
Consequently,
\begin{align*}
e_u(\vec H) + \phi \cdot  e_d(\vec H) 
&\le (a_v(\vec H)+\phi \cdot (b_v(\vec H) + c_v(\vec H)))\binom{n}{k-1} + (1+\epsilon)\binom{n-1}{k} + C
\\ &\le (1+\epsilon)\binom{n}{k-1} + (1+\epsilon)\binom{n-1}{k} + C
\\
& \le (1+\epsilon)\binom{n}{k} + C.
\end{align*}
This completes the induction.
\end{proof}

Thus, it remains to show~\cref{prop:linkgraphturan}. 

Since $\vec T_3=\{vxy, vx\wc z, vyz\}$ is forbidden, we see that every edge of $B_v$ is non-triangular in the graph $H_v$.
This leads to the following natural question. What is the maximum possible density of non-triangular edges in a graph of given edge density (over 1/2)?
Fortunately, this problem was recently solved by F\"uredi and Maleki~\cite{FM17}.
An optimal construction comes from partitioning the vertex set into $A\cup B\cup C$ and adding all edges within $A$ as well as edges between $A\cup C$ and $B$. 
In this construction, the non-triangular edges are those between $B$ and $C$.
One should optimize over possible sizes of the partition $A \cup B \cup C$ to maximize the number of non-triangular edges while maintaining the required total number of edges (that this number is optimal is proved asymptotically in \cite{FM17} and exactly in \cite{GL18} for sufficiently large graphs).
The precise statement that we will use is given below, where the the function $f(\rho)$ below is best possible due to the construction just described.

\begin{theorem}[F\"uredi and Maleki \cite{FM17}] \label{thm:FM} 
	Fix $0 \le \rho \le 1$. In an $n$-vertex graph with at least $\rho \binom{n}{2}$ edges, the number of non-triangular edges is at most
	\[
	(f(\rho) + o(1)) \binom{n}{2},
	\]
	where
	\begin{equation} \label{eq:t-delta}
	f(\rho) = \max \set{ 2yz : x+y+z=1, x^2 + 2xy + 2yz \ge \rho, \ x,y,z \in \RR_{\ge 0}}.
	\end{equation}
\end{theorem}

In contrast, the easier \cref{lem:furedi} amounts to the inequality $f(\rho) \le 1 - \rho$, whereas $f(\rho)$ is strictly convex for $1/2 \le \rho \le 1$ with $f(1/2) = 1/2$ and $f(1) = 0$.
Here the strict convexity of $f$ is crucial for our proof.

\begin{proof}[Proof of \cref{prop:linkgraphturan}]
Suppose that for some vertex $v$ of $\vec H$, one has both
\[
a_v+\frac{3\phi-1}{2}\cdot b_v+c_v > 1+\epsilon  \qquad \text{and} \qquad
a_v+\phi \cdot (b_v + c_v) > 1+\epsilon.
\]
Since every edge of $B_v$ is non-triangular in $H_v$, by \cref{thm:FM},
\[
b_v \le f(a_v + b_v + c_v) + o(1).
\]
It follows, provided $n$ is sufficiently large, there exist real numbers $x,y,z,a,b,c$ satisfying the following system (to get to this system, we set $a=a_v$, $b= \min\{b_v,f(a_v+b_v+c_v)\}$, and $c=c_v$, so that $b > b_v - o(1)$, which implies the first two inequalities as $1+\epsilon-o(1)$ is strictly larger than $1$ if $n$ is sufficiently large):
\begin{align*}
	a + \frac{3\phi-1}{2}\cdot b +c &> 1, \\
	a + \phi \cdot (b + c) &> 1, \\
	x + y + z &= 1, \\
	x^2 + 2xy + 2yz &\ge a+b+c, \\
	2yz & \ge b, \\
	a,b,c,x,y,z &\ge 0.
\end{align*}
The following \textsc{Mathematica} code outputs to \texttt{False} when executed (under a second on a modern computer). 
It proves that the above system has no real solutions for $\phi = 1.909$, thereby yielding the desired contradiction.
This computer verification is rigorous and uses Collins' cylindrical algebraic decomposition \cite{Col75} (a form of quantifier elimination). See \cite{Kau10} for a tutorial.

\begin{verbatim}
    phi = 1909/1000;
    CylindricalDecomposition[
      a + b (3 phi - 1)/2 + c > 1 &&
      a + phi (b + c) > 1 &&  
      x + y + z == 1 &&
      x^2 + 2 x y + 2 y z >= a + b + c &&
      2 y z >= b &&
      a >= 0 && b >= 0 && c >= 0 && x >= 0 && y >= 0 && z >= 0, 
      {x, y, z, a, b, c}]
\end{verbatim}
\end{proof}

\section{Hypergraph Tools}\label{s:hyptools}
We recall several standard tools on hypergraphs that will prove useful in our analysis of minimal $k$-SAT formulae and $k$-PDGs.

\subsection{Densities of blowups} \label{sec:blowup}

Given a $k$-graph $H$, we write $H[b]$ for the \emph{$b$-blowup} of $H$, defined to be the $k$-graph obtained by replacing every $v \in V(H)$ with $b$ duplicates, where $k$ vertices in $V(H[b])$ form an edge if and only if the vertices they originate from form an edge in $H$.

By a standard Cauchy--Schwarz argument, we know that if some $n$-vertex $k$-graph contains $\delta n^{\abs{V(H)}}$ homomorphic copies of $H$, then it contains at least $\delta n^{2\abs{V(H)}}$ homomorphic copies of $H[2]$.
For sufficiently large $n$, all but $O_H(n^{2\abs{V(H)}-1})$ such homomorphic copies of $H[2]$ are actual subgraphs. So we have the following (here ``copies'' refers to isomorphic copies as subgraphs).

\begin{lemma}\label{lem:blowupkst} 
For every $\epsilon > 0$ and $k$-graph $H$, there is some $\delta > 0$ and $n_0$ such that for all $n > n_0$, if an $n$-vertex $k$-graph has at least $\epsilon n^{\abs{V(H)}}$ copies of $H$ then it has at least $\delta n^{2\abs{V(H)}}$ copies of $H[2]$.
\end{lemma}

Given a $k$-SAT formula $G$, its \emph{$b$-blowup}, denoted $G[b]$, is obtained from $G$ by replacing each variable by $b$ identical duplicates. For example, for the $3$-SAT formula $G = \{x_1x_2x_3,x_1\ol x_3x_4\}$, we have
\begin{multline*}    
G[2]=\left\{x_1x_2x_3, x_1x_2x_3',x_1x_2'x_3,x_1x_2'x_3',
x_1'x_2x_3, x_1'x_2x_3',x_1'x_2'x_3,x_1'x_2'x_3',\right. \\
\left.x_1\ol {x_3}x_4,x_1\ol {x_3}x_4', x_1\ol {x_3'}x_4,x_1\ol {x_3'}x_4',
x_1'\ol {x_3}x_4,x_1'\ol {x_3}x_4', x_1'\ol {x_3'}x_4,x_1'\ol {x_3'}x_4'\right\}
\end{multline*}

We have a similar conclusion for blowups of formulae.
\begin{lemma}\label{lem:formula-blowup}
For every $\epsilon > 0$ and $k$-SAT formula $G$, there is some $\delta > 0$ and $n_0$ such that for all $n > n_0$, if a $k$-SAT formula on $n$ variables has at least $\epsilon n^{v(G)}$ copies of $G$, then it has at least $\delta n^{2v(G)}$ copies of $G[2]$.
\end{lemma}

\subsection{Kruskal--Katona theorem}

We need the following special case of the Kruskal--Katona theorem~\cite{Kru63,Kat68}.
Here a \emph{simplex} in a $k$-graph is a clique on $k+1$ vertices.

\begin{theorem} \label{thm:kruskalkatona}
A $k$-uniform hypergraph on $n$-vertices with at most $\beta n^k/k!$ edges contains at most $\beta^{\frac{k+1}{k}} n^{k+1}/(k+1)!$ simplices.
\end{theorem}

Here is a quick proof of this statement using Shearer's entropy inequality~\cite{CGFS86} (see \cite{Kah01}).

\begin{proof}
Let $(X_1, \dots, X_{k+1})$ be the vertices of a uniformly chosen simplex in the $k$-graph, with the $k+1$ vertices permutated uniformly at random.
By Shearer's inequality, letting $\vec X_{-i}$ denote $(X_1, \dots, X_{k+1})$ with the $i$-th coordinate removed, we have
\begin{align*}
k \log((k+1)!\text{\#simplices}) 
&= k H(X_1, \dots, X_{k+1}) 
\\
&\le H(\vec X_{-1}) + \cdots +  H(\vec X_{-(k+1)}) 
\\
&= (k+1) H(X_1, \dots, X_{k})
\le (k+1) \log (k!\text{\#edges}). 
\end{align*}
Since $k!\text{\#edges} \le \beta n^k$, we have $(k+1)!\text{\#simplices} \le \beta^{\frac{k+1}{k}}n^{k+1}$.
\end{proof}

\subsection{Hypergraph containers}\label{ss:containerlemma}

For any formula $B$, we let $v(B)$ denote the number of variables that appear in some clause in $B$. For a finite set $\cB$ of $k$-SAT formulae, we let
\[
    m(\cB)=\max_{B\in\cB}|B|.
\]

We use the hypergraph container theorem, proved independently by 
Balogh, Morris, and Samotij \cite{BMS15} 
and 
Saxton and Thomason \cite{ST15}
to show the following.

\begin{theorem}\label{thm:gencontub}
Let $\cB$ be a finite set of simple $k$-SAT formulae.
For every $\delta > 0$, there exists $C = C(\cB, \delta)$ such that for every $n$, there exists a collection $\cG=\cG(\cB,\delta)$ of formulae on variable set $X = \{x_1, \dots, x_n\}$ such that:
\begin{enumerate}
    \item[(a)] Every $\cB$-free formula with variables in $X$ is a subformula of some $G \in \cG$, and
    \item[(b)] For every $G \in \cG$ and $B \in \cB$, $G$ has at most $\delta n^{v(B)}$ copies of $B$, and
    \item[(c)] $\abs{\cG} \le n^{Cn^{k-1/(m(\cB)-1)}}$.
\end{enumerate}
\end{theorem}

In the remainder of this section, we deduce the above claim from the more general hypergraph container theorem. To state the general result, we introduce some notation for an  $\ell$-uniform hypergraph $\cH$. Let $\mathcal I(\cH)$ denote the collection of independent vertex sets in $\cH$. For every $j\in\{1,\dots,\ell\}$, let $\Delta_j(\cH)$ denote the maximum $j$-codegree of vertices in $\cH$, i.e., the maximum number of edges containing a given $j$-vertex subset in $\cH$. Let $\cP(V)$ denote the family of all subsets of $V$, and let ${V\choose \leq s}$ denote the family of those subsets of $V$ having size at most $s$.

For an increasing family (i.e., closed under taking supersets)  of vertex sets $\cF\subset \cP(V(\cH))$, we say that $\cH$ is $(\cF,\epsilon)$-dense if $e(\cH[A])\geq\epsilon e(\cH)$ for every $A\in\cF$. A simple example of $\cF\subset \cP(V(\cH))$ for which $\cH$ is $(\cF,\epsilon)$-dense is the family
$$\cF_\epsilon=\{A\subset V(\cH):e(\cH[A])\geq\epsilon e(\cH)\}.$$

We apply the following version of the
hypergraph container theorem as stated in \cite[Theorem 2.2]{BMS15},
A similar result was proved in \cite{ST15}.

\begin{theorem}[Hypergraph container theorem]\label{t:hc}
For every $\ell\in\NN$ and $c,\epsilon>0$, there exists $C_0=C_0(\ell,c,\epsilon)>0$ such that the following holds. Let $\cH$ be an $\ell$-uniform hypergraph and $\cF\subset \cP(V(\cH))$ be an increasing family of sets such that $|A|\geq \epsilon v(\cH)$ for all $A\in\cF$. Suppose that $\cH$ is $(\cF,\epsilon)$-dense and there exists $p\in(0,1)$ such that, for every $j\in\{1,\dots,\ell\}$, we have
$$\Delta_j(\cH)\leq c\cdot p^{j-1}\cdot\frac{e(\cH)}{v(\cH)}.$$ 
Then there exists a family $\mathcal S\subset {V(\cH)\choose \leq C_0\cdot p\cdot v(\cH)}$ and functions $f:\mathcal S\rightarrow\ol{\cF}$, $g:\cI(\cH)\rightarrow\cS$ such that for every $I\in \mathcal I(\mcH)$, we have:
\begin{enumerate}
    \item[(i)] $g(I)\subset I$;
    \item[(ii)] $I\setminus g(I)\subset f(g(I))$.
\end{enumerate}
\end{theorem}

To prove~\cref{thm:gencontub} using~\cref{t:hc}, we
proceed in two steps. First, we enlarge $\cB$ to a set $\cB'$ of formulae, each with $m(\cB)$ clauses. This allows us to construct an $m(\cB)$-uniform hypergraph $\cH_k$ on the set of all clauses on $X=\{x_1,\dots,x_n\}$, and apply the container argument. We then verify the codegree bounds of~\cref{t:hc}, which relies on all the elements of $\cB$ being simple.

We use~\cref{t:hc} to show that for every $\cB$ and $\delta>0$, we can take family $\cG$ of at most $n^{Cn^{k-1/(m(\cB)-1)}}$ formulae on $X$, such that: (1) every $n$-vertex $\cB$-free formula on $X$ is a subformula of some $G\in\cG$; (2) for any $B\in\cB$, every $G\in\cG$ has at most $\delta n^{v(B)}$ copies of $B$.

\begin{proof}[Proof of~\cref{thm:gencontub} using~\cref{t:hc}]
Fix a finite set $\cB$ of simple $k$-SAT formulae. We set $m= m(\cB)$, $v= km$ and $\cB'$ to be the set of all isomorphism classes of formulae $B'$ having the following properties:
\begin{itemize}
\item $B'$ has exactly $m$ clauses;
\item $B'$ has a variable set of size $v$; here we allow isolated variables (i.e., variables not used by any clause in $B'$);
\item $B'$ is simple;
\item there exists some $B \in \cB$ such that $B'$ has a copy of $B$.
\end{itemize}
Notice that $\cB'$ is uniquely determined by $\cB$.

Let $\mcH_k = (V, E(\mcH_k))$ be an $m$-uniform hypergraph, where $V$ is the set of all $k$-SAT clauses on $X=\{x_1,\dots,x_n\}$ (so $|V| = 2^k{n\choose k}$), and
$$E(\mcH_k)=\{G\subset V\mid G\text{ is isomorphic to some }B'\in \cB'\}.$$

Every $k\SAT$ formula $G$ on $X$ can be viewed as a vertex subset of $\mcH_k$. For any $B \in \cB$, if $G$ has more than $\delta n^{v(B)}$ copies of $B$, then there exists some $\eta_0=\eta_0(B,\delta)>0$ and $B'\in \cB'$ such that $G$ has more than $\eta_0 n^{v}$ copies of $B'$, with the property that $B'$ has a copy of $B$. Since $\cB$ is finite, we can take
$$\eta=\min_{B\in\cB}\eta_0(B,\delta)>0.$$
This implies that there exists $\eta=\eta(\cB,\delta)>0$ such that if $G$ has more than $\delta n^{v(B)}$ copies of $B$ for any $B\in\cB$, then $|E(\mcH_k[G])|>\eta n^{v}$.

Since every formula $B'\in\cB'$ has variable set of size $v$, for a formula $G\subset V$, the number of formulae $G'\subset V$ such that (1) $G\subset G'$ and (2) there exists some $B'\in\cB'$ with $G'$ isomorphic to $B'$ is $O_\cB(n^{v - v(G)})$. 
In particular, taking $G$ to be the empty formula, we have $e(\cH_k)= O_\cB(n^v)$.

We therefore have the following upper bounds on the maximum $j$-codegrees of vertices in $\cH_k$. When $j=1$, since every single-clause formula $G=\{C\}$ has $v(G)=k$, we have $\Delta_1(\cH_k)=O_\cB(n^{v-k})$.
For all $j\in\{2,\dots, m\}$, since $\cB'$ is a set of simple formulae, and any simple formula $G=\{C_1,\dots,C_j\}$ with $j$ clauses has $v(G)\geq k+1$, we have $\Delta_j(\cH_k)= O_\cB(n^{v-k-1})$.

Let $p=n^{-1/(m-1)}$. Using a trivial lower bound $e(\cH_k)\geq {n\choose v}$ (say) and the fact that $|V|= 2^k{n\choose k}$, we can choose some $c=c(\cB)>0$ such that

$$\Delta_1(\mcH_k)\leq c\cdot\frac{{n\choose v}}{2^k{n\choose k}}\leq c\cdot \frac{e(\cH_k)}{|V|}$$
and
$$\Delta_j(\cH_k)\leq c\cdot \frac{1}{n}\cdot\frac{{n\choose v}}{2^k{n\choose k}}\leq  c\cdot p^{j-1}\cdot\frac{e(\cH_k)}{|V|}$$
for all $j\in\{2,\dots,m\}$.

For any $\delta>0$, we set $\eta=\eta(\cB,\delta)>0$ as above, $\epsilon=\epsilon(\cB,\delta)>0$ sufficiently small, and
$$\cF=\{A\subset V(\cH):|A|\geq \epsilon \cdot v(\cH_k)\text{ and }e(\cH_k[A])\geq \epsilon \cdot e(\cH_k)\},$$
so that $\cF$ is an increasing family of sets such that $|A|\geq \epsilon v(\cH_k)$ for all $A\in\cF$, and $\cH_k$ is $(\cF,\epsilon)$-dense.

Applying~\cref{t:hc}, we know that there exists $C_0=C_0(m,c,\epsilon)=C_0(\cB,\delta)$, a family
$$\mathcal S\subset {V\choose \leq C_0\cdot n^{-1/(m-1)}\cdot |V|}$$
and functions $f:\mathcal S\rightarrow\ol{\cF}$, $g:\cI(\cH)\rightarrow\cS$ such that for every $I\in \mathcal I(\mcH)$, we have:
\begin{enumerate}
    \item[(i)] $g(I)\subset I$;
    \item[(ii)] $I\setminus g(I)\subset f(g(I))$.
\end{enumerate}

Let $\cG=\cG(\cB,\delta)=\{S \cup f(S):S\in\cS\}$. We show that $\cG$ meets all the conditions decribed in~\cref{thm:gencontub}.

(a) \emph{Every $\cB$-free formula with $n$ variables is a subformula of some $G \in \cG$.} 

Since every $\cB$-free formula is clearly $\cB'$-free, it is an independent set $I\in \cI(\cH_k)$. Since $I\setminus g(I)\subset f(g(I))$ and $g(I)\in\cS$, we have $I\subset (g(I)\cup f(g(I))) \in\cG$.

(b) \emph{For every $G \in \cG$ and $B \in \cB$, $G$ has at most $\delta n^{v(B)}$ copies of $B$.} 

Every $G \in \cG$ is of the form $G=S \cup f(S)$ for some $S\in\cS$. Since $|S|\leq  C_0\cdot n^{-1/(m-1)}\cdot |V|$, the number of edges incident to $S\setminus f(S)$ is at most
$$
    C_0\cdot n^{-1/(m-1)}\cdot |V|\cdot \Delta_1(\cH_k)
    \leq C_0\cdot n^{-1/(m-1)}\cdot 2^k{n\choose k}\cdot  \Delta_1(\cH_k)
    =O_{\cB,\delta}\left(n^{v-1/(m-1)}\right).
$$
Meanwhile, since $f(S)\in\ol{\cF}$, we either have $e(\cH_k[f(S)])<\epsilon e(\cH_k)\leq \epsilon|\cB'|\cdot {n\choose v}\cdot (v!)$ or have $|f(S)|<\epsilon v(\cH)$, in which case $e(\cH_k[f(S)])<|\cB'|\cdot {\epsilon n\choose v}\cdot (v!)$.
Hence for every $G=S\cup f(S) \in \cG$, we have (recall that we have set $\epsilon=\epsilon(\cB,\delta)>0$ sufficiently small)
$$e(\cH_k[G])\leq \epsilon|\cB'|\cdot {n\choose v}\cdot (v!)+O_{\cB,\delta}\left(n^{v-1/(m-1)}\right)\leq \eta n^{v}.$$
Recall that if $G$ has more than $\delta n^{v(B)}$ copies of $B$ for any $B\in\cB$, then we will have $e(\mcH_k[G])>\eta n^{v}$. Hence for every $B \in \cB$, $G$ has at most $\delta n^{v(B)}$ copies of $B$.

(c) \emph{$\abs{\cG} \le n^{Cn^{k-1/(m-1)}}$ for some $C=C(\cB,\delta)$.} 

Since $|\cG|\leq |\cS|\leq \left|\binom{V}{\leq C_0\cdot n^{-1/(m-1)}\cdot |V|}\right|$ with $|V|=2^k{n\choose k}$, we have
\begin{align*}
    |\cG|&\leq \sum_{s=1}^{C_0\cdot n^{-1/(m-1)}\cdot |V|}{|V|\choose s}\leq C_0\cdot n^{-1/(m-1)}\cdot |V|\cdot {|V|\choose C_0\cdot n^{-1/(m-1)}\cdot |V|}\\
    &\leq n^{Cn^{k-1/(m-1)}}=n^{Cn^{k-1/(m-1)}}
\end{align*}
for some $C=C(C_0)=C(\cB,\delta)$.
\end{proof}

\section{Weak upper bound on the number of $k$-SAT functions using containers}\label{s:firstcount}

In this section, we prove the upper bound in \cref{thm:weak-count}.

\begin{theorem} \label{thm:formulae-upperbound}
For every fixed $k \ge 2$, the number of minimal $k$-SAT formulae on $n$ variables is at most $2^{(\pi(\cF_k, \log_2 3) + o(1)) \binom{n}{k}}$.
\end{theorem}

To prove this theorem, we first observe that every minimal formula is nearly simple. 
To  count simple minimal formulae, we apply the container method.
The container method requires us to characterize formulae with many instances of small non-minimal simple formulae, and these turn out to be related to the appearance of $\cF_k$ in an associated $k$-PDG.

\subsection{Minimal formulae are nearly simple}\label{ss:minimalsimple}

\begin{lemma}\label{lem:pairclausesnonmin}
The 2-blowup of a pair of clauses on the same set of $k$ variables is non-minimal.
\end{lemma}

For example, the lemma tells us that the 2-blowup of $\{xyz,xy\ol z\}$ is non-minimal.

\begin{proof}
Up to relabeling and/or negating some of the variables, we can write the pair of clauses as $v_1 v_2 \cdots v_k$ and $\overline{v_1} \cdots \ol {v_j}v_{j+1}\cdots v_k$ for some $1\leq j\leq k$. 
Then its 2-blowup has clauses (among others)
$$v_1 v_2 \cdots v_k,v_1' v_2 v_3 \cdots v_k, v_1 v_2' v_3 \cdots v_k, \ldots, v_1 \cdots v_{j-1}v_j'v_{j+1}\cdots v_k, \overline{v_1'} \cdots \ol {v_j'}v_{j+1}\cdots v_k.$$ 
It is impossible to satisfy the first clause only, since it would involve setting $v_1 = \cdots = v_k = 1$ and $v'_1 = \cdots = v'_j = 0$, which would then satisfy the last clause.
So the 2-blowup is non-minimal.
\end{proof}

\begin{rem}\label{rem:parallel}
In the above proof, if the initial pair of clauses differ by at least two negations, i.e., $j\ge 2$, then the selected clauses of the 2-blowup form a simple non-minimal subformula. 
We will use this fact later in \cref{ss:supersaturation}.
\end{rem}

\begin{prop}\label{prop:nearly-simple}
Every minimal $k$-SAT formula on $n$ variables can be made simple by deleting $o(n^k)$ clauses.
\end{prop}

\begin{proof}
Let $G$ be a minimal $k$-SAT formula on $n$ variables.
By \cref{lem:pairclausesnonmin}, $G$ does not contain any 2-blowup of a pair of clauses on the same $k$ variables. 
It follows from \cref{lem:formula-blowup} that the number of $k$-subsets of variables that support at least two clauses in $G$ is $o(n^k)$. Removing all such clauses yields a simple subformula of $G$.
\end{proof}

\subsection{From $k$-SAT to $k$-PDGs}

As hinted by \cref{rem:parallel}, for the container argument, we focus our attention on the following special type of formulae.

\begin{defn}[Semisimple formula] \label{def:semisimple}
A $k$-SAT formula is \emph{semisimple} if for every $k$-subset $S$ of variables, one of the following is true:
\begin{enumerate}
    \item[(a)] There is no clause on $S$, or
    \item[(b)] There is exactly one clause on $S$, or 
    \item[(c)] There are exactly two clauses on $S$ and they differ by exactly one variable negation.
\end{enumerate}
\end{defn}

\begin{example}
The 3-SAT formula $\{ x_1\overline{x_2}x_3, x_1\ol{x}_2\ol{x}_3,x_1x_2\overline{x_4}\}$ is semisimple, whereas the formula $\{ x_1\overline{x_2}x_3, x_1 x_2 \ol{x_3}\}$ is not semisimple.
\end{example}

We define a forgetful map from semisimple formulae to $k$-PDGs.
\begin{defn} 
Define
\[
\type\colon \{\text{semisimple $k$-SAT formula on $n$ variables} \} \to \{\text{$k$-PDGs on $n$ vertices}\}
\]
as follows. 
Given a semisimple formula $G$, let  $\type(G)$ be the $k$-PDG with vertices being the variables of $G$, where we add an undirected edge to every $k$-subset of variables supporting exactly one clause in $G$, and add a directed edge to every $k$-subset of variables support exactly two clauses in $G$, and for the latter, we direct the edge towards the variable where the two clauses differ by negation.
\end{defn}

\begin{example}
The semisimple formula
\[
abc,ab\ol{c}, \ol{a}b \ol{d},bd\ol{e},\ol{b}d\ol{e}
\]
has type
\[
ab\wc{c}, abd, \wc{b}de.
\]
\end{example}

\subsection{Forbidden subgraphs and non-minimal formulae}

The following proposition explains why $\cF_k$ plays a central role.

\begin{proposition}\label{prop:fktypenonmin}
Let $G$ be a $k$-SAT semisimple formula whose type lies in $\cF_k$.
Then $G[2]$ has a simple non-minimal subformula.
\end{proposition} 

\begin{example}
We consider a $3\SAT$ example.
An example of a semisimple formula with type $\vec{H}= (\{1,2,3,4\}, \{12\wc 3, 23\wc 4\})$ is the collection of $4$-clauses on $4$-variables $G =\{abc, ab\ol c, \ol b c d, \ol bc \ol d\} $. 
The $2$-blowup of $G$ has the following simple non-minimal formula: $B = \{abc, ab'c, ab\ol{c'}, \ol{b'}cd,\ol{b'}c'\ol{d}\}$. We can observe that $B$ is non-minimal by noting that to only satisfy $abc$, we must assign $a,b,c \mapsto 1$ which forces us to assign $b' \mapsto 0, c' \mapsto 1$ to avoid satisfying the second or third clause in $B$. However then we must assign $d \mapsto 0$ to avoid satisfying the fourth clause in $B$ and this combination of assignments satisfies the fifth clause in $B$.
\end{example}

\begin{proof}[Proof of \cref{prop:fktypenonmin}]
Suppose the initialization edge in $\type(G)$ is undirected. For clarity, we use two examples of $G$ and $\type(G)$ to illustrate this result that capture the fully generality of our argument.  We then summarize the general algorithm for identifying a simple non-minimal formula in the $2$-blowup of a general $G$.

We first look at the case where the final edge in $\type(G)$ does not introduce any new vertex. Let
\begin{align*}
    G&=\{abcd,a\bar bce,a\bar bc\bar e, \bar ac\bar df,\bar ac\bar d\bar f, \bar a  de\bar f\}\\
    &=\{C_0,C_1^{(1)},C_1^{(2)},C_2^{(1)},C_2^{(2)},C_3\},
\end{align*}
so that every label names the clause of the same order (i.e., $C_0=abcd$, \dots, $C_3=\bar a  de\bar f$). We have
$$\type(G)=(\{1,2,\dots,6\},\{1234,123\wc5,134\wc6,1456\}).$$
Using the notation in~\cref{d:forbidden}, $\type(G)$ is a 4-PDG in $\cF_4$ with:
\begin{itemize}
    \item $j=2$, $S=\{1,2,3,4\}$, $T=\{1,2,\dots,6\}$;
    \item Initialization: $e_0=1234$;
    \item Extension: $e_1=123\wc5$, $e_2=134\wc6$;
    \item Closure: $e_3=1456$.
\end{itemize}

We pick a simple non-minimal formula
$$G'=\{\widetilde C_0,\widetilde C_0^{(1)},\dots, \widetilde C_0^{(4)},\widetilde C_1,\widetilde C_2,\widetilde C_3\}\subset G[2]$$ 
where
\begin{align*}
    \widetilde C_0,\widetilde C_0^{(1)},\dots, \widetilde C_0^{(4)}&\in C_0[2],\\
    \widetilde C_1&\in C_1^{(1)}[2]\cup C_1^{(2)}[2],\\
    \widetilde C_2&\in C_2^{(1)}[2]\cup C_2^{(2)}[2],\\
    \widetilde C_3&\in C_3[2],
\end{align*}
using the following steps:
\begin{itemize}
    \item Set $\widetilde C_0=abcd\in C_0[2]$.
    \item Set $\widetilde C_0^{(1)}=a'bcd, \widetilde C_0^{(2)}=ab'cd, \widetilde C_0^{(3)}=abc'd, \widetilde C_0^{(4)}=abcd'\in C_0[2]$.
    \item Set $\widetilde C_3=\ol{a'}de\ol{f'}\in C_3[2]$. We pick the variables $a',d,e,f'$ because the clause $C_4=\ol a d e\ol f\in G$ uses positive literals $d,e$ and negative literals $\ol a,\ol f$. 
    \item Set $\widetilde C_1=a\ol{b'}c\ol e$. We pick the variables $a,b',c$ because $C_1^{(1)}, C_1^{(2)}\in G$ use positive literals $a,c$ and negative literal $\ol b$. We pick the variable $e$ because $\widetilde C_3$ uses the variable $e$, and then pick the literal $\ol e$ (instead of $e$) because $e$ appears as a positive literal in $\widetilde C_3$, and we want the literal here to be the opposite of it.
    \item Set $\widetilde C_2=\ol{a'}c\ol{d'}f'$. We pick the variables $a',c,d'$ because $C_2^{(1)}, C_2^{(2)}\in G$ use positive literal $c$ and negative literals $\ol a,\ol d$. We pick the variable $f'$ because $\widetilde C_3$ uses the variable $f'$, and then pick the literal $f'$ (instead of $\ol{f'}$) because $\ol{f'}$ appears as a negative literal in $\widetilde C_3$, and we want the literal here to be the opposite of it.
\end{itemize}

We show that it is impossible to uniquely satisfy $\widetilde C_0\in G'$. Suppose to the contrary that there exists a witness that maps $\widetilde C_0$ to 1 and all other clauses in $G'$ to 0. We clearly have $a,b,c,d\mapsto 1$. Since all of $C_0^{(1)},\dots,C_0^{(4)}$ are mapped to 0, we then have $a',b',c',d'\mapsto 0$. Now since $\widetilde C_1$ is mapped to 0, we have $e\mapsto 1$. Since $\widetilde C_2$ is mapped to 0, we have $f'\mapsto 0$. This forces $\widetilde C_{3}$ to get mapped to 1, a contradiction.

We now look at the case where the final edge in $\type(G)$ is pointed at some new vertex. Let
\begin{align*}
    G&=\{abcd,\bar ab\bar ce, \bar ab\bar c\bar e, a c\bar df, ac\bar d\bar f, b\bar e\bar fg,b\bar e\bar f\bar g\}\\
    &=\left\{C_0,C_1^{(1)},C_1^{(2)},C_2^{(1)},C_2^{(2)},C_3^{(1)},C_3^{(2)}\right\},
\end{align*}
so that every label names the clause of the same order. We have
$$\type(G)=(\{1,2,\dots,7\},\{1234,123\wc5,134\wc6,256\wc7\}).$$
Using the notation in~\cref{d:forbidden}, $\type(G)$ is a 4-PDG in $\cF_4$ with:
\begin{itemize}
    \item $j=2$, $S=\{1,2,3,4\}$, $T=\{1,2,\dots,6\}$;
    \item Initialization: $e_0=1234$;
    \item Extension: $e_1=123\wc5$, $e_2=134\wc6$;
    \item Closure: $e_3=256\wc7$.
\end{itemize}

We pick a simple non-minimal formula
$$G'=\{\widetilde C_0,\widetilde C_0^{(1)},\dots, \widetilde C_0^{(4)},\widetilde C_1,\widetilde C_1',\widetilde C_2,\widetilde C_3,\widetilde C_3'\}\subset G[2]$$ 
where
\begin{align*}
    \widetilde C_0,\widetilde C_0^{(1)},\dots, \widetilde C_0^{(4)}&\in C_0[2],\\
    \widetilde C_1,\widetilde C_1'&\in C_1^{(1)}[2]\cup C_1^{(2)}[2],\\
    \widetilde C_2&\in C_2^{(1)}[2]\cup C_2^{(2)}[2],\\
    \widetilde C_3,\widetilde C_3'&\in C_3^{(1)}[2]\cup C_3^{(2)}[2],
\end{align*}
using the following steps.
\begin{itemize}
    \item Set $\widetilde C_0=abcd\in C_0[2]$.
    \item Set $\widetilde C_0^{(1)}=a'bcd, \widetilde C_0^{(2)}=ab'cd, \widetilde C_0^{(3)}=abc'd, \widetilde C_0^{(4)}=abcd'\in C_0[2]$.
    \item Set $\widetilde C_3=b\bar{e'}\bar{f'}g$ and $\widetilde C_3'=b\bar{e}\bar{f'}\bar g\in C_3[2]$. For $\widetilde C_3$, we pick the variables $b,e',f'$ because the clauses $C_3^{(1)},C_3^{(2)}\in G$ uses positive literal $b$ and negative literals $\ol e,\ol f$. We then pick the positive literal $g$ (it does not matter whether we use $g$ or $g'$ here). We obtain $\widetilde C_3'$ from $\widetilde C_3$ by replacing the variable $e'$ with $e$ and replacing the literal $g$ with $\ol g$.
    \item Set $\widetilde C_1=\ol{a'}b\ol{c'}e'$ and $\widetilde C_1'=\ol{a'}b\ol{c'}e$. We pick $\widetilde C_1$ exactly as how we would pick $\widetilde C_1$ in the previous case. We obtain $\widetilde C_1'$ from $\widetilde C_1$ by replacing the variable $e$ with $e'$.
    \item Set $\widetilde C_2=ac\ol{d'} f'$. We pick $\widetilde C_2$ exactly as how we would pick $\widetilde C_2$ in the previous case.
\end{itemize}

We show that it is impossible to uniquely satisfy $\widetilde C_0\in G'$. Suppose to the contrary that there exists a witness that maps $\widetilde C_0$ to 1 and all other clauses in $G'$ to 0. By the same reasoning as before, we have $a,b,c,d\mapsto 1$ and $a',b',c',d'\mapsto 0$. Now since $\widetilde C_1,\widetilde C_1'$ are mapped to 0, we have $e,e'\mapsto 0$. Since $\widetilde C_2$ is mapped to 0, we have $f'\mapsto 0$. This forces at least one of $\widetilde C_{3},\widetilde C_{3}'$ to get mapped to 1, a contradiction.

We can generalize this strategy to any such $G$. Let $C_0\in G$ be the clause corresponding to the initialization edge $e_0$. Consider a witness $\mathsf w\in\{0,1\}^n$ that maps the original clause in $C_0[2]$ to 1 and all other clauses in $G[2]$ to 1. It follows that for every variable $v\in S$, we must assign one of $v,v'$ to 0 and the other to 1. Then for each of $e_1,\dots,e_{j+1}$, we can  pick one or two clauses appropriately from the 2-blowup of clauses in $G$ corresponding to it, and show that it is impossible to map all of these clauses to 0. The strategy when the initialization edge in $\type(G)$ is directed is very similar, and we do not elaborate here.
\end{proof}

If a semisimple formula $G$ has type $\vec{H}$ where $\vec{H}$ contains many copies of some $\vec{F} \in \cF_k$, then we'd like to show that $G$ contains many non-minimal, simple formulae. We make this intuition quantitative below.

\subsection{Supersaturation}\label{ss:supersaturation}

Now we show that, roughly speaking, if a formula has many clauses (with respect to some weighting), then it contains a small simple non-miminal subformula.

For a $k$-SAT formula $G$ on $n$ variables, define
\[
\alpha_i(G) {n \choose k} = \left| \left\{S \in {V \choose k} \mid  G \text{ has exactly } i \text{ clauses with variables } S \right\} \right|.
\]
Furthermore, define 
$\alpha_{2,1}(G){n \choose k}$ to be the number of $S \in {V \choose k}$ such that 
$G$ has exactly two clauses with variables $S$ and the pair differs on exactly on variable by negation.
Also set
\[
\alpha_{2,2}(G) = \alpha_2(G) - \alpha_{2,1}(G).
\]

\begin{lemma}~\label{lem:supersat-1}
Let $\epsilon > 0$ and $\theta>0$.
There exists $\delta = \delta(\epsilon, \theta) > 0$ such that for all sufficiently large $n$, if an $n$-variable $k$-SAT formula $G$ satisfies $\alpha_1(G)+\theta \alpha_{2,1}(G)\geq \pi(\cF_k,\theta)+\epsilon$, 
then there is some simple non-minimal subformula $B$ on at most $4k$ vertices, such that $G$ contains at least $\delta n^{v(B)}$ copies of $B$ as subformulae.
\end{lemma}

\begin{proof}Fix $\epsilon > 0$ and $\theta>0$. Consider an $n$-variable $k$-SAT formula $G$ such that $\alpha_1(G)+\theta \alpha_{2,1}(G)\geq \pi(\cF_k,\theta)+\epsilon$. Up to deleting all the clauses of $G$ not accounted for by $\alpha_1(G)$ and $\alpha_{2,1}(G)$, we can assume that $G$ is semisimple.

Let $\vec H = \type(G)$. Then $\vec H$ has $\alpha_1(G) \binom{n}{k}$ undirected edges and $\alpha_{2,1}(G)\binom{n}{k}$ directed edges.
It follows by the definition of Tur\'an density $\pi(\cF_k,\theta)$ and a standard supersaturation argument that there is some $\vec F \in \cF_k$ such that $\vec H$ contains at least $\Omega(n^{v(\vec F)})$ copies of $\vec F$.
There are only finitely many isomorphism classes of semisimple formulae with type $\vec F$, so by pigeonhole, there is some $k$-SAT semisimple formula $G'$, with $\type(G')=\vec F$, such that $G$ contains at least $\Omega(n^{v(\vec F)})$ copies of $G'$ as subformulae.
Then, by \cref{lem:formula-blowup}, $G$ contains $\Omega(n^{2v(\vec F)})$ copies of $G'[2]$ as subformulae.
By \cref{prop:fktypenonmin}, each copy of $G'[2]$ contains some simple non-minimal subformula. 
Note that $G'[2]$ has at most $4k$ variables since every element of $\cF_k$ has at most $2k$ variables. The conclusion then follows immediately after another application of pigeonhole.
\end{proof}

\begin{lemma}~\label{lem:supersat-2}
Let $\epsilon > 0$.
There exists $\delta = \delta(\epsilon)>0$ such that for all sufficiently large $n$, if an $n$-variable $k$-SAT formula has at least $\alpha_i(G) \ge \epsilon$ for some $i \ge 3$ or $\alpha_{2,2}(G) \ge \epsilon$, then there is some simple non-minimal subformula $B$ on at most $2k$ vertices such that $G$ contains at least $\delta n^{v(B)}$ copies of $B$ as subformulae.
\end{lemma}

\begin{proof}Fix $\epsilon > 0$. Consider an $n$-variable $k$-SAT formula $G$ such that $\alpha_i(G) \ge \epsilon$ for some $i \ge 3$, or $\alpha_{2,2}(G) \ge \epsilon$. By pigeonhole, there is some $2 \le j \le k$ such that $G$ has at least $\Omega(\epsilon n^k)$ subformulae of the form $G' = \{v_1 v_2 \cdots v_k$ and $\overline{v_1} \cdots \ol {v_j}v_{j+1}\cdots v_k\}$.
By \cref{lem:formula-blowup}, $G$ contains $\Omega(n^{2k})$ copies of $G'[2]$ as subformulae. Since every copy of $G'[2]$ has a non-minimal simple formula by \cref{rem:parallel}, the result follows.
\end{proof}

\subsection{Applying the container theorem}

Recall the definition of the $\alpha_i$'s from the beginning of \cref{ss:supersaturation}.

\begin{defn} 
Define the \textit{weight} of the formula of a $k$-SAT formula $G$ by
\[
\wt(G)=\alpha_1(G)+\log_23 \cdot \alpha_2(G)+\log_24 \cdot \alpha_3(G)+\dots+\log_2(2^k+1) \cdot \alpha_{2^k}(G).
\]
\end{defn}

Given a $k$-SAT formula $G$ on $n$ variables, the number of simple $k$-SAT subformulae of $G$ is
\[
\prod_{j =1}^{2^k} (j+1)^{\alpha_j(G) {n \choose k}} = 2^{\wt(G) {n \choose k}}
\]
since there are $j+1$ ways to choose at most one clause from  $j$ clauses. This motivates our definition of weight.

We are now ready to prove \cref{thm:formulae-upperbound} that the number of minimal $k$-SAT formulae on $n$ variables is at most $2^{(\pi(\cF_k, \log_2 3) + o(1)) \binom{n}{k}}$.

\begin{proof}[Proof of \cref{thm:formulae-upperbound}]
Let $\cB$ be the set of all non-minimal simple formulae on at most $4k$ variables, one for each isomorphism class (so that $\cB$ is finite).
Fix any $\epsilon > 0$.
Pick $\delta > 0$ to be smaller than the minimum of the $\delta$'s chosen in \cref{lem:supersat-1} (with $\theta = \log_2 3$) and \cref{lem:supersat-2}.
Applying the container result, \cref{thm:gencontub}, we find a collection $\cG$ of $n$-variable formulae on $X=\{x_1,\dots,x_n\}$ such that
\begin{enumerate}
    \item [(a)] Every $\cB$-free formula on $X$ is a subformula of some $G \in \cG$, so in particular, every simple minimal formula on $X$ is a subformula of some $G \in \cG$; and
    \item[(b)] For every $G \in \cG$ and $B \in \cB$, $G$ contains at most $\delta n^{v(B)}$ copies of $B$, which implies, by \cref{lem:supersat-1,lem:supersat-2}, that $\alpha_1(G) + \log_2 3 \cdot \alpha_{2,1}(G)< \pi(\cF_k, \log_2 3) + \epsilon$, $\alpha_{2,2}(G) < \epsilon$ and $\alpha_i(G)< \epsilon$ for all $i \ge 3$; and 
    \item[(c)] $\abs{\cG} \le 2^{o(n^k)}$.
\end{enumerate}

The number of simple minimal $k$-SAT subformula of each $G \in \cG$ is thus at most
\[
2^{\wt (G) \binom{n}{k}} \le 2^{(\pi(\cF_k, \log_2 3) + O(\epsilon))\binom{n}{k}}.
\]
By taking a union bound over all $G \in \cG$, of which there are at most $2^{o(n^k)}$, and noting that $\epsilon$ can be taken to be arbitrarily small, we see that the number of simple minimal $k$-SAT formula on $n$ variables is $2^{(\pi(\cF_k, \log_2 3) + o(1))\binom{n}{k}}$. 

Finally, by \cref{prop:nearly-simple},  we can obtain any minimal $k$-SAT formulae by adding $o(n^k)$ clauses to a simple minimal formula. This adds a neglible factor $\binom{2^k{n\choose k}}{\le o(n^k)} = 2^{o(n^k)}$. So the total number of minimal $k$-SAT formulae on $n$ variables is $2^{(\pi(\cF_k, \log_2 3) + o(1))\binom{n}{k}}$.
\end{proof}

\section{A lower bound on the number of $k$-SAT functions}\label{sec:lb}

In this section, we prove the lower bound in \cref{thm:weak-count}, namely that for every fixed $k \ge 2$, there are at least $2^{(\pi(\cF_k, \log_2 3)+o(1))\binom{n}{k}}$ distinct $k$-SAT functions on $n$ variables.

\begin{defn}
Given an $k$-PDG $\vec H$, let the \textit{positive instance} of $\vec H$ be the semisimple $k\SAT$ formula $G \in \type^{-1}(\vec H)$ obtained as follows:
\begin{itemize}
    \item  $G$ has the clause $v_1\cdots v_k$ if and only if there exists an edge (directed or undirected) on $\{v_1,\dots, v_k\}$ in $\vec H$;
    \item $G$ has the clause $v_1\cdots v_{k-1}\ol v_k$ if and only if $v_1\cdots  v_{k-1}\wc v_k\in E(\vec H)$.
\end{itemize}
\end{defn}

We observe that the positive instance $G \in \type^{-1}(\vec H)$ on $n$ vertices has $2^{\alpha(\vec H) {n \choose k}} 3^{\beta(\vec H) {n \choose k}}=2^{(\alpha(\vec H)+\log_23\cdot \beta(\vec H)) {n \choose k}}$ simple subformulae, as there are $\alpha(\vec H){n\choose k}$ $k$-subsets of variables in $G$ that support exactly one clause, and $\beta(\vec H){n\choose k}$ $k$-subsets of variables in $G$ that support exactly two clauses.

\begin{lemma}\label{l:subminimal}
Fix $\cF_k$-free $k$-PDG $\vec H$ on $n$ vertices with positive instance $G \in \type^{-1}(\vec H)$. Let $\cS$ denote the collection of simple subformulae of $G$. Then all $G' \in \cS$ are minimal.
\end{lemma}
\begin{proof}Fix any $G' \in \cS$. To prove that $G'$ is minimal, we show that for every clause $C_0\in G'$, we can find a witness $\mathsf w\in\{0,1\}^n$ such that $\mathsf w$ maps $C_0$ to 1 and all other clauses in $G'$ to $0$.

Fix any clause $C_0\in G'$. First suppose that $C_0$ is of the form $v_1\cdots v_k$. Let
$$X_1=\{v_1,\dots,v_k\}\cup\{w:v_{i_1}\cdots v_{i_{k-1}}\ol w\in G\text{ for some }i_1,\dots,i_{k-1}\in[k]\}.$$
Choose $\mathsf w\in\{0,1\}^n$ such that $\mathsf w(v)=1$ if $v\in X_1$ and $\mathsf w(v)=0$ otherwise. We show that $\mathsf w$ is a witness to $C_0$. For every $C\in G'\setminus\{C_0\}$, there are three cases:
\begin{itemize}
    \item $C$ uses at least two variables outside $X_1$. In this case, since $C$ has at most one negative literal, we know that $\mathsf w$ maps $C$ to 0.
    \item $C$ uses exactly one variable $u_0$ outside $X_1$. In this case, we show that the corresponding literal in $C$ cannot be $\ol u_0$, so $\mathsf w$ still maps $C$ to 0. By contradiction, suppose that $C$ contains $\ol u_0$, so $C$ is of the form $v_{i_1}\cdots v_{i_j}w_{i_{j+1}}\cdots w_{i_{k-1}}\ol u_0$ for some $i_1,\dots,i_{j}\in[k]$ and some $w_{i_{j+1}},\dots,w_{i_{k-1}}\in X_1\setminus \{v_1,\dots,v_k\}$. However, if $j=k-1$, then $u_0\in X_1$, a contradiction; if $j<k-1$, then the underlying $k$-PDG $\vec H$ is not $\cF_k$-free, as it has a subgraph of the form
    $$\qquad\{v_1\cdots v_k, v_{i^{(j+1)}_1}\cdots v_{i^{(j+1)}_{k-1}}\wc w_{i_{j+1}},\dots, v_{i^{(k-1)}_1}\cdots v_{i^{(k-1)}_{k-1}}\wc w_{i_{k-1}}, v_{i_1}\cdots v_{i_j}w_{i_{j+1}}\cdots w_{i_{k-1}}\wc u_0\},$$ 
    another contradiction.
    \item $C$ only uses variables in $X_1$. In this case, to show that $\mathsf w$ maps $C$ to 0, it suffices to show that $C$ cannot be monotone. By contradiction, suppose that $C$ is of the form $v_{i_1}\dots v_{i_j}w_{i_{j+1}}\dots w_{i_{k}}$ for some $i_1,\dots,i_{j}\in[k]$ and some $w_{i_{j+1}},\dots,w_{i_{k}}\in X_1\setminus \{v_1,\dots,v_k\}$. However, if $j=k$, then $C=C_0$, a contradiction;  if $j<k$, then the underlying $k$-PDG $\vec H$ is not $\cF_k$-free, as it has a subgraph of the form
    $$\{v_1\cdots v_k, v_{i^{(j+1)}_1}\cdots v_{i^{(j+1)}_{k-1}}\wc w_{i_{j+1}},\dots, v_{i^{(k)}_1}\cdots v_{i^{(k)}_{k-1}}\wc w_{i_{k}}, v_{i_1}\dots v_{i_j}w_{i_{j+1}}\dots w_{i_{k}}\},$$
    another contradiction.
\end{itemize}
We now suppose that $C_0$ is of the form $v_1\dots v_{k-1}\ol v_k$. Let
$$X_1=\{v_1,\dots,v_{k-1}\}\cup\{w:v_1\cdots v_{k-1}\ol w\in G'\}.$$
Choose $\mathsf w\in\{0,1\}^n$ such that $\mathsf w(v)=1$ if $v\in X_1$ and $\mathsf w(v)=0$ otherwise. By an almost identical argument, we can show that $\mathsf w$ is a witness to $C_0$.
\end{proof}

\begin{lemma}\label{l:distinctsub}
Fix $\cF_k$-free $k$-PDG $\vec H$ on $n$ vertices with positive instance $G \in \type^{-1}(\vec H)$. Let $\cS$ denote the collection of simple subformulae of $G$. If $G_1 \neq G_2 \in \cS$, then $G_1, G_2$ encode distinct $k\SAT$ functions.
\end{lemma}
\begin{proof}
Consider two distinct formulae in $\cS$, $G_1$ and $G_2$. First suppose that there exists a $k$-subset of variables $S\in{X\choose k}$ such that $G_1$ has a clause $C_1$ on $S$ but $G_2$ does not. By the above claim, $C_1\cup G_2$ is minimal, so there exists some $\mathsf w\in\{0,1\}^n$ that maps $C_1$ to 1 (and thus $G_1$ to 1) and $G_2$ to 0. Hence $G_1$ and $G_2$ encode distinct $k$-SAT functions.

Now suppose that there exists a $k$-subset of variables $S\in{X\choose k}$ such that $G_2$ has a clause $C_2$ on $S$ but $G_1$ does not. By the same argument, we know that $G_1$ and $G_2$ encode distinct $k$-SAT functions.

Finally, suppose that the clauses in $G_1$ and the clauses in $G_2$ lie on exactly the same $k$-subsets of variables. In this case, since $G_1\neq G_2$, there must exist some $S=\{v_1,\dots,v_k\}\in{X\choose k}$ such that $v_1\cdots v_k\in G_1$ and $v_1\cdots \ol v_k\in G_2$ (or vice versa). In this case, since $\{v_1\cdots v_k\}\cup G_2\setminus\{v_1\cdots \ol v_k\}$ is minimal, there exists some $\mathsf w\in\{0,1\}^n$ that maps $v_1\cdots v_k$ to 1 (and thus $G_1$ to 1) and $G_2\setminus\{v_1\cdots \ol v_k\}$ to 0 (and thus $G_2$ to 0). Hence $G_1$ and $G_2$ encode distinct $k$-SAT functions.
\end{proof}

Combining~\cref{l:subminimal} and~\cref{l:distinctsub} implies that for $\cF_k$-free PDG $\vec H$ with positive instance $G$, all simple subformulae of $G$ are minimal and yield distinct $k\SAT$ functions. This immediately implies the following lower bound.

\begin{cor}\label{cor:lb}
For every $\cF_k$-free $k$-PDG $\vec H$ on $n$ vertices, the number of distinct $k$-SAT functions on $n$ variables is at least $2^{(\alpha(\vec H)+\log_23\cdot \beta(\vec H)) {n \choose k}}$.
\end{cor}

\begin{proof}[Proof of Theorem~\ref{thm:weak-count}] 
The upper bound follows from~\cref{thm:formulae-upperbound}. 
To obtain the lower bound, notice that for every $n$, we have
$$\max\{\alpha(\vec H) + \log_23 \cdot (\vec H): \vec H\text{ is an $\cF_k$-free $k$-PDG on $n$ variables}\}\geq \pi(\cF_k, \log_23).$$
This is because we always have
\begin{multline*}
    \max\{\alpha(\vec H) + \log_23 \cdot \beta(\vec H): \vec H\text{ is an $\cF_k$-free $k$-PDG on $n$ variables}\}\\
    \leq \max\{\alpha(\vec H) + \log_23 \cdot \beta(\vec H): \vec H\text{ is an $\cF_k$-free $k$-PDG on $n-1$ variables}\}
\end{multline*}
due to an averaging argument, so the left-hand side is non-increasing with respect to $n$.

Hence for all $n$, we can we can take an $\cF_k$-free $k$-PDG $\vec H_n$ with
$$\alpha(\vec H_n) + \log_23 \cdot \beta(\vec H_n) \geq \pi(\cF_k, \log_23).$$
By~\cref{cor:lb}, the number of distinct $k$-SAT functions on $n$ variables is at least $2^{\pi(\cF_k, \log_23){n\choose k}}$.
\end{proof}

\section{Formulae that are far from unate}\label{s:stability}

We prove~\cref{thm:strong-count} in the following two sections. Let $\cI(n)$ be the set of minimal $k$-SAT formulae on $n$ variables; notice that the number of $k\SAT$ functions is at most $|\cI(n)|$. In this section, we reduce the problem of bounding $|\cI(n)|$ to bounding the size of $\cI^*(n, \zeta)$, the collection of minimal formulae that satisfy a given witness condition. This collection of formulae is defined as follows:

\begin{defn}
Fix a $k$-SAT formula $G$. 
Given a variable $x$, let $m(x)$ denote the number of clauses in $G$ in which the positive literal $x$ appears; analogously, $m(\overline{x})$ is the number of clauses in  $G$ in which the negative literal $\overline{x}$ appears.
\end{defn}

\begin{defn}\label{d:cistar}
For $\zeta>0$, let $\cI^*(n, \zeta)$ be the collection of minimal $k$-SAT formulae where $m(x) \ge m(\overline{x})$ for every variable $x$, and each clause in each formula has a witness $\mathsf w$ (an element of $\{0, 1\}^n$ that satisfies only that clause in the formula) with fewer than $\zeta n$ variables assigned to $1$. 
\end{defn}

We work to obtain a more precise estimate on the number of simple, minimal $k\SAT$ formulae by analyzing the set of containers arising from applying~\cref{thm:gencontub}, where $\cB$ is the set of simple non-minimal formulae on at most $2(k+1)$ vertices.

In this section, we use these containers to enumerate those relatively ``small'' minimal formulae $G$, characterized by having $\alpha_1(G)$ bounded away from $1$. We also bound the number of simple, minimal formulae that are ``far from unate.'' We are able to bound the remaining ``large'' and ``almost unate'' formulae in terms of $|\cI^*(n, \zeta)|$, and give an upper bound on this quantity in~\cref{s:istar}, which allows us to complete the proof of~\cref{thm:strong-count}. 

We make the notion of being ``almost unate'' precise below.

\begin{defn}
For every $\rho > 0$, we say that a $k$-SAT formula $G$ is $\rho$-\emph{close to unate} it can be made unate by removing up to  $\rho\binom{n}{k}$ clauses. 
Otherwise we say that $G$ is $\rho$-\textit{far from unate}.

In addition, for every variable $v$, we say that $v$ is \textit{$\rho$-close to unate} if either $m(v)\leq \rho{n-1\choose k-1}$ or $m(\ol v)\leq \rho {n-1\choose k-1}$. Otherwise we say that $v$ is \emph{$\rho$-far from unate}.
\end{defn}

\subsection{Large, simple $\rho$-far from unate formulae have many subformulae in $\cB$}

We prove that for some appropriately chosen $\rho>0$, $a>0$ and $n$ sufficiently large, every simple $\rho$-far from unate formula with more than $(1-a){n\choose k}$ clauses has many non-unate $(k+1)$-cliques, where a $(k+1)$-clique is a simple formula with $k+1$ clauses on $k+1$ variables. 

By~\cref{lem:formula-blowup}, such a formula also contains many 2-blowups of non-unate $(k+1)$-cliques. Since the 2-blowup of every non-unate $(k+1)$-clique is non-minimal, such $G$ has many simple non-minimal subformulae on $2(k+1)$ vertices. The following definition of ``link" formula will be used in our proof:

\begin{defn}
Given $k$-SAT formula $G$ and variable $v$, the \textit{positive link} of $v$, denoted $G_{v}$, is the $(k-1)\SAT$ formula whose clauses are obtained by taking all clauses in $G$ that contain literal $v$ and removing $v$ from each such clause. The \textit{negative link} of $v$, denoted $G_{\ol v}$, is the $(k-1)\SAT$ formula obtained by taking all all clauses in $G$ that contain literal $\overline{v}$ and removing $\overline{v}$ from each such clause.
\end{defn}

\begin{prop}\label{p:nonpositivenonunate}
For every $0 < \rho < 1$, we can take sufficiently small $a := a(\rho) > 0$ such that the following holds for $n$ sufficiently large. If $G$ is a simple $\rho$-far from unate formula with at least $(1 - a){n\choose k}$ clauses, then for some $\delta_1 = \delta_1(\rho, a)>0$, $G$ has at least $\delta_1 {n \choose k+1}$ non-unate $(k+1)$-cliques.
\end{prop}
\begin{proof}
Take $c = \rho/2k$, $N>c^{-1}$ and  $a > 0$ sufficiently small so that
$$1-kN\cdot a -c^\frac{k}{k-1}-(1-c)^\frac{k}{k-1}>0.$$
Fix some $G$ that is $\rho$-far from unate with $|G|> (1 - a){n\choose k}$. We first show that there are at least $cn$ variables $v\in X$ that are $c$-far from unate in $G$. Notice that we can obtain a unate subformula $G'\subset G$ by doing the following: for every variable $v\in X$, if $m(v)\geq m(\ol v)$, then we delete all clauses in $G$ containing the literal $\ol v$; otherwise we delete all clauses in $G$ containing the literal $v$. If there are less than  $c n$ variables $v\in X$ that are $c$-far from unate in $G$, then we can delete at most
\begin{align*}
    n\cdot c {n-1\choose k-1}+ cn \cdot {n-1\choose k-1} \le \rho{n\choose k}
\end{align*}
clauses to make $G$ unate, which contradicts the fact that $G$ is $\rho$-far from unate.
Also notice that there are at most $n/N$ vertices $v\in X$ with $m(v)+m(\ol v)\leq (1-Na){n-1\choose k-1}$, as otherwise we would have
$$|G|<\frac{1}{k}\left(\frac{n}{N}(1-Na){n-1\choose k-1}+\left(1-\frac{1}{N}\right)n{n-1\choose k-1}\right)\leq (1-a){n\choose k}.$$
Combining the above tells us that there are at least $\left(c-N^{-1}\right)n$ variables that are $c$-far from unate with $m(v)+m(\ol v)> (1-Na){n-1\choose k-1}$.

For every such $v$, consider the positive and negative link $(k-1)$-graphs $G_v$ and $G_{\ol v}$. By the Kruskal--Katona theorem in the form of~\cref{thm:kruskalkatona}, the number of $k$-cliques in $G_v$ is at most $\left((k-1)!|G_v|/(n-1)^{k-1}\right)^\frac{k}{k-1}\cdot (n-1)^{k}/k!$ and the number of $k$-cliques in $G_{\ol v}$ is at most $\left((k-1)!|G_{\ol v}|/(n-1)^{k-1}\right)^\frac{k}{k-1}\cdot (n-1)^{k}/k!$. Since $|G_v|=m(v)$, $|G_{\ol v}|=m(\ol v)$ and $v$ is $c$-far from unate, we know that $|G_v|,|G_{\ol v}| \geq c{n-1\choose k-1}$. By convexity, the number of monochromatic $k$-cliques in $G_v\cup G_{\ol v}$ is at most 
$$\left(c^\frac{k}{k-1}+(1-c)^\frac{k}{k-1}+o(1)\right){n-1\choose k}.$$
Meanwhile, since $|G_v|+|G_{\ol v}|=m(v)+m(\ol v)> (1-Na){n-1\choose k-1}$, the number of $k$-cliques in $G_v\cup G_{\ol v}$ is at least $(1-Nk\cdot a ){n-1\choose k}$.
The number of non-monochromatic $k$-cliques in $G_v\cup G_{\ol v}$ is thus at least
$$\left(1-N k\cdot a-c^\frac{k}{k-1}-(1-c)^\frac{k}{k-1}+o(1)\right){n-1\choose k}.$$

Hence $G$ has at least
$$\frac{1}{k+1}\left(c-N^{-1}\right)n\cdot \left(1-N k\cdot a-c^\frac{k}{k-1}-(1-c)^\frac{k}{k-1}+o(1)\right){n-1\choose k}$$
non-unate $(k+1)$-cliques. Choosing any $0<\delta_1 < \left(c-N^{-1}\right)\left(1-N k\cdot a-c^\frac{k}{k-1}-(1-c)^\frac{k}{k-1}\right)$ yields the result.
\end{proof}

It remains to show that the 2-blowup of any non-unate simplex is non-minimal.

\begin{prop}\label{prop:blowupsofk+1cliquesnonminimal}
Let $H$ be a non-unate simple formula with $k+1$ clauses on $k+1$ variables. Then the 2-blowup $H[2]$ is non-minimal.
\end{prop}
\begin{proof}Let $\{v_1,\dots, v_{k+1}\}$ be the variable set of $H$, and $\{v_1,v_1',\dots,v_{k+1},v_{k+1}'\}$ be the variable set of $H[2]$. Up to negating and/or relabeling some of the variables, $H$ contains a subformula $H'$ comprising $3$ clauses: a clause $C_1 = v_2 \cdots v_{k-1} v_k v_{k+1}$, a clause $C_2$ with positive literal $v_1$ and variables $v_2, v_3, \ldots, v_{k-1}, v_k$ (that may appear as positive or negative literals), and a clause $C_3$ with negative literal $\overline{v_1}$ and variables $v_2, v_3, \ldots, v_{k-2}, v_{k-1}, v_{k+1}$ (that may appear as positive or negative literals).

We show that $H'[2]$ is non-minimal. Since
$$\{v_2v_3\cdots v_{k+1},v_2'v_3\cdots v_{k+1},v_2v_3'\cdots v_{k+1},\dots,v_2v_3\cdots v_{k+1}'\}\subset H'[2],$$
if $\mathsf w\in\{0,1\}^{k+1}$ is a witness to the clause $v_2v_3\cdots v_{k+1}\in H'[2]$, then we must have
$$\mathsf w(v_2)=\dots=\mathsf w(v_{k+1})=1,$$ 
$$\mathsf w(v_2')=\dots=\mathsf w(v_{k+1}')=0.$$
Now if $\mathsf w(v_1)=1$, then there must be a clause in the 2-blowup of $C_2$ that gets mapped to 1. If $\mathsf w(v_1)=0$, then there must be a clause in the 2-blowup of $C_3$ that gets mapped to 1. This is a contradiction, so there is no witness $\mathsf w\in\{0,1\}^{k+1}$ to the clause $v_2v_3\cdots v_{k+1}\in H'[2]$.
\end{proof}

Combining~\cref{lem:formula-blowup}, ~\cref{p:nonpositivenonunate} and~\cref{prop:blowupsofk+1cliquesnonminimal}, we have the following:
\begin{cor}\label{cor:manyblowupsofk+1cliques}
For every $\rho>0$, we can take sufficiently small $a := a(\rho) > 0$ such that the following holds for $n$ sufficiently large. If $G$ is a simple $\rho$-far from unate formula with at least $(1 - a){n\choose k}$ clauses, then there exists $\delta_2 = \delta_2(\rho, a)$ such that $G$ has at least $\delta_2 n^{2(k+1)}$ simple non-minimal subformulae on $2(k+1)$ variables.
\end{cor}

\subsection{Reducing the problem to bounding the size of some $\cI^*(n, \zeta)$}
\label{ss:reduction}

We are now able to give an upper bound on the number of minimal $k$-SAT formulae that do not satisfy a particular witness condition, thereby reducing the problem of bounding the size of $\cI(n)$ to bounding the size of $\cI^*(n,\zeta)$.

We first show that for any collection of $\rho$-close to unate formulae $\cG$, all but few elements of $\bigcup_{G\in\cG}\{H\subset G:H\text{ is minimal}\}$ lie in $\cI^*(n,\zeta)$ up to replacing some of the variables by their negations.

\begin{lemma}\label{lem:familyrhoclosetounate}
For any collection $\cG$ of $\rho$-close to unate formulae and $\zeta>0$, we have
\begin{align*}
    \left|\bigcup_{G\in \cG}\{H\subset G:H\textup{ is minimal}\}\right|&\leq  |\cG|\cdot \left({n\choose\floor{\zeta n}}\left({\floor{\zeta n}\choose k}+1\right)2^{{n\choose k}-{\floor{\zeta n}\choose k}}2^{\rho{n\choose k}}\right)\\
    &\qquad+2^n|\cI^*(n,\zeta)|.
\end{align*}

\end{lemma}
\begin{proof}For every $G\in\cG$, there exists a negation of some variables that maps $G$ to some formula $G'$ having at most $\rho{n\choose k}$ non-monotone clauses. Without loss of generality, we assume that $G'$ has the greatest number of non-monotone clauses. In other words, no negation of variables maps $G$ to a formula having more non-monotone clauses. Notice that this automatically implies that $m(x)\geq m(\ol {x})$ in $G'$ for every $x\in X$.

We give an upper bound on the number of minimal subformulae of $G'$ that do not lie in $\cI^*(n, \zeta)$. By \cref{d:cistar}, for all $H\notin\mathcal \cI^*(n, \zeta)$, there exists a clause in $H$ and a witness $\mathsf w\in\{0,1\}^n$ to that clause, such that $\mathsf w$ assigns at least $\floor{\zeta n}$ variables to 1. Let $(v_1,\dots,v_{\floor{\zeta n}})$ be the lexicographically smallest $(\floor{\zeta n})$-tuple of variables, such that there exists some witness $\mathsf w$ to some clause in $H$ with $\mathsf w(v_1)=\dots=\mathsf w(v_{\floor{\zeta n}})=1$. Notice that since $\mathsf w$ is a witness to some clause in $H$, it cannot map two clauses in $H$ to 1. Hence at most one clause of the form $v_{i_1}\cdots v_{i_k}$ can appear in $H$.

We obtain a minimal formula $H\subset G'$ with $H\notin\mathcal \cI^*(n, \zeta)$ by making the following set of choices.
\begin{itemize} 
    \item We pick the non-monotone clauses in $H$. There are at most $2^{\rho {n \choose k}}$ such ways to make this selection since $G'$ has at most $\rho {n \choose k}$ non-monotone clauses.
    \item We pick the lexicographically smallest subset of literals $\{v_1, \ldots, v_{\floor{\zeta n}}\}$ so that there exists some witness $\mathsf w$ to a clause in $H$ with $\mathsf w(v_1)=\dots=\mathsf w(v_{\floor{\zeta n}})=1$. There are at most ${n \choose \floor{\zeta n}}$ ways to make such a choice.
    \item We choose the monotone clauses of the form $v_{i_1}\cdots v_{i_k}$, of which there are $({\floor{\zeta n}\choose k}+1)$ possible subsets since we can include at most one such clause.
    \item Finally, we choose the remaining monotone clauses, with variables not entirely drawn from $v_1, \ldots, v_{\floor{\zeta n}}$.
\end{itemize}
This yields the following upper bound on the number of  minimal formulae $H \subset G'$ with $H \not \in \cI^*(n, \zeta)$:
$$    {n\choose\floor{\zeta n}}\left({\floor{\zeta n}\choose k}+1\right)2^{{n\choose k}-{\floor{\zeta n}\choose k}}2^{\rho{n\choose k}}.$$

Hence all but at most $|\cG|\cdot \left({n\choose\floor{\zeta n}}\left({\floor{\zeta n}\choose k}+1\right)2^{{n\choose k}-{\floor{\zeta n}\choose k}}2^{\rho{n\choose k}}\right)$ elements of $\bigcup_{G\in\cG}\{H\subset G:H\text{ is minimal}\}$ lie in $\cI^*(n,\zeta)$ up to replacing some of the variables by their negations, and we have the desired bound.
\end{proof}

We now state our main reduction step.

\begin{prop}\label{prop:reducetoboundik}
For every $\zeta>0$ and $\theta>\log_23$ with $\pi(\cF_k,\theta)=1$, there exists $\epsilon>0$ such that for all sufficiently large $n$ we have
$$|\cI(n)| \leq 2^{(1-\epsilon)\binom{n}{k}} + 2^{n} \abs{\cI^*(n, \zeta)}.$$
\end{prop}

\begin{proof}

Given $\zeta>0$ and $\theta>\log_23$ with $\pi(\cF_k,\theta)=1$, we first choose $\rho>0$ sufficiently smaller than $\zeta$ and $\theta$, and then choose $\epsilon>0$ sufficiently smaller than $\rho$.

Let $\cB$ be the set of all simple non-minimal $k\SAT$ formulae on at most $2(k+1)$ vertices. Applying~\cref{thm:gencontub} with $\cB$ and some $\delta > 0$ sufficiently smaller than $\epsilon$, we obtain a collection $\cG=\cG(\cB,\delta)$ of formulae that meets the following conditions:
    \begin{itemize}
        \item Every $\cB$-free formula with $n$ variables is a subformula of some $G \in \cG$;
        \item For every $G \in \cG$ and $B \in \cB$, $G$ has at most $\delta n^{v(B)}$ copies of $B$;
        \item $\abs{\cG} \le n^{C(\cB,\delta)n^{k-1/(m(\cB)-1)}}$.
    \end{itemize}
Moreover, due to~\cref{lem:supersat-1}, ~\cref{lem:supersat-2} and the assumption that $\delta$ is sufficiently smaller than $\epsilon$, every $G\in \cG(\cB,\delta)$ also meets the following conditions:
\begin{itemize}
    \item $\alpha_1(G)+\theta\cdot\alpha_{2,1}(G)< 1+\epsilon$;
    \item $\log_23 \cdot \alpha_{2,2}(G)+\log_24 \cdot \alpha_3(G)+\dots+\log_2(2^k+1) \cdot \alpha_{2^k}(G)<\epsilon$.
\end{itemize}

Every formula in $\cI(n)$ is $\cB$-free, so it is a subformula of some $G\in\cG(\cB,\delta)$. Hence we have
$$|\cI(n)|\leq\left|\bigcup_{G\in \cG(\cB,\delta)}\{H\subset G:H\text{ is minimal}\}\right|.$$
Define a partition $\cG(\cB,\delta)=\cG^{\sm} \sqcup \cG^{\lg}$ by
\begin{align*}
    \cG^{\sm}&=\left\{G\in \cG(\cB,\delta):\alpha_1(G)\leq 1- 2\epsilon - \frac{3\epsilon }{\theta - \log_23}\right\},\\
    \cG^{\lg}&=\left\{G\in \cG(\cB,\delta):\alpha_1(G)> 1- 2\epsilon - \frac{3\epsilon }{\theta - \log_23}\right\}.
\end{align*}

We first give an upper bound on the number of minimal subformulae of every $G\in \cG^{\sm}$. According to~\cref{prop:nearly-simple}, every minimal $k\SAT$ formula can be made simple by deleting $o(n^k)$ clauses. This implies that for every $\epsilon>0$, we have
\begin{align*}
    (\text{\# of minimal subformulae of $G$})&\leq \left(\sum_{t=1}^{o( n^k)}{2^k{n\choose k}\choose t}\right)\cdot(\text{\# of simple minimal subformulae of $G$})\\
    &\leq 2^{\frac\epsilon2 {n\choose k}}\cdot (\text{\# of simple minimal subformulae of $G$})
\end{align*}
for sufficiently large $n$. 

We now show that every $G\in \cG^{\sm}$ has $\wt(G)<1-\epsilon$, so $G$ has at most $2^{(1-\epsilon){n\choose k}}$ simple minimal subformulae. For all $G\in \cG^{\sm}$, since
$$\log_23 \cdot \alpha_{2,2}(G)+\log_24 \cdot \alpha_3(G)+\dots+\log_2(2^k+1) \cdot \alpha_{2^k}(G)<\epsilon,$$ we have
$$\wt(G) < \alpha_1(G)+\log_23\cdot \alpha_{2,1}(G)+\epsilon.$$
If $\alpha_1(G)+\theta\cdot\alpha_{2,1}(G)\leq 1-2\epsilon$, then since $\theta>\log_23$, we clearly have $\wt(G) < 1-\epsilon$. If $\alpha_1(G)+\theta\cdot\alpha_{2,1}(G)> 1-2\epsilon$, then since $\alpha_1(G)\leq 1- 2\epsilon - \frac{3\epsilon}{\theta - \log_23}$, we have $\alpha_{2,1}(G)>\frac{3\epsilon}{\theta-\log_23}$, so that
$$\wt(G) < \alpha_1(G) + \theta \cdot \alpha_{2,1}(G) - (\theta - \log_23)\cdot \alpha_{2,1}(G) +\epsilon< (1 + \epsilon) - 3 \epsilon +\epsilon = 1 - \epsilon.$$
Hence every $G\in \cG^{\sm}$ has $\wt(G)<1 - \epsilon$, which implies that $G$ has at most $2^{(1-\epsilon){n\choose k}}$ simple minimal subformulae. Therefore, when $n$ is sufficiently large, every $G\in \cG^{\sm}$ has at most $2^{\left(1-\frac\epsilon2\right){n\choose k}}$ minimal subformulae, and we have the upper bound
\begin{align}\label{eq:gsmall}
    \left|\bigcup_{G\in \cG^{\sm}}\{H\subset G:H\text{ is minimal}\}\right|\leq |\cG^{\sm}|\cdot 2^{\left(1-\frac\epsilon2\right){n\choose k}}.
\end{align}

We now give an upper bound on $\left|\bigcup_{G\in \cG^{\lg}}\{H\subset G:H\text{ is minimal}\}\right|$. For all $G\in\cG^{\lg}$, we know that
$$\alpha_1(G) > 1 - 2\epsilon - \frac{3\epsilon}{\theta - \log_23}.$$
Choose a simple subformula $G_s\subset G$ with $|G_s|=\alpha_1(G){n\choose k}$.
Since $G$ has at most $\delta n^{v(B)}\leq \delta n^{2(k+1)}$ copies of any $B\in\cB$,  we know that $G_s\subset G$ has a total of at most $$|\cB|\cdot \delta n^{2(k+1)}\overset{(*)}{<}\delta_2\left(\frac{\rho}{2},2\epsilon + \frac{3\epsilon }{\theta - \log_23}\right)n^{2(k+1)}$$
simple non-minimal formulae on $2(k+1)$ variables, where $\delta_2$ is defined in~\cref{cor:manyblowupsofk+1cliques}. Here (*) is due to the fact that $\delta$ is much smaller than $\epsilon$ (and thus also $\rho,\theta$). Since $\epsilon$ is significantly smaller than $\rho$, ~\cref{cor:manyblowupsofk+1cliques} applies, so $G_s$ is $\frac{\rho}{2}$-close to unate.

Meanwhile, since $\alpha_1(G)$ is large, $\alpha_{2,1}(G),\alpha_{2,2}(G),\alpha_3(G),\dots,\alpha_{2^k}(G)$ are all very small. In particular, we have
$$\alpha_{2,1}(G)<\frac{1}{\theta}\left((1+\epsilon)-\left(1 - 2\epsilon - \frac{3\epsilon }{\theta - \log_23}\right)\right)=\frac{3\epsilon}{\theta}+\frac{3\epsilon }{\theta(\theta - \log_23)}$$
and
$$\log_23 \cdot \alpha_{2,2}(G)+\log_24 \cdot \alpha_3(G)+\dots+\log_2(2^k+1) \cdot \alpha_{2^k}(G)<\epsilon.$$
This gives
\begin{align*}
    |G\setminus G_s|&\leq (\alpha_{2,1}(G)+\alpha_{2,2}(G)){n\choose k}+2\alpha_{3}(G){n\choose k}+\dots+(2^k-1)\alpha_{2^k}(G){n\choose k}\\
    &\overset{(**)}{<}\frac{\rho}{2}{n\choose k},
\end{align*}
where (**) is due to the fact that $\epsilon$ is sufficiently smaller than $\rho$. 

Combining the above, we know that every $G\in \cG^{\lg}$ is $\rho$-close to unate. Hence $\cG^{\lg}$ is a family of $\rho$-close to unate formulae. By~\cref{lem:familyrhoclosetounate}, we have
\begin{align}\label{eq:glarge}
\begin{split}
\left|\bigcup_{G\in \cG^{\lg}}\{H\subset G:H\text{ is minimal}\}\right|&\leq  |\cG^{\lg}|\cdot\left({n\choose\floor{\zeta n}}\left({\floor{\zeta n}\choose k}+1\right)2^{{n\choose k}-{\floor{\zeta n}\choose k}}2^{\rho{n\choose k}}\right)\\
&\qquad+2^n|\cI^*(n,\zeta)|\\
&\overset{(***)}{\leq} |\cG^{\lg}|\cdot  2^{(1-\epsilon){n\choose k}}+2^n|\cI^*(n, \zeta)|,
\end{split}
\end{align}
for sufficiently large $n$, where (***) is due to the fact that $\epsilon$ is sufficiently smaller than $\rho$.

Since
$$|\cG^{\sm}|,|\cG^{\lg}|\leq|\cG(\cB,\delta)|\leq n^{C(\cB,\delta)n^{k-1/(m(\cB)-1)}},$$
combining~\cref{eq:gsmall} and~\cref{eq:glarge}, we have the following upper bound on $|\cI(n)|$ for $n$ sufficiently large:
\begin{align*}
    |\cI(n)|&\leq \left|\bigcup_{G\in \cG(\cB,\delta)}\{H\subset G:H\text{ is minimal}\}\right|\\
    &\leq |\cG^{\sm}|\cdot 2^{\left(1-\frac\epsilon2\right){n \choose k}}+|\cG^{\lg}|\cdot 2^{(1-\epsilon){n\choose k}}+2^n|\cI^*(n, \zeta)|\\
    &\leq n^{C(\cB,\delta)n^{k-1/(m(\cB)-1)}}\left(2^{\left(1-\frac\epsilon2\right){n \choose k}}+2^{(1-\epsilon){n\choose k}+n}\right)+2^n|\cI^*(n, \zeta)|\\
    &\leq 2^{\left(1-\frac\epsilon3\right){n \choose k}}+ 2^n|\cI^*(n, \zeta)|. \qedhere
\end{align*}
\end{proof}

\section{Nearly unate formulae}\label{s:istar}

In this section we prove the following recursive bound.
\begin{theorem}\label{t:boundcistar}
There exist $\zeta>0$ and $c>0$ such that for all $n$, we have
\begin{multline*}
|\cI^*(n, \zeta)| \le 2^{n\choose k}+\exp_2\sqb{(1-c){n\choose k}}+\sum_{i=1}^{k-2}\exp_2\sqb{i(1-c){n\choose k-1}}|\cI(n-i)|
\\+\exp_2\sqb{(1-c)k{n\choose k-1}} |\cI(n-k)|+\exp_2\sqb{{n\choose k}-cn^{k-2}}.
\end{multline*}
\end{theorem}

Our approach in this section extends the arguments in Ilinca--Kahn \cite[Section 8]{IK12}.
To obtain the above bound, we consider a sequence of subcollections of $\cI^*(n, \zeta)$ and bound the size of each one, showing that most of the formulae in $\cI^*(n, \zeta)$ are monotone.
\begin{enumerate}
\item We first show that most formulae in $\cI^*(n, \zeta)$ use each of the $n$ variables a comparable number of times amongst their clauses, thus disregarding those formulae that have very few occurrences of at least one variable.
\item We then observe that most formulae in $\cI^*(n, \zeta)$ have very few non-monotone clauses that contain at most $k-2$ negative literals.
\item Next, we show that most formulae in $\cI^*(n, \zeta)$ have very few non-monotone clauses that contain $k-1$ or $k$ negative literals.
\item Finally, we give an upper bound on the number of non-monotone formulae that satisfy the above three conditions.
\item We combine the above to obtain a recursive bound on $|\cI^*(n, \zeta)|$ that gives the desired stability result.
\end{enumerate}

We let $H(\cdot)$ denote the binary entropy function, $H(p)=-p\log_2p-(1-p)\log_2(1-p)$. Recall the standard inequality
\[
\sum_{t\leq p n}{n\choose i}\leq 2^{H(p)n} \quad \text{ for all } 0<p\le 1/2 \text{ and } n \in \NN.
\]

We also recall several definitions regarding $k\SAT$ formulae. Let $X=\{x_1,\dots,x_n\}$ denote the set of $n$ variables on which we count the $k$-SAT formulae in $\cI^*(n, \zeta)$. For every clause $C$ in a minimal formula $G$, we say that $\mathsf w\in\{0,1\}^n$ is a witness to $G$ if $\mathsf w$ maps $C$ to 1 and all other clauses in $G$ to $0$.

The following steps~\cref{s:startik} to~\cref{s:endik} roughly correspond to Steps 1 to 5 in~\cite[Section 8]{IK12}.

\subsection{Most minimal formulae have large support for each variable}\label{s:startik}We start by observing that most formulae in $\cI^*(n, \zeta)$ have lots of clauses involving $v$ for every variable $v\in X$. More precisely, we show that for most formulae in $\cI^*(n, \zeta)$, each variable $v \in X$ is used in at least $\frac{1}{10k}{n-1\choose k-1}$ clauses.

\begin{defn} Let $$\cI_1^*(n, \zeta)=\left\{G\in \cI^*(n, \zeta):\text{each variable $v\in X$ is used at least $\frac{1}{10 k}{n-1\choose k-1}$ times in $G$}\right\}.$$
\end{defn}
\begin{lemma}\label{l:step1}
We have
$$|\cI^*(n, \zeta)\setminus \cI_1^*(n, \zeta)| < \exp_2\sqb{\frac12 {n-1\choose k-1}} |\cI(n-1)|.$$
\end{lemma}
\begin{proof}
We count the number of $G\in \cI^*(n, \zeta)\setminus \cI_1^*(n, \zeta)$. There are at most $$n\sum_{t\leq \frac{1}{10 k}{n-1\choose k-1}}{{2^k{n-1\choose k-1}}\choose t}  < \exp_2\sqb{H\left(\frac{1}{ 10k\cdot 2^k}\right)2^k{n-1\choose k-1}}  < \exp_2\sqb{\frac12{n-1\choose k-1}} $$ ways to choose a variable $v\in X$ that occurs fewer than $\frac{1}{10 k} {n-1 \choose k-1}$ times in $G$, along with the clauses containing $v$. The remaining clauses in $G$ form a minimal formula on $n-1$ variables, so there are at most $|\cI(n-1)|$ ways to choose the remaining clauses. Combining gives the desired bound.
\end{proof}

Using the fact that every variable appears in many clauses in $\cI_1^*(n, \zeta)$ we are able to conclude that for formulae in $\cI^*_1(n, \zeta)$, most clauses are monotone. We make the following more precise observations that will prove useful in subsequent steps.
\begin{lemma}\label{lem:ci_1starproperties}
For every $G \in \cI_1^*(n, \zeta)$, the following are true.
\begin{enumerate}[label=(\alph*)]
    \item For every $k-1$ variables $v_1,\dots, v_{k-1}$, there are at most $\zeta n$ variables $w$ such that $v_1\cdots v_{k-1}\ol w\in G$. The same bound holds for $w$'s with $\ol v_1v_2\cdots v_{k-1}\ol w\in G$, $w$'s with $ v_1\ol v_2v_3 \cdots v_{k-1}\ol w\in G$,  $w$'s with $\ol v_1\ol v_2 v_3\cdots  v_{k-1}\ol w\in G$, and so on.
    \item For every variable $v$ and $j\in\{1,\dots,k-1\}$, there are at most $\zeta n^{k-1}$ clauses in $G$ containing the positive literal $v$ and exactly $j$ negative literals.
    \item For every variable $v$ and $j\in\{1,\dots,k-1\}$, there are at most $\zeta n^{k-1}$ clauses in $G$ containing the negative literal $\ol v$ and exactly $j$ negative literals other than $\ol v$.
    \item For every $i\in \{1,\dots,k\}$, there are at most $\zeta n^{k}$ clauses in $G$ having exactly $i$ negative literals.
    \item $G$ has at most $k\zeta n^{k}$ non-monotone clauses. 
    \item For every variable $v\in X$, $G$ has at least $\frac{1}{20k}  {n-1 \choose k-1} - (k-1)\zeta n^{k-1}$ monotone clauses containing $v$.
\end{enumerate}
\end{lemma}
\begin{proof}We prove these properties of formulae $G\in \cI_1^*(n, \zeta)$.
\begin{enumerate}[label=(\alph*)] 
    \item If there are more than $\zeta n$ such $w$, then consider a witness $\mathsf w\in\{0,1\}^n$ to some $v_1\cdots v_{k-1}\ol w_0$. For any other $w\neq w_0$ such that  $v_1\cdots v_{k-1}\overline{w} \in G$, we must have $\mathsf w(w)=1$, a contradiction. Other cases are similar.
    \item For any $k-2$ literals $ \ol v_1, \dots, \ol v_{j-1}, v_j,\ldots,v_{k-2}$, by (a), $G$ contains at most $\zeta n$ clauses of the form $v\ol v_1 \cdots \ol v_{j-1} v_j\cdots v_{k-2}\ol w$. There are at most $n^{k-2}$ ways to choose these $k-2$ literals.
    \item Same as above.
    \item This easily follows from (b) and (c).
    \item Follows from (d).
    \item Since $G \in \cI^*(n, \zeta)$, we have $m(v) \ge m(\overline{v})$ for all variables $v\in X$. Since $G \in \cI_1^*(n, \zeta)$, we have $m(v) + m(\overline{v}) \ge \frac{1}{10k} {n-1\choose k-1}$. This implies $m(v) \ge \frac{1}{20k}{n-1\choose k-1}$. By (b), there are at most $(k-1)\zeta n^{k-1}$ non-monotone clauses that contain the positive literal $v$.
\end{enumerate}
\end{proof}

\subsection{Most minimal formulae have few clauses with negated literals} 
We will eventually show that most formulae in $\cI_1^*(n, \zeta)$ have very few clauses containing any negative literals. In this subsection, we give an upper bound on the number of formulae in $\cI_1^*(n, \zeta)$ having a large numbers of clauses containing $i$ negative literals, for any $1 \le i \le k-2$.

\begin{defn}For every $\beta_1>0$, we define 
\begin{multline*}
    \cI^*_{2,1}(n, \zeta,\beta_1)=\left\{G\in \cI_{1}^*(\zeta):\text{for all $u\in X$, $G$ contains at most $\beta_1n^{k-1}$ clauses}\right. \\\left.\text{of the form } \ol uv_1\cdots v_{k-1}\right\}.
\end{multline*}
For every $i\in\{2,\dots,k-2\}$ and $\beta_i>0$, we define
\begin{multline*}
    \cI^*_{2,i}(n, \zeta,\beta_1,\beta_i)=\left\{G\in \cI_{2,1}^*(n, \zeta,\beta_1):\text{for all $\{u_1,\dots,u_i\}\in{X\choose i}$, }\right. \\\left.\text{$G$ contains at most $\beta_in^{k-i}$ clauses of the form } \ol u_1\cdots\ol u_iv_1\cdots v_{k-i}\right\}.
\end{multline*}
For every vector $\vec\beta=(\beta_1,\dots,\beta_{k-2})$, we define $\cI_2^*(n, \zeta,\vec\beta)=\bigcap_{i=2}^{k-2}\cI^*_{2,i}(n, \zeta,\beta_1,\beta_i)$.
\end{defn}
We will first show that most formulae in $\cI_1^*(n, \zeta)$ are in $\cI_{2,1}^*(n, \zeta,\beta_1)$ for every $\beta_1 > 0$ and sufficiently small $\zeta$.

\subsubsection{Formulae with lots of clauses $\overline{u}v_1 \cdots v_{k-1}$}

\begin{lemma}\label{l:step2.1}
For every $\beta_1>0$, there exist $\zeta>0$ and $c > 0$ such that
$$|\cI^*_{1}(n, \zeta)\setminus \cI_{2,1}^*(n, \zeta,\beta_1)|\leq \exp_2\sqb{(1-c){n\choose k}}+\exp_2\sqb{(1-c){n\choose k-1}}|\cI (n-1)|.$$
\end{lemma}

\begin{proof}

For every $\beta_1>0$, we choose $\xi>0$ sufficiently smaller than $\beta_1$, choose $\theta_1>0$  sufficiently smaller than $\xi$, and then choose $\zeta>0$  sufficiently smaller than $\theta_1$. Take a small constant $c=c(\beta_1,\xi,\theta_1,\zeta)>0$.

For every $G\in \cI^*_{1}(n,\zeta)\setminus \cI_{2,1}^*(n, \zeta,\beta_1)$, by definition, there exists some $u\in X$ such that $G$ contains more than $\beta_1n^{k-1}$ clauses of the form $\ol uv_1\cdots v_{k-1}$. Set $X_u=X\setminus \{u\}$ and
\begin{align*}
    N_u^+&=\left\{\{v_1,\dots,v_{k-1}\}\subset X_u:uv_1\cdots v_{k-1}\in G\right\},\\
    N_u^-&=\left\{\{v_1,\dots,v_{k-1}\}\subset X_u:\ol uv_1\cdots v_{k-1}\in G\right\},\\
    S_1&=\left\{\{v_1,\dots,v_{k-2}\}\subset X_u:\codeg_{N_u^+}(\{v_1,\dots,v_{k-2}\})<\theta_1n\right\},\\
    S_2&=\left\{\{v_1,\dots,v_{k-2}\}\subset X_u:\codeg_{N_u^-}(\{v_1,\dots,v_{k-2}\})<\theta_1n\right\},\\
    T&={X_u\choose k-2}\setminus(S_1\cup S_2),
\end{align*}
where  $\codeg_{N_u^+}$ (resp. $\codeg_{N_u^-}$) is defined to be
$$\codeg_{N_u^+}(\{v_1,\dots,v_{k-2}\})=|\{w\in X_u: \{v_1,\dots,v_{k-2},w\}\in N_u^+\}|.$$
Since there are at most
$$|S_2|\theta_1 n+\left({n-1\choose k-2}-|S_2|\right)n\leq (|S_1|+|T|)n+\frac{\theta_1 n^{k-1}}{(k-2)!}$$
clauses of the form $\ol uv_1\cdots v_{k-1}$, we have
$$\beta_1n^{k-1}\leq (|S_1|+|T|)n+\frac{\theta_1 n^{k-1}}{(k-2)!},$$
and at least one of the following: 
\begin{enumerate}
    \item $|T|>\theta_1 n^{k-2}$; 
    \item $|T|\leq \theta_1 n^{k-2}$ and $|S_1|\geq \left(\beta_1-\frac{(k-2)!+1}{(k-2)!}\cdot \theta_1\right)n^{k-2}>\xi n^{k-2}$. (Here we used the fact that  $\theta_1$ is sufficiently smaller than $\xi$.)
\end{enumerate}

\begin{claim}
For sufficiently small $c>0$, there are at most $\exp_2\left[(1-c){n\choose k}\right]$ formulae $G\in \cI^*_{1}(\zeta)\setminus \cI_{2,1}^*(n, \zeta,\beta_1)$ where $|T| > \theta_1 n^{k-2}$, for $T$ defined as above.
\end{claim}
\begin{proof}
Consider one such $G$.
Notice that if $\{v_1,\dots,v_{k-2},v_{k-1}\}\in N_u^+$ and $\{v_1,\dots,v_{k-2},v_{k}\}\in N_u^-$, then we cannot have $v_1\cdots v_{k-2}v_{k-1}v_k\in G$ due to the non-minimal formula
$$\{v_1\cdots v_{k-2}v_{k-1}u, v_1\cdots v_{k-2}v_{k}\ol u, v_1\cdots v_{k-2}v_{k-1}v_k\}$$
(as it is impossible to satisfy only the first clause). Hence after specifying (1) $u$ and $T$, (2) the elements in $N_u^+$ containing some element in $T$, and (3) the elements in $N_u^-$ containing some element in $T$, we know at least 
$$\frac{1}{{k\choose 2}}\cdot |T|\cdot \theta_1n\cdot \theta_1n>\frac{1}{{k\choose 2}}\cdot \theta_1n^{k-2}\cdot \theta_1n\cdot \theta_1n=\frac{\theta_1^3n^k}{{k\choose 2}}$$
monotone clauses that cannot belong to $G$. 
To specify some $G \in \cI_1^*(n, \zeta) \backslash \cI_{2, 1}(n,\zeta, \beta_1)$ where $|T| > \theta_1 n^{k-2}$, it suffices to make the following choices. 

\begin{itemize}
\item We choose a variable $u \in X$ where $G$ has more than $\beta_1 n^{k-1}$ clauses $\overline{u} v_1 \cdots v_{k-1}$. There are at most $n$ such choices.
\item We pick the set $T$. Since each element of $T$ is a set of $k-2$ variables, there are at most $2^{{n\choose k-2}}$ choices for $T$.
\item We choose the elements of $N_u^+$ and $N_u^-$ that contain some element of $T$. Since $N_u^+,N_u^-\subset {X_u\choose k-1}$ and $T\subset {X_u\choose k-2}$, the number of choices is at most $2^{|T|n} \cdot 2^{|T|n} \leq \exp_2[2\theta_1n^{k-1}]$.
\item We choose the non-monotone clauses in $G$.  \cref{lem:ci_1starproperties} tells us that $G$ has at most $k\zeta n^k$ non-monotone clauses, so the number of such choices is at most $\binom{(2^k-1)\binom{n}{k}}{\le k \zeta n^k} = \exp_2[O_k(\zeta \log(1/\zeta))\binom{n}{k}]$.
\item Finally, we choose the monotone clauses in $G$. Per above, we know that there are at least $\theta_1^3 n^k / {k \choose 2}$ monotone clauses that cannot belong to $G$; this implies the simple upper bound
$\exp_2\left[{n\choose k}-\frac{\theta_1^3n^k}{{k\choose 2}}\right]$.
\end{itemize}
Hence the number of possible $G$ is at most
$$n\cdot 2^{{n\choose k-2}}\cdot \exp_2[2\theta_1n^{k-1}]\cdot  \exp_2\left[O_k(\zeta \log(1/\zeta))\binom{n}{k}+{n\choose k}-\frac{\theta_1^3n^k}{{k\choose 2}}\right]\overset{(*)}{\leq} \exp_2\left[(1-c){n\choose k}\right],$$
where (*) is due to the fact that $\zeta$ is sufficiently smaller than $\theta_1$, and $c$ is small.
\end{proof}

\begin{claim}
For sufficiently small $c>0$, there are at most $\exp_2\sqb{(1-c){n\choose k-1}}|\cI(n-1)|$ formulae $G\in \cI^*_{1}(n,\zeta)\setminus \cI_{2,1}^*(n, \zeta,\beta_1)$ where $|T|\leq \theta_1 n^{k-2}$ and $|S_1| >\xi n^{k-2}$, for $S_1, T$ defined as above.
\end{claim}
\begin{proof}
Consider one such $G$. We can uniquely specify $G$ via the following 5 choices:
\begin{enumerate}
    \item $u$;
    \item the monotone clauses containing $u$;
    \item the clauses containing $\ol u$ where all other literals are positive;
    \item the clauses containing (i) either $u$ or $\ol u$; (ii) some negative literal other than $\ol u$;
    \item the clauses not involving the variable $u$.
\end{enumerate}
Observe that choosing (2) and (3) is equivalent to choosing $N_u^+$ and $N_u^-$. Since the number of monotone clauses in $G$ using $u$ is at least $\frac{1}{20}{n-1\choose k-1}-(k-1)\zeta n^{k-1}$ according to \cref{lem:ci_1starproperties}, and also at most $\frac{1}{k-1}(|S_1|\theta_1n+({n-1\choose k-2}-|S_1|)n)$, we have the inequality
$$\frac{1}{20}{n-1\choose k-1}-(k-1)\zeta n^{k-1}\leq \frac{1}{k-1}\left(|S_1|\theta_1n+\left({n-1\choose k-2}-|S_1|\right)n\right).$$
After some rearrangement, this gives an upper bound on the size of $S_1$ and thus a lower bound on the size of $S_2$:
\begin{align*}
    &|S_1|\leq \frac{1}{1-\theta_1}\left(\frac{1}{(k-2)!}+(k-1)^2\zeta-\frac{0.9}{20(k-2)!}\right)n^{k-2},\\
    &|S_2|\overset{(*)}{\geq} {n\choose k-2}-\theta_1n^{k-2}-\frac{1}{1-\theta_1}\left(\frac{1}{(k-2)!}+(k-1)^2\zeta-\frac{0.9}{20(k-2)!}\right)n^{k-2}\overset{(**)}{\geq}\xi n^{k-2}.
\end{align*}
Here (*) is due to the fact that $|T|\leq\theta_1n^{k-2}$, and (**) is due to the fact that $\theta_1$ is sufficiently smaller than $\xi$, and $\zeta$ is sufficiently smaller than $\theta_1$. We then count the ways to choose such $G$.
\begin{itemize} 
\item We choose a variable $u \in X$ where $G$ has more than $\beta_1 n^{k-1}$ clauses $\overline{u} v_1 \cdots v_{k-1}$. There are at most $n$ such choices.
\item We choose sets $S_1, S_2, T$. There are at most $2^{{n\choose k-2}} \cdot 2^{{n\choose k-2}} = 4^{{n\choose k-2}}$ to choose $S_1, S_2$ (and consequently $T$).
\item We then choose which elements of $N_u^+$ meet $S_1 \cup T$ and which elements of $N_u^-$ meet $S_2 \cup T$. The number of such choices is at most
\begin{align*}
    2^{|T|n}\left(\sum_{t=1}^{\theta_1n}{n\choose t}\right)^{|S_1|} \cdot 2^{|T|n}\left(\sum_{t=1}^{\theta_1n}{n\choose t}\right)^{|S_2|}&\leq \exp_2\sqb{ 2\theta_1n^{k-1}+H(\theta_1)n (|S_1|+|S_2|)}\\
    &\leq \exp_2\sqb{ 2\theta_1n^{k-1}+H(\theta_1)n^{k-1}}.
\end{align*}
\item We then choose the elements of $N_u^-$ that only meet $S_1\setminus S_2$ and the elements of $N_u^+$ that only meet $S_2\setminus S_1$. By the Kruskal--Katona theorem in the form of \cref{thm:kruskalkatona}, the number of such choices is at most
\begin{align*}
\exp_2 &\sqb{ \frac{((k-2)!|S_1\setminus S_2|)^\frac{k-1}{k-2}}{(k-1)!}+\frac{((k-2)!|S_2\setminus S_1|)^\frac{k-1}{k-2}}{(k-1)!} } \\
&\leq \exp_2 \sqb{ \frac{((k-2)!)^\frac{k-1}{k-2}}{(k-1)!}\left(|S_1\setminus S_2|^\frac{k-1}{k-2}+|S_2\setminus S_1|^\frac{k-1}{k-2} \right)}\\
&\overset{(*)}\le \exp_2 \sqb{\frac{((k-2)!)^\frac{k-1}{k-2}}{(k-1)!} \left(\xi^\frac{k-1}{k-2}+\left(\frac{1}{(k-2)!}-\xi\right)^\frac{k-1}{k-2}\right)n^{k-1}}.
\end{align*}
We obtain (*) by noting that since $x^{(k-1)/(k-2)}$ is convex and $|S_2|\geq \xi n^{k-2}$, we have
$$|S_1\setminus S_2|^\frac{k-1}{k-2}+|S_2\setminus S_1|^\frac{k-1}{k-2}\leq |S_1|^\frac{k-1}{k-2}+\left({n\choose k-2}-|S_1|\right)^\frac{k-1}{k-2}\leq \left(\xi^\frac{k-1}{k-2}+\left(\frac{1}{(k-2)!}-\xi\right)^\frac{k-1}{k-2}\right)n^{k-1}.$$
\item We choose those non-monotone clauses containing the variable $u$ and at least one negated literal that is not $\ol u$. By \cref{lem:ci_1starproperties}, there are at most
$$\left((k-1)!\sum_{t=1}^{\zeta n^{k-1}}{{n\choose k-1}\choose t}\right)^{2(k-1)}=\exp_2\left[O_k(\zeta\log(1/\zeta)){n\choose k-1}\right]$$
choices.
\item We choose the clauses not containing $u$, which we can do in at most $|\cI (n-1)|$ ways.
\end{itemize}
Consequently the number of $G \in \cI^*_{1}(n,\zeta)\setminus \cI_{2,1}^*(n,\zeta,\beta_1)$ with $|T|\leq \theta_1 n^{k-2}$ and $|S_1|>\xi n^{k-2}$ is bounded above by
\begin{align*}
    &n \cdot 4^{{n\choose k-2}} \cdot  \exp_2\left[2\theta_1n^{k-1}+H(\theta_1)n^{k-1}+\frac{((k-2)!)^\frac{k-1}{k-2}}{(k-1)!}\left(\xi^\frac{k-1}{k-2}+\left(\frac{1}{(k-2)!}-\xi\right)^\frac{k-1}{k-2}\right)n^{k-1}\right.\\
    &\left.\hspace{8cm}+O_k(\zeta\log(1/\zeta)){n\choose k-1}\right]|\cI (n-1)|\\
    &\overset{(*)}{\leq}\exp_2\sqb{(1-c){n\choose k-1}}|\cI (n-1)|,
\end{align*}
where (*) is due to the fact that $\theta_1$ and $\zeta$ are sufficiently smaller than $\xi$, and $c$ is small.
\end{proof}
Combining the above two cases, we have
$$|\cI^*_{1}(n,\zeta)\setminus \cI_{2,1}^*(n, \zeta,\beta_1)|\leq \exp_2\sqb{(1-c){n\choose k}}+\exp_2\sqb{(1-c){n\choose k-1}}|\cI (n-1)|.$$
\end{proof}

\subsubsection{Formulae with lots of clauses $\overline{u_1} \cdots \ol u_i v_1 \cdots v_{k-i}$}
Next, we bound the number of formulae in $\cI_{2,1}^*(n, \zeta,\beta_1)$ that have many clauses of the form $\ol u_1\cdots \ol u_iv_1\cdots v_{k-i}$ for some $2\leq i\leq k-2$ and $u_1,\dots,u_i\in X$.

\begin{lemma}\label{l:step2.2}
Fix some  $i\in\{2,\dots,k-2\}$. Then, for any $\beta_i>0$, there exist $\beta_1>0$, $\zeta>0$ and $c > 0$ such that
$$|\cI^*_{2,1}(n, \zeta,\beta_1)\setminus \cI_{2,i}^*(n, \zeta,\beta_1,\beta_i)|\leq \exp_2\left[(1-c){n\choose k}\right]+\exp_2\left[i(1-c){n\choose k-1}\right]|\cI(n-i)|.$$
\end{lemma}

\begin{proof}
Fix $i$. For every $\beta_i>0$, we choose $\theta_i>0$ and $\beta_1>0$ sufficiently smaller than $\beta_i$, and then choose $\zeta>0$ sufficiently smaller than $\theta_i$. Take a small constant $c = c(\beta_i,\theta_i, \beta_1,\zeta) > 0$.

For every $G\in \cI_{2,1}^*(n,\zeta,\beta_1)\setminus \cI_{2,i}^*(n,\zeta,\beta_1,\beta_i)$, by definition, there exist some $u_1,\dots,u_i\in X$ such that there are more than $\beta_in^{k-i}$ clauses of the form $\ol u_1\cdots\ol u_iv_1\cdots v_{k-i}$. Let $X_{\mathbf u}=X\setminus\{u_1,\dots,u_i\}$. For every $j\in\{1,\dots,i\}$ and $\mathbf v=\{v_1,\dots,v_{k-i}\}\subset X_{\mathbf u}$, set
$$N_{u_j,\mathbf v}=\left\{\{w_1,\dots,w_{i-1}\}\in {X_{\mathbf u}\choose i-1}:u_jv_1\cdots v_{k-i}w_1\cdots w_{i-1}\in G\right\}.$$
Additionally, let
\begin{align*}
    M&=\left\{\{v_1,\dots,v_{k-i}\}\in{X_{\mathbf u}\choose k-i}:\ol u_1\cdots\ol u_iv_1\cdots v_{k-i}\in G\right\},\\
    A&=\left\{\mathbf v=\{v_1,\dots,v_{k-i}\}\in M:\left|\bigcap_{j=1}^iN_{u_j,\mathbf v}\right|\geq\theta_in^{i-1}\right\}, \\
    B&=M\setminus A.
\end{align*}
Since $|M|>\beta_in^{k-i}$, at least one of the following holds:
\begin{enumerate}
    \item $|A|>\theta_in^{k-i}$;
    \item $|B|>(\beta_i-\theta_i)n^{k-i}$.
\end{enumerate}

\begin{claim}
For sufficiently small $c>0$, there are at most $\exp_2\left[(1-c){n\choose k}\right]$ formulae $G \in \cI_{2,1}^*(n,\zeta,\beta_1)\setminus \cI_{2,i}^*(n,\zeta,\beta_1,\beta_i)$, with $|A|>\theta_in^{k-i}$ for $A$ defined as above.
\end{claim}
\begin{proof}
Consider one such $G$. Since $G$ is minimal, for every $\mathbf v=\{v_1,\dots,v_{k-i}\}\in M$ and $\{w_1,\dots,w_{i-1}\}\in \bigcap_{j=1}^iN_{u_j,\mathbf v}$, to avoid the non-minimal subformula
$$\{v_1\cdots v_{k-i}w_1\cdots w_{i-1}z, v_1\cdots v_{k-i}w_1\cdots w_{i-1}u_1,\ldots,v_1\cdots v_{k-i}w_1\cdots w_{i-1}u_i, \ol u_1\cdots\ol u_iv_1\cdots v_{k-i}\},$$
(as it is impossible to satisfy only the first clause), we cannot have $v_1\cdots v_{k-i}w_1\cdots w_{i-1}z\in G$ for any $z\in X$. Hence after specifying (1) $u_1,\dots,u_i$, (2) $A$ and (3) $\bigcap_{j=1}^iN_{u_j,\mathbf v}$ for every $\mathbf v\in A$, we know at least
$$\frac{1}{i{k\choose i}}\cdot |A|\cdot \theta_in^{i-1}\cdot (n-(k-1))\geq \frac{1}{i{k\choose i}}\cdot \theta_in^{k-i}\cdot \theta_in^{i-1}\cdot (n-(k-1))\geq \frac{\theta_i^2\cdot k!}{i{k\choose i}}{n\choose k}$$
monotone clauses that cannot belong to $G$. We can specify some $G$ by making the following set of choices.
\begin{itemize}
    \item We choose $i$ variables $u_1, \ldots , u_i$. This can be done in at most ${n\choose i}$ ways.
    \item We choose $A$. There are at most $2^{n\choose {k-i}}$ possibilities.
    \item For each $\mathbf v\in A$, we pick  $\bigcap_{j=1}^iN_{u_j,\mathbf v}$. There are at most $2^{{n\choose {i-1}}|A|}\leq 2^{n^{k-1}}$ ways to make this set of choices.
    \item We select the non-monotone clauses in $G$, from at most $\exp_2\sqb{O_k(\zeta\log(1/\zeta)){n\choose k}}$ choices.
    \item Finally, we choose the monotone clauses in $G$. We can do this in at most $$\exp_2\left[{n\choose k}-\frac{\theta_i^2\cdot k!}{i{k\choose i}}{n\choose k}\right]$$ ways given the set of monotone clauses excluded by the above argument. 
\end{itemize}
Hence the number of possible $G$ is at most
\begin{align*}
    &{n\choose i}\cdot 2^{n\choose {k-i}}\cdot 2^{n^{k-1}}\cdot \exp_2\left[O_k(\zeta\log(1/\zeta)){n\choose k}+{n\choose k}-\frac{\theta_i^2\cdot k!}{i{k\choose i}}{n\choose k}\right]\\
    &\overset{(*)}{\leq}\exp_2\left[(1-c){n\choose k}\right],
\end{align*}
where (*) is due to the fact that $\zeta$ is sufficiently smaller than $\theta_i$, and $c$ is small.
\end{proof}

\begin{claim}
For sufficiently small $c>0$, there are at most $\exp_2\left[i(1-c){n\choose k-1}\right]|\cI(n-i)|$ formulae $G \in \cI^*_{2,1}(n,\zeta,\beta_1)\setminus \cI^*_{2,i}(n,\zeta,\beta_1,\beta_i)$, with  $|B|>(\beta_i-\theta_i)n^{k-i}$ for $B$ defined as above.
\end{claim}
\begin{proof}
Consider one such $G$. 
We set
$$    Q=\left\{\{w_1,\dots,w_{k-1}\}\in {X_{\mathbf u}\choose k-1}:\exists j\in[i]\text{ such that }u_jw_1\cdots w_{k-1}\notin G\right\},$$
so that for every $\{w_1,\dots,w_{k-1}\}\in Q$, there are at most $2^i-1$ ways to choose which of $u_1w_1\cdots w_{k-1}$, \dots, $u_iw_1\cdots w_{k-1}$ lie in $G$. Notice that for every $\mathbf v=\{v_1,\dots,v_{k-i}\}$, we have
$$\left\{\{v_1,\dots,v_{k-i},w_1,\dots,w_{i-1}\}\in{X_\mathbf{u}\choose k-1}:\{w_1,\dots,w_{i-1}\}\notin \bigcap_{j=1}^iN_{u_j,\mathbf v}\right\} \subset Q.$$
We can fully specify $G$ via the following choices.
\begin{itemize}
    \item We choose $u_1,\dots,u_i$. There are at most ${n \choose i}$ ways to do this.
    \item We choose $B$. There are at most $2^{n\choose k-i}$ possibilities.
    \item We then choose $\bigcap_{j=1}^iN_{u_j,\mathbf v}$ for every $\mathbf v\in B$. Since every $\mathbf v\in B$ has $|\bigcap_{j=1}^iN_{u_j,\mathbf v}|<\theta_in^{i-1}$, the number of such choices is bounded above by $$\left(\sum_{t=1}^{\theta_in^{i-1}}{{n\choose i-1}\choose t}\right)^{|B|}\leq\exp_2\left[O_k(\theta_i\log(1/\theta_i)){n\choose i-1}{n\choose k-i}\right].$$
    \item We choose the monotone clauses involving some $u_j$. Since $|B|>(\beta_i-\theta_i)n^{k-i}$ and $|\bigcap_{j=1}^iN_{u_j,\mathbf v}|<\theta_in^{i-1}$ for every $\mathbf v\in B$,  after specifying $B$ and $\bigcap_{j=1}^iN_{u_j,\mathbf v}$ for every $\mathbf v\in B$, we know at least $${k-1\choose i-1}^{-1} (\beta_i-\theta_i)n^{k-i} \left({|X_\mathbf{u}\setminus \{v_1,\dots,v_{k-i}\}|\choose i-1}-\theta_in^{i-1}\right)$$ $(k-1)$-subsets of variables that belong to $Q$. Hence there are at most 
\begin{multline*}
    \exp_2\left[i{n-1\choose k-1}-(i-\log(2^i-1)){k-1\choose i-1}^{-1}(\beta_i-\theta_i)n^{k-i} \left({n-1-(k-i)\choose i-1}-\theta_in^{i-1}\right)\right]
\end{multline*}
ways of choosing the monotone clauses in $G$ that contain some $u_j$.
\item We next select the clauses containing some $\ol u_j$ where all other literals are positive. Since $G\in \cI_{2,1}^*(n, \zeta,\beta_1)$, there can be at most $\beta_1n^{k-1}$ such clauses for every $u_j$, so the total number of choices is at most
$$\left(\sum_{t\leq \beta_1n^{k-1}}{{n-1\choose k-1}\choose t}\right)^i\leq\exp_2\sqb{O_k(\beta_1\log(1/\beta_1)){n-1\choose k-1}}.$$
\item We then select the clauses containing some variable $u_j$ and some negative literal other than $\ol u_j$. By~\cref{lem:ci_1starproperties}, we can choose them in at most $\exp_2\sqb{O_k(\zeta\log(1/\zeta)){n\choose k-1}}$ ways.
\item Finally we select the clauses that do not involve any variable $u_j$. There are at most $|\cI(n-i)|$ choices.
\end{itemize}
Hence the number of possible $G$ is at most

\begin{align*}
    &{n\choose i}\cdot 2^{n\choose k-i}\cdot |\cI(n-i)|\\
    &\quad\cdot  \exp_2\left[O_k(\theta_i\log(1/\theta_i)){n\choose i-1}{n\choose k-i}+O_k(\zeta\log(1/\zeta)){n\choose k-1}+O_k(\beta_1\log(1/\beta_1)){n-1\choose k-1}\right.\\
    &\left.\qquad\qquad+i{n-1\choose k-1}-(i-\log(2^i-1)){k-1\choose i-1}^{-1}(\beta_i-\theta_i)n^{k-i} \left({n-1-(k-i)\choose i-1}-\theta_in^{i-1}\right)\right]\\
    &\overset{(*)}{\leq} \exp_2\left[i(1-c){n\choose k-1}\right]|\cI(n-i)|,
\end{align*}
where (*) is due to the fact that $\theta_i$, $\beta_1$ and $\zeta$ are all sufficiently smaller than $\beta_i$, and $c$ is small.
\end{proof}

Combining the above two cases, we have
$$|\cI^*_{2,1}(n,\zeta,\beta_1)\setminus \cI_{2,i}^*(n, \zeta,\beta_1,\beta_i)|\leq \exp_2\sqb{(1-c){n\choose k}}+\exp_2\sqb{i(1-c){n\choose k-1}}|\cI(n-i)|$$
for sufficiently small $c > 0$.
\end{proof}

\subsection{Monotone neighborhoods of variables} We build the tools for bounding the number of formulae in $\cI_2^*(n, \zeta,\vec\beta)$ having a large number of clauses containing either $k-1$ or $k$ negative literals.
Here we consider the size of the ``monotone neighborhoods'' of variable $u$ and the intersections of these neighborhoods, an idea we make more precise below.

\begin{defn}
For every $u\in X$, let $Z_{u} = \{\{v_1, \ldots v_{k-1}\} \in \binom{X}{k-1} \mid uv_1\cdots v_{k-1} \in G\}$, i.e., those $(k-1)$-subsets of variables that form a monotone clause with $u$. Let $\overline{Z_u} = {X \backslash \{u\} \choose k-1} \backslash Z_u$, i.e., those $(k-1)$-subsets of variables that do not form a monotone clause with $u$. Let

\begin{align*}
\cI_3^*(n, \zeta, \vec\beta) &=\left\{G \in \cI_2^*(n, \zeta, \vec\beta) : \forall u_1, u_2, \ldots, u_k \in X,\,\right. \\ 
&\left.\qquad|Z_{u_1} \cap \cdots \cap Z_{u_{k-1}} \cap Z_{u_k}|, |Z_{u_1} \cap \cdots \cap Z_{u_{k-1}} \cap \overline{Z_{u_k}}|,|Z_{u_1} \cap \cdots \cap \overline{Z_{u_{k-1}}} \cap \overline{Z_{u_k}}|,\right. \\ 
&\left.\qquad\qquad\dots,|\overline{Z_{u_1}} \cap \cdots \cap \overline{Z_{u_{k-1}}} \cap \overline{Z_{u_k}} | \ge \frac{1}{2^{k+1}} {n-k \choose k-1} \right\}.
\end{align*}

\end{defn}

\begin{lemma} \label{l:step3}
For all $\vec \beta=(\beta_1,\dots,\beta_{k-2}) > 0$ with $\beta_1$ sufficiently small, there exist $\zeta, c > 0$ such that
$$|\cI_2^*(n, \zeta, \vec\beta) \backslash \cI_3^*(n, \zeta, \vec\beta)| \le \exp_2\left[k(1-c){n\choose k-1}\right]|\cI(n-k)|.$$
\end{lemma}
\begin{proof}
Fix $\vec \beta > 0$ with $\beta_1$ sufficiently small. Take $\zeta>0$ sufficiently smaller than $\beta_1$, and  $c=c(\beta_1,\zeta)>0$. To specify some $G\in\cI_2^*(n, \zeta, \vec\beta) \backslash \cI_3^*(n, \zeta, \vec\beta)$, it suffices to pick the following.
\begin{itemize}
    \item We choose $u_1, \ldots, u_k$ violating the condition for $\cI_3^*(n, \zeta, \vec\beta)$, which we can do in at most ${n \choose k}$ ways.
    \item We choose the non-monotone clauses that include at least one of $u_1, \ldots, u_k$. For every $u_j$, since there are at most $\beta_1n^{k-1} + 2(k-1)\zeta n^{k-1}$ such clauses, we choose from at most 
    \begin{align*}
    &\sum_{t \le (\beta_1 + 2(k-1)\zeta)n^{k-1}} {(2^k - 1){ n - 1 \choose k-1} \choose t} \\
    &\le \exp_2 \left[  H \left( \frac{(\beta_1 + 2(k-1)\zeta) n^{k-1}}{(2^k - 1){n-1 \choose k-1}} \right) \cdot (2^k - 1) {n-1 \choose k-1} \right]
    \end{align*} choices.
    This gives in total at most 
    $$\exp_2\left[ O_k((\zeta+\beta_1)\log(1/(\zeta+\beta_1))){n \choose k-1}\right]$$
    choices.
    \item We pick the monotone clauses involving at least one of $u_1, \ldots, u_k$. Note that there are $2^{O(n^{k-2})}$ choices of subsets of clauses involving at least two of $u_1,\ldots, u_k$. Meanwhile, the number of monotone clauses involving exactly $1$ of $u_1, \ldots u_k$ is bounded by the number of choices of an ordered partition of ${X \backslash \{ u_1, \ldots, u_k\} \choose k-1}$ into $2^k$ parts (i.e., choosing membership in $Z_{u_i}$ or $\overline{Z_{u_i}}$ for every $(k-1)$-subset of variables), at least one of which has size less than $\frac{1}{2^{k+1}} {n-k \choose k-1}$. Note that for $N = {n-k \choose k-1}$ and $K = 2^k$, there are $\binom{N}{< \frac{1}{2^{k+1}} N} \le 2^{H\left(\frac{1}{2^{k+1}}\right) \cdot N}$ to choose $Z_{u_i}$ to be small for any fixed $i$ and $(K - 1)^N$ ways to choose membership for the remaining $u_j$ for $j \neq i$.
    The total number of choices is bounded above by $$2^k \exp_2\left[\left(H \left(\frac{1}{2^{k+1}}\right) + \log_2 (2^k - 1) \right) {n-k \choose k-1}  \right].$$
    \item Finally, we choose the clauses not involving any of $u_1, \ldots, u_k$. We can do this in at most $|\cI(n-k)|$ ways.
\end{itemize}
Hence the number of possible $G\in \cI_2^*(n, \zeta, \vec\beta) \backslash \cI_3^*(n, \zeta, \vec\beta)$ is at most
\begin{align*}
    &2^k {n\choose k}\cdot |\cI(n-k)|\cdot  \exp_2 \left[ O_k((\zeta+\beta_1)\log(1/(\zeta+\beta_1))){n \choose k-1}\right.\\
    &\left.\hspace{5cm}+\left(H \left(\frac{1}{2^{k+1}}\right) + \frac{2^{k+1}-1}{2^{k+1}}\cdot \log_2 (2^k-1)\right) {n-k \choose k-1} \right]\\
    &\overset{(*)}{<} \exp_2\left[k(1-c){n\choose k-1}\right]|\cI(n-k)|,
\end{align*}
where (*) holds since $\zeta$ is sufficiently smaller than $\theta_1$, and $c$ is small.
\end{proof}

\subsection{Most formulae in $\cI_3^*(n, \zeta, \vec \beta)$ are monotone}\label{s:endik}
In this step, we bound the number of formulae in $\cI_{3}^*(n,\zeta,\vec\beta)$ that are not monotone. Before stepping into the proof, we first summarize some properties of every $G\in \cI_3^*(n,\zeta,\vec\beta)$ that will be useful subsequently.
\begin{itemize}
    \item For every $k-1$ variables $v_1,\dots, v_{k-1}$, there are at most $\zeta n$ variables $w$ such that $v_1\cdots v_{k-1}\ol w\in G$. The same bound holds for $w$'s with $\ol v_1v_2\cdots v_{k-1}\ol w\in G$, $w$'s with $ v_1\ol v_2v_3 \cdots v_{k-1}\ol w\in G$,  $w$'s with $\ol v_1\ol v_2 v_3\cdots  v_{k-1}\ol w\in G$, and so on.
    \item For every $i\in \{1,\dots,k\}$, there are at most $\zeta n^{k}$ clauses in $G$ having exactly $i$ negative literals.
    \item For every $i\in\{1,\dots,k-2\}$ and $\{u_1,\dots,u_i\}\in{X\choose i}$, $G$ contains at most $\beta_in^{k-i}$ clauses of the form $\ol u_1\cdots\ol u_iv_1\cdots v_{k-i}$;
    \item For every $u_1,\dots,u_k\in X$, we have $|Z_{u_1}\cap\dots\cap Z_{u_k}|\geq \frac{1}{2^{k+1}} {n-k \choose k-1}$.
\end{itemize}

Define $\cI_{4}^*(n,\zeta,\vec\beta)$ to be the subset of monotone formulae in $\cI_{3}^*(n,\zeta,\vec\beta)$.

\begin{defn}Let
$$\cI_{4}^*(n,\zeta,\vec\beta)=\{G\in \cI_3^*(n, \zeta,\vec\beta):\text{$G$ is monotone}\}.$$
\end{defn}

We give an upper bound on $|\cI_{3}^*(n,\zeta,\vec\beta)\setminus \cI_{4}^*(n,\zeta,\vec\beta)|$. The following observation about $\ell$-hypergraphs generalizes a theorem on graphs \cite{Hak65} (also see \cite[Theorem 61.1]{Sch03}), and will be useful in our proof.

\begin{lemma}
\label{lem:orientation}Every $\ell$-uniform hypergraph $H$ can be directed (i.e., every edge directed at some vertex) such that for every $\ell-1$ vertices $v_1,\dots,v_{\ell-1}$, there are at most $$\ceil{\frac{(\ell-1)!}{(\ell!)^\frac{\ell-1}{\ell}}\cdot e(H)^\frac{1}{\ell}}$$ edges that contain all of $v_1,\dots,v_{\ell-1}$ and are directed at a vertex $w$ other than $v_1,\dots,v_{\ell-1}$.
\end{lemma}
\begin{proof}
Write
$$L:=\ceil{\frac{(\ell-1)!}{(\ell!)^\frac{\ell-1}{\ell}}\cdot e(H)^\frac{1}{\ell}}.$$
For an $\ell$-graph $H$, construct an associated bipartite graph $B_H$ in which one part is $E(H)$, the edge set of $H$, and the other part is the multiset $\left\{S^L:S\in {V(H)\choose \ell-1}\right\}$. For every $e\in E(H)$ and $S\in {V(H)\choose \ell-1}$, we define that $e\sim S$ in $B_H$ if and only if $S\subset e$. Observe that an edge direction of $H$ in which every $\ell-1$ vertices support at most $L$ edges directed at another vertex is equivalent to an $E(H)$-perfect matching in $B_H$, which exists if and only if $|W|\leq|N_{B_H}(W)|$ for every $W\subset E(H)$.

For every $W\subset E(H)$, $N_{B_H}(W)$ is the set of those $S\in {V(H)\choose \ell-1}$ that are contained in some $e\in W$. By the Kruskal--Katona theorem in the form of \cref{thm:kruskalkatona}, every collection of $(\ell-1)$-vertex sets $\cS\subset{V(H)\choose \ell-1}$ covers at most $\frac{((\ell-1)!|\cS|)^\frac{\ell}{\ell-1}}{\ell!}$ $\ell$-edges, so every $W\subset E(H)$ satisfies
\begin{align*}
    |W|&\leq \frac{\left((\ell-1)!\cdot\frac{|N_{B_H}(W)|}{L}\right)^\frac{\ell}{\ell-1}}{\ell!},\text{ and }\\
    |N_{B_H}(W)|&\geq\frac{(\ell!)^\frac{\ell-1}{\ell}|W|^\frac{\ell-1}{\ell}}{(\ell-1)!}\cdot L\geq  \frac{(\ell!)^\frac{\ell-1}{\ell}|W|^\frac{\ell-1}{\ell}}{(\ell-1)!}\cdot \frac{(\ell-1)!}{(\ell!)^\frac{\ell-1}{\ell}}\cdot e(H)^\frac{1}{\ell}\\
    &= |W|^\frac{\ell-1}{\ell}\cdot e(H)^\frac{1}{\ell}\geq |W|.\qedhere
\end{align*}
\end{proof}

We will also use the following version of Shearer's inequality \cite{CGFS86} (see also \cite[Lemma 6.5]{IK12}). 
\begin{lemma}\label{lem:shearer}
Let $W$ be a set and $\cF$ a family of subsets of $W$. Let $\cH$ be a hypergraph with vertex set $W$, with $\deg_{\cH}(v)\geq k$ for every $v\in W$. Then
$$\log_2|\cF|\leq\frac{1}{k}\sum_{A\in \cH}\log_2|\Tr(\cF,A)|,$$
where  $\Tr(\cF,A)=\{F\cap A:F\in\cF\}$.
\end{lemma}

We now show that most $k$-SAT formulae in $\cI_3^*(\vec\beta,\zeta)$ are also in $\cI_4^*(\vec\beta,\zeta)$ for sufficiently small $\vec\beta, \zeta$.

\begin{lemma}\label{l:step4}
We can choose $\vec \beta=(\beta_1,\dots,\beta_{k-2})$, $\zeta>0$ and $c > 0$ such that 
$$|\cI_3^*(n, \zeta,\vec\beta)\setminus \cI_{4}^*(\zeta,\vec\beta)|< \exp_2\left[{n\choose k}-cn^{k-2}\right].$$
\end{lemma}

\begin{proof}
Choose sufficiently small $\vec \beta=(\beta_1,\dots,\beta_{k-2})$, $\zeta>0$ and $c=c(\vec \beta,\zeta) > 0$.

Observe that every $G\in \cI_3^*(n, \zeta,\vec\beta)$ cannot contain any clause with exactly $k-1$ negative literals. This is because if $\ol u_1\cdots\ol u_{k-1}v\in G$, then since there exists some $\{w_1,\dots,w_{k-1}\}\in Z_{u_1}\cap\dots\cap Z_{u_{k-1}}\cap Z_v$, $G$ would have a non-minimal suformula
$$\{w_1\cdots w_{k-1}v, w_1\cdots w_{k-1}u_1,\dots, w_1\cdots w_{k-1}u_{k-1}, w_1\cdots w_{k-1}v, \ol u_1\cdots\ol u_{k-1}v\}$$
(as it is impossible to satisfy only the first clause).

Next observe that the number of $G\in \cI_3^*(n, \zeta)$ containing a clause with exactly $k$ negative literals is at most $\exp_2[(1-c){n\choose k}]$. Indeed, if $\ol u_1\cdots\ol u_k\in G$, then for all $\{w_1,\dots, w_{k-1}\}\in\bigcap_{i=1}^kZ_{u_i}$, to avoid the non-minimal subformula
$$\{w_1\cdots w_{k-1}v, w_1\cdots w_{k-1}u_1,\dots,w_1\cdots w_{k-1}u_k, \ol u_1\cdots\ol u_k\}$$
(as it is impossible to satisfy only the first clause), there cannot be any $v\in X\setminus \{u_1,\dots,u_k\}$ such that $w_1\cdots w_{k-1}v\in G$. Hence after specifying $u_1,\dots,u_k$ and $\bigcap_{i=1}^kZ_{u_i}$, we know at least $$\frac{1}{k}\cdot (n-2k+1)\left|\bigcap_{i=1}^kZ_{u_i}\right|\geq \frac{n-2k+1}{k\cdot 2^{k+1}}{n-k\choose k-1}$$
monotone clauses that cannot belong to $G$. We can specify such $G$ in the following manner.
\begin{itemize}
    \item We choose $u_1, \ldots, u_k$. There are at most ${n\choose k}$ choices.
    \item We choose $\bigcap_{i=1}^kZ_{u_i}$. This can be done in at most $2^{{n\choose k-1}}$ ways.
    \item We choose the non-monotone clauses in $G$. The number of choices is bounded above by $\exp_2\left[O_k(\zeta\log(1/\zeta)){n\choose k}\right]$.
    \item Finally, we pick the monotone clauses in $G$, of which the number of choices is bounded above by $$\exp_2\sqb{{n\choose k}-\frac{n-2k+1}{k\cdot 2^{k+1}}{n-k\choose k-1}}.$$
\end{itemize}
Hence the number of possible $G$ is at most
$${n\choose k}\cdot 2^{{n\choose k-1}}\cdot \exp_2\left[O_k(\zeta\log(1/\zeta)){n\choose k}+{n\choose k}-\frac{n-2k+1}{k\cdot 2^{k+1}}{n-k\choose k-1}\right]\overset{(*)}{\leq} \exp_2\sqb{(1-c){n\choose k}},$$
where (*) is due to the fact that $\zeta$ and $c$ are sufficiently small.

Finally, we bound the number of $G\in \cI_3^*(n, \zeta,\vec\beta)\setminus \cI_4^*(n, \zeta,\vec\beta)$ containing no clause with exactly $k-1$ or $k$ negative literals. For every such $G$, we decompose $G=G_0\sqcup G_1\sqcup \dots\sqcup G_{k-2}$, such that every $G_i$ is the set of clauses in $G$ with exactly $i$ negative literals. Observe that $|G_1|,\dots,|G_{k-2}|\leq \zeta n^k$ as $G\in \cI^*(n, \zeta)$. We will prove in a second that:
\begin{claim}\label{claim:step4final}
For all $t_1,\dots,t_{k-2}\leq \zeta n^k$, we have
\begin{align*}
    & \left|\left\{G\in \cI_3^*(n, \zeta,\vec\beta)\setminus \cI_4^*(n, \zeta,\vec\beta):\text{$G$ contains no clause with exactly $k$ negative literals, }\right.\right.\\
    &\left.\left.\qquad\qquad\qquad\qquad\qquad\qquad\qquad |G_1|=t_1,\dots, |G_{k-2}|=t_{k-2} \right\}\right|\\
    & <\exp_2\left[{n\choose k}-2cn^{n-2}\right].
\end{align*}
\end{claim}

Since the number of choices of $t_1,\dots,t_{k-2}$ is at most $(\zeta n^k)^{k-2}$, given~\cref{claim:step4final}, we know that there are at most $(\zeta n^k)^{k-2}\cdot \exp_2\left[{n\choose k}-2cn^{k-2}\right]$ formulae in $|\cI_3^*(n, \zeta,\vec\beta)\setminus \cI_4^*(n, \zeta,\vec\beta)|$ containing no clause with exactly $k-1$ or $k$ negative literals. Combining this with the previous two observations, we have the desired bound
\begin{align*}
    |\cI_3^*(n, \zeta,\vec\beta)\setminus \cI_4^*(n, \zeta,\vec\beta)|&\leq (\zeta n^k)^{k-2}\cdot \exp_2\left[{n\choose k}-2cn^{k-2}\right]+\exp_2\left[(1-c){n\choose k}\right]\\
    &\leq \exp_2\left[{n\choose k}-cn^{k-2}\right].
\end{align*}
\end{proof}

\begin{proof}[Proof of~\cref{claim:step4final}]
Consider some $G\in \cI_3^*(n, \zeta,\vec\beta)\setminus \cI_4^*(n, \zeta,\vec\beta)$ such that $G$ contains no clause with exactly $k-1$ or $k$ negative literals, and $|G_1|=t_1$,\dots, $|G_{k-2}|=t_{k-2}$.

For every $i\in[k-2]$, we construct an auxiliary collection $G_i'\subset {X\choose i}\times {X\choose k-i-1}\times X$ based on $G_i$, by doing the following: for every clause $\ol u_1\cdots\ol u_iv_1\cdots v_{k-i}\in G_i$, we pick one $j\in\{1,\dots,k-i\}$ and put the element $(\{u_1,\dots,u_i\},\{v_1,\dots,v_{j-1},v_{j+1},\dots, v_{k-i}\},v_j)$ in $G_i'$. In other words, we obtain $G_i'$ from $G_i$ by making the set of positive literals in every clause ``directed" at some $v_j$. 

There are clearly many ways to build $G_i'$. We claim that there exists some $G_i'$ such that for every $u_1,\dots,u_i\in X$ and $v_1,\dots,v_{k-i-1}\in X\setminus\{u_1,\dots,u_i\}$, there are at most $$\ceil{\frac{(k-i-1)!}{((k-i)!)^\frac{k-i-1}{k-i}}\cdot \beta_i^\frac{1}{k-i}\cdot n}$$ elements of the form $(\{u_1,\dots,u_i\},\{v_1,\dots,v_{k-i-1}\},w)$. For every $u_1,\dots,u_i\in X$, define the $(k-i)$-hypergraph $H$ by
\begin{align*}
    V(H)&=X\setminus\{u_1,\dots,u_i\},\\
    E(H)&=\{\{v_1,\dots,v_{k-i}\}:\ol u_1\cdots\ol u_iv_1\cdots v_{k-i}\in G_i\}.
\end{align*}
Since $G$ contains at most $\beta_in^{k-i}$ clauses of the form $\ol u_1\dots\ol u_iv_1\dots v_{k-i}$, we have $e(H)\leq \beta_in^{k-i}$, so by~\cref{lem:orientation} we know that $H$ can be directed such that every $v_1,\dots,v_{k-i-1}\in X\setminus\{u_1,\dots,u_i\}$ are contained in at most $\ceil{\frac{(k-i-1)!}{((k-i)!)^\frac{k-i-1}{k-i}}\cdot \beta_i^\frac{1}{k-i}\cdot n}$ edges directed at some other vertex. This gives us the $G_i'$ as wanted.

Now we fix this $G_i'$. Define the $(i+1)$-multigraph $H_i$ by
\begin{align*}
   V(H_i)&={X\choose k-i},\\
   E(H_i)&=\{\{\{u_1,v_1,\dots,v_{k-i-1}\},\dots, \{u_i,v_1,\dots,v_{k-i-1}\},\{w,v_1,\dots,v_{k-i-1}\}\}:
   \\&\hspace{6cm} (\{u_1,\dots,u_i\},\{v_1,\dots,v_{k-i-1}\},w)\in G_i'\}.
\end{align*}
For every $e=\{\{u_1,v_1,\dots,v_{k-i-1}\},\dots, \{u_i,v_1,\dots,v_{k-i-1}\},\{w,v_1,\dots,v_{k-i-1}\}\}\in E(H_i)$ and any other $i$ variables $a_1,\dots,a_i\notin\{u_1,\dots, u_i,v_1,\dots,v_{k-i-1},w\}$, since $\ol u_1\cdots \ol u_{i}v_1\cdots v_{k-i-1}w\in G_i$, to avoid the non-minimal subformula
$$\{v_1\cdots v_{k-i-1}a_1\cdots a_iw, v_{1}\cdots v_{k-i-1}a_1\cdots a_iu_1,\dots, v_1\cdots v_{k-i-1}a_1\cdots a_iu_i,\ol u_1\cdots \ol u_{i}v_1\cdots v_{k-i-1}w\}$$
(as it is impossible to satisfy only the first clause), we have
$$G_0\cap \{v_1\cdots v_{k-i-1}a_1\cdots a_iw, v_{1}\cdots v_{k-i-1}a_1\cdots a_iu_1,\dots, v_1\cdots v_{k-i-1}a_1\cdots a_iu_i\}<i+1.$$
Define the formula $F(e,a_1,\dots,a_i)$ by
$$F(e,a_1,\dots,a_i)=\{v_1\cdots v_{k-i-1}a_1\cdots a_iw, v_{1}\cdots v_{k-i-1}a_1\cdots a_iu_1,\dots, v_1\cdots v_{k-i-1}a_1\cdots a_iu_i\}.$$

We also observe that given any vertex $\mathbf v\in V(H_i)$, we have
$$\deg_{H_i}(\mathbf v)\leq (k-i)\left({n\choose i-1}\cdot\ceil{\frac{(k-i-1)!}{((k-i)!)^\frac{k-i-1}{k-i}}\cdot \beta_i^\frac{1}{k-i}\cdot n}+{n\choose i-1}\cdot\zeta n\right).$$
This is because when $\mathbf v$ takes the role of (say) $\{u_1,v_1,\dots,v_{k-i-1}\}$ in an edge, to identify an edge containing $\mathbf v$, it suffices to:
\begin{itemize}
    \item identify $u_1$ from the $k-i$ variables in $\mathbf v$; 
    \item choose $u_2,\dots,u_i\in X$ from at most ${n\choose i-1}$ choices;
    \item choose $w$ from at most $\ceil{\frac{(k-i-1)!}{((k-i)!)^\frac{k-i-1}{k-i}}\cdot \beta_i^\frac{1}{k-i}\cdot n}$ choices, according to the property of $G_i'$.
\end{itemize}
When $\mathbf v$ takes the role of $\{w,v_1,\dots,v_{k-i-1}\}$ in an edge, to identify an edge containing $\mathbf v$, it suffices to:
\begin{itemize}
    \item identify $w$ from the $k-i$ variables in $\mathbf v$;
    \item choose the lexicographically smallest $i-1$ variables among $u_1,\dots, u_i$ from at most ${n\choose i-1}$ choices;
    \item choose the remaining variable from at most $\zeta n$ choices, according to~\cref{lem:ci_1starproperties}.
\end{itemize}

Let $\nu_i$ and $\tau_i$ denote the matching and vertex cover numbers of $H_i$. Choose $\rho_i>0$ sufficiently small, yet still larger than $\beta_i$ and $\zeta$ such that we have the degree bound
$$\deg_{H_i}(\mathbf v)\leq (k-i)\left({n\choose i-1}\cdot\ceil{\frac{(k-i-1)!}{((k-i)!)^\frac{k-i-1}{k-i}}\cdot \beta_i^\frac{1}{k-i}\cdot n}+{n\choose i-1}\cdot\zeta n\right)\leq\rho_in^i$$
for every $\mathbf v\in V(H_i)$. We then have the  inequalities $$(i+1)\nu_i\geq\tau_i\geq\frac{t_i}{\rho_i n^i}.$$

Given $t_i$, $\nu_i$ and $\tau_i$, the number of possible $G_i$ is bounded by the product of:

\begin{itemize}
    \item ${{n\choose k-i}\choose \tau_i}\leq {n^{k-i}\choose\tau_i}$ (choosing a vertex cover $\mathcal T_i$ of $H_i$);
    \item ${{k\choose i}\tau_i{n\choose i}\choose t_i}$ (choosing a collection of $t_i$ clauses, each using at least one member of $\mathcal T_i$).
\end{itemize}
Suppose $G_1,\dots,G_{k-2}$ have been determined and consider the number of possible $G_0$. Let $\mathcal M_i$ be some maximum matching of $H_i$; say $\mathcal M_i=\{e_j:j\in[\nu_i]\}$.
Define the family of monotone formulae $\mathcal J_i$ by
\begin{align*}
    \mathcal J_i&=\bigcup_{j=1}^{\nu_i}\{F(e_j,a_1,\dots,a_i):a_1,\dots,a_i\notin\{u_1,\dots, u_i,v_1,\dots,v_{k-i-1},w\}\\
    & \qquad\qquad\text{ with }e_j=\{\{u_1,v_1,\dots,v_{k-i-1}\},\dots, \{u_i,v_1,\dots,v_{k-i-1}\},\{w,v_1,\dots,v_{k-i-1}\}\}\}.
\end{align*}
For every possible $G_0$ and every $F(e_j,a_1,\dots,a_i)\in \mathcal J_i$, we have $|G_0\cap F(e_j,a_1,\dots,a_i)|< i+1.$
Moreover, since every $F(e_j,a_1,\dots,a_i)$ is of the form $\{y_1\cdots y_{k-1}z_1,\dots, y_1\cdots y_{k-1}z_{i+1}\}$, and every formula of the form $\{y_1\cdots y_{k-1}z_1,\dots, y_1\cdots y_{k-1}z_{i+1}\}$ can come from at most ${k-1\choose i}$ different  $e_j\in\mathcal M_i$, we have
$$|\mathcal J_i|\geq\frac{\nu_i{n-k\choose i}}{{k-1\choose i}}.$$

Let $\mathcal H$ be a (non-uniform) multigraph, whose vertex set is the set of all monotone clauses on variable set $X$, and whose edge set is the union of all $\mathcal J_i$ and a number of single-clause formulae, such that every monotone clause lies in exactly $\sum_{i=1}^{k-2}{k-1\choose i}$ edges. In other words, we have
\begin{align*}
    V(\mcH)={X\choose k},\qquad
    E(\mcH)=\left(\bigcup_{i=1}^{k-2}\mathcal J_i\right)\cup\left(\bigcup_{C\in {X\choose k}}\{C\}^{\sum_{i=1}^{k-2}\left({k-1\choose i}-\eta_i(C)\right)}\right),
\end{align*}
where $\eta_i(C)\leq {k-1\choose i}$ is the number of edges in $\mathcal J_i$ that contain $C$.  Notice that every vertex in $\cH$ has degree exactly $\sum_{i=1}^{k-2}{k-1\choose i}$.

Let $\mathcal G_{0}$ be the collection of all possible $G_0$. By~\cref{lem:shearer}, we have
$$\log_2|\mathcal G_{0}|\leq  \left(\sum_{i=1}^{k-2}{k-1\choose i}\right)^{-1}\sum_{F\in \mathcal J_i} \log_2|\Tr(\cG_0,F)|=\left(2^{k-1}-2\right)^{-1}\sum_{F\in \mathcal J_i} \log_2|\Tr(\cG_0,F)|.$$ For all $F(e_j,a_1,\dots,a_i)\in \mathcal J_i$, since $|G_0\cap F(e_j,a_1,\dots,a_i)|< i+1$ for all $G_0\in \mathcal G_{0}$, we have $|\text{Tr}(\mathcal G_{0},F(e_j,a_1,\dots,a_i))|\leq 2^{i+1}-1$. Since we also have all $\text{Tr}\left(\mathcal G_0,\{T\}^{\sum_{i=1}^{k-2}\left({k-1\choose i}-\eta_i(T)\right)}\right)\leq 2$, it follows that
\begin{align*}
    \log_2|\mathcal G_{0}|
    &\leq \left(2^{k-1}-2\right)^{-1}\left(\sum_{C\in {V\choose k}}\sum_{i=1}^{k-2}\left({k-1\choose i}-\eta_i(C)\right)+\sum_{i=1}^{k-2}|\mathcal J_i|\log_2(2^{i+1}-1)\right)\\
    &\overset{(*)}{=}{n\choose k}-\left(2^{k-1}-2\right)^{-1}\left(\sum_{i=1}^{k-2}(i+1-\log_2(2^{i+1}-1))|\mathcal J_i|\right)\\
    &\overset{(**)}{\leq} {n\choose k}-\left(2^{k-1}-2\right)^{-1}\left(\sum_{i=1}^{k-2}\frac{(i+1-\log_2(2^{i+1}-1)){n-k\choose i}\nu_i}{{k-1\choose i}}\right)\\
    &\overset{(***)}{\leq} {n\choose k}-\left(2^{k-1}-2\right)^{-1}\left(\sum_{i=1}^{k-2}\frac{(i+1-\log_2(2^{i+1}-1)){n-k\choose i}\tau_i}{(i+1){k-1\choose i}}\right)
\end{align*}
where (*) is because $\sum_{T\in {V\choose k}}\eta_i(T)=(i+1)|\mathcal J_i|$, (**) is because $|\mathcal J_i|\geq\frac{\nu_i{n-k\choose i}}{{k-1\choose i}}$, and (***) is because $\frac{\tau_i}{i+1}\leq\nu_i$. For every $i\in\{1,\dots, k-1\}$, set $$\delta_i=\left(2^{k-1}-2\right)^{-1}\cdot \frac{i+1-\log_2(2^{i+1}-1)}{(i+1){k-1\choose i}}>0.$$

Given all the $t_i$, $\nu_i$ and $\tau_i$ for $i\in\{1,\dots,k-2\}$, combining the above with the previous observation that there are at most ${n^{k-i}\choose\tau_i}{{k\choose i}\tau_i{n\choose i}\choose t_i}$ possible $G_i$, we obtain that the number of possible $G\in\cI_3^*(n, \zeta,\vec\beta)\setminus \cI_4^*(n, \zeta,\vec\beta)$  is at most
\begin{align*}
    &\exp_2\left[{n\choose k}-\left(\sum_{i=1}^{k-2}\delta_i{n-k\choose i}\tau_i\right)\right]\cdot \prod_{i=1}^{k-2}{n^{k-i}\choose\tau_i}{{k\choose i}\tau_i{n\choose i}\choose t_i}\\
    &\leq\exp_2\left[{n\choose k}+\sum_{i=1}^{k-2}\left(O_k(\log n)\tau_i+H\left(\frac{t_i}{{k\choose i}\tau_i{n\choose i}}\right)\tau_in^i-\delta_i{n-k\choose i}\tau_i\right)\right]\\
    &\overset{(*)}{\leq} \exp_2\left[{n\choose k}+\sum_{i=1}^{k-2}\left(O_k(\log n)+H\left(\frac{\rho_i\cdot 2(i!)}{{k\choose i}}\right)n^i-\delta_i{n-k\choose i}\right)\tau_i\right]\\
    &\leq \exp_2\left[{n\choose k}-3cn^{k-2}\right],
\end{align*}
where (*) is due to the fact that $\tau_i\geq\frac{t_i}{\rho_i n^i}$. For every $i\in\{1,\dots,k-2\}$, there are at most $\frac{{n\choose k-i}}{i+1}$ possible values of $\nu_i$ and at most ${n\choose k-i}$ possible values of $\tau_i$. Hence the number of possible $G\in\cI_4^*(n, \zeta,\vec\beta)$ given $t_1,\dots,t_{k-2}$ is at most $\exp_2\left[{n\choose k}-3cn^{k-2}\right]\cdot\left(\prod_{i=1}^{k-2}\frac{{n\choose k-i}}{i+1}{n\choose k-i}\right)\leq \exp_2\left[{n\choose k}-2cn^{k-2}\right]$.
\end{proof}

\subsection{Bound on $|\cI^*(n, \zeta)|$}

\begin{proof}[Proof of Theorem~\ref{t:boundcistar}] 

Decomposing $\cI^*(n, \zeta)$, we have

\begin{align*}
    |\cI^*(n, \zeta)| 
    &\leq |\cI^*(n, \zeta)\setminus\cI_1^*(n, \zeta, \beta_1)|+|\cI_1^*(n, \zeta, \beta_1)\setminus\cI_2^*(n, \zeta,\vec\beta)| + |\cI_2^*(n, \zeta,\vec\beta)\setminus\cI_3^*(n, \zeta,\vec\beta)|\\ &+|\cI_3^*(n, \zeta,\vec\beta)\setminus\cI_4^*(n, \zeta,\vec\beta)|+|\cI_4^*(n, \zeta,\vec\beta)|.
\end{align*}

Recall that we have the following inequalities from a combination of \cref{l:step1}, \cref{l:step2.1}, \cref{l:step2.2}, \cref{l:step3} and \cref{l:step4}, where $\vec \beta > 0$ chosen so that $\beta_1$ is sufficiently small (in general and also with respect to $\beta_2, \ldots \beta_k$), $\zeta > 0$ is sufficiently small relative to $\vec \beta$, and $c'> c > 0$ are constants chosen to be small relative to $\zeta$ and $\vec \beta$.
\begin{flalign*}
 |\cI^*(n, \zeta)\setminus\cI_1^*(n, \zeta)|   &\leq \exp_2\sqb{\frac{1}{2}{n-1\choose k-1}}\abs{\cI(n-1)},\\
 |\cI_1^*(n, \zeta, \beta_1)\setminus\cI_2^*(n, \zeta,\vec\beta)| &\le (k-2)\exp_2\sqb{(1-c'){n\choose k}}+\exp_2\sqb{(1-c'){n\choose k-1}}|\cI(n-1)|
 \\&\quad+\sum_{i=2}^{k-2}\exp_2\sqb{i(1-c'){n\choose k-1}}|\cI(n-i)|,\\
 |\cI_2^*(n, \zeta,\vec\beta)\setminus\cI_3^*(n, \zeta,\vec\beta)|   &\le\exp_2\sqb{k(1-c'){n\choose k-1}}|\cI(n-k)|,\\
 |\cI_3^*(n, \zeta,\vec\beta)\setminus\cI_4^*(n, \zeta,\vec\beta)| &\le   \exp_2\sqb{{n\choose k}-c'n^{k-2}},\\
 |\cI_4^*(n, \zeta,\vec\beta)| &\le 2^{{n\choose k}}.
 \end{flalign*}
 Combining the above inequalities gives the desired upper bound.
 \begin{flalign*}
  |\cI^*(n, \zeta)|   &\le 2^{n\choose k}+\exp_2\sqb{(1-c){n\choose k}}+\sum_{i=1}^{k-2}\exp_2\sqb{i(1-c){n\choose k-1}}|\cI(n-i)|\\
    &\qquad+
    \exp_2\sqb{(1-c)k{n\choose k-1}} |\cI(n-k)|+\exp_2\sqb{{n\choose k}-cn^{k-2}}. \qedhere
\end{flalign*}
\end{proof}

This immediately gives the desired tight asymptotic.

\begin{proof}[Proof of Theorem~\ref{thm:strong-count}]

Choose some $\theta>\log_23$ with $\pi(\cF_k,\theta)= 1$ and some $\zeta<\zeta_0$ as in \cref{t:boundcistar}. We then choose $\epsilon>0$ as in \cref{prop:reducetoboundik}.

Recall from \cref{prop:reducetoboundik} that for $n$ sufficiently large, we have
$$|\cI(n)|\leq  2^{\left(1-\epsilon\right)\binom{n}{k}} + 2^{n} \abs{\cI^*(n, \zeta)}.$$
Then, substituting the bound in~\cref{t:boundcistar}, we obtain by induction that for some constant $\Delta,c_0>0$ and all $n$,
\begin{equation}\label{eq:strongub}
    \abs{\cI(n)}\leq (1+\Delta\cdot 2^{-c_0n})B(n),
\end{equation}
where $B(n)=2^{{n\choose k}+n}$.
\end{proof}


\begin{thebibliography}{10}

\bibitem{All07}
Peter Allen, \emph{Almost every 2-{SAT} function is unate}, Israel J. Math.
  \textbf{161} (2007), 311--346.

\bibitem{BBCLMS17}
J\'{o}zsef Balogh, Neal Bushaw, Maur\'{\i}cio Collares, Hong Liu, Robert
  Morris, and Maryam Sharifzadeh, \emph{The typical structure of graphs with no
  large cliques}, Combinatorica \textbf{37} (2017), 617--632.

\bibitem{BMS15}
J\'{o}zsef Balogh, Robert Morris, and Wojciech Samotij, \emph{Independent sets
  in hypergraphs}, J. Amer. Math. Soc. \textbf{28} (2015), 669--709.

\bibitem{BMS18}
J\'{o}zsef Balogh, Robert Morris, and Wojciech Samotij, \emph{The method of
  hypergraph containers}, Proceedings of the {I}nternational {C}ongress of
  {M}athematicians---{R}io de {J}aneiro 2018. {V}ol. {IV}. {I}nvited lectures,
  World Sci. Publ., Hackensack, NJ, 2018, pp.~3059--3092.

\bibitem{BS20}
J\'{o}zsef Balogh and Wojciech Samotij, \emph{An efficient container lemma},
  Discrete Anal. (2020), Paper No. 17, 56pp.

\bibitem{BB03}
B\'{e}la Bollob\'{a}s and Graham~R. Brightwell, \emph{The number of {$k$}-{SAT}
  functions}, Random Structures Algorithms \textbf{22} (2003), 227--247.

\bibitem{BBL03}
B\'{e}la Bollob\'{a}s, Graham~R. Brightwell, and Imre Leader, \emph{The number
  of 2-{SAT} functions}, Israel J. Math. \textbf{133} (2003), 45--60.

\bibitem{CGFS86}
F.~R.~K. Chung, R.~L. Graham, P.~Frankl, and J.~B. Shearer, \emph{Some
  intersection theorems for ordered sets and graphs}, J. Combin. Theory Ser. A
  \textbf{43} (1986), 23--37.

\bibitem{Col75}
George~E. Collins, \emph{Quantifier elimination for real closed fields by
  cylindrical algebraic decomposition}, Automata theory and formal languages
  ({S}econd {GI} {C}onf., {K}aiserslautern, 1975), 1975, pp.~134--183.

\bibitem{EFR86}
P.~Erd\H{o}s, P.~Frankl, and V.~R\"{o}dl, \emph{The asymptotic number of graphs
  not containing a fixed subgraph and a problem for hypergraphs having no
  exponent}, Graphs Combin. \textbf{2} (1986), 113--121.

\bibitem{EKR76}
P.~Erd\H{o}s, D.~J. Kleitman, and B.~L. Rothschild, \emph{Asymptotic
  enumeration of {$K_{n}$}-free graphs}, Colloquio {I}nternazionale sulle
  {T}eorie {C}ombinatorie ({R}ome, 1973), {T}omo {II}, 1976, pp.~19--27.

\bibitem{Fur92}
Zolt\'{a}n F\"{u}redi, \emph{The maximum number of edges in a minimal graph of
  diameter {$2$}}, J. Graph Theory \textbf{16} (1992), 81--98.

\bibitem{FM17}
Zolt\'{a}n F\"{u}redi and Zeinab Maleki, \emph{The minimum number of triangular
  edges and a symmetrization method for multiple graphs}, Combin. Probab.
  Comput. \textbf{26} (2017), 525--535.

\bibitem{FPS05}
Zolt\'{a}n F\"{u}redi, Oleg Pikhurko, and Mikl\'{o}s Simonovits, \emph{On
  triple systems with independent neighbourhoods}, Combin. Probab. Comput.
  \textbf{14} (2005), 795--813.

\bibitem{GL18}
Vytautas Gruslys and Shoham Letzter, \emph{Minimizing the number of triangular
  edges}, Combin. Probab. Comput. \textbf{27} (2018), 580--622.

\bibitem{Hak65}
S.~L. Hakimi, \emph{On the degrees of the vertices of a directed graph}, J.
  Franklin Inst. \textbf{279} (1965), 290--308.

\bibitem{IK09}
L.~Ilinca and J.~Kahn, \emph{On the number of 2-{SAT} functions}, Combin.
  Probab. Comput. \textbf{18} (2009), 749--764.

\bibitem{IK12}
L.~Ilinca and J.~Kahn, \emph{The number of 3-{SAT} functions}, Israel J. Math.
  \textbf{192} (2012), 869--919.

\bibitem{Kah01}
Jeff Kahn, \emph{An entropy approach to the hard-core model on bipartite
  graphs}, Combin. Probab. Comput. \textbf{10} (2001), 219--237.

\bibitem{Kat68}
G.~Katona, \emph{A theorem of finite sets}, Theory of graphs ({P}roc.
  {C}olloq., {T}ihany, 1966), 1968, pp.~187--207.

\bibitem{Kau10}
Manuel Kauers, \emph{How to use cylindrical algebraic decomposition}, S\'{e}m.
  Lothar. Combin. \textbf{65} (2010/12), Art. B65a, 16.

\bibitem{Kee11}
Peter Keevash, \emph{Hypergraph {T}ur\'{a}n problems}, Surveys in combinatorics
  2011, London Math. Soc. Lecture Note Ser., vol. 392, Cambridge Univ. Press,
  Cambridge, 2011, pp.~83--139.

\bibitem{Kru63}
Joseph~B. Kruskal, \emph{The number of simplices in a complex}, Mathematical
  optimization techniques, Univ. of California Press, Berkeley, Calif., 1963,
  pp.~251--278.

\bibitem{MCK81}
Brendan~D. McKay, \emph{Practical graph isomorphism}, (1981), pp.~45--87.  
  
 \bibitem{MCK14}
Brendan~D. McKay, \emph{Practical graph isomorphism II}, J. Symbolic Comput. \textbf{60} (2014), pp.~94--112.  

\bibitem{ST15}
David Saxton and Andrew Thomason, \emph{Hypergraph containers}, Invent. Math.
  \textbf{201} (2015), 925--992.

\bibitem{Sch03}
Alexander Schrijver, \emph{Combinatorial optimization. {P}olyhedra and
  efficiency}, Springer-Verlag, Berlin, 2003.

\end{thebibliography}

\appendix 

\section{Brute force verification of Conjecture~\ref{conj:BBL} for $k=5$ (by Nitya Mani and Edward Yu)}\label{a:comp}

\subsection{Overview} 
As observed in Remark~\ref{r:finite-check}, to verify that $\pi(\mcF_k, \theta) = 1$ for some fixed $k$ and some $\theta > \log_2 3$, 
it suffices to show that for some constant $n > k$, $\alpha(\vec{H}) + (\log_2 3) \beta(\vec{H}) \le 1$ over all $\vec T_k$-free $k$-PDGs $\vec{H}$ on exactly $n$ vertices (with $\alpha(\vec H)\binom{n}{k}$ undirected edges and $\beta(\vec H)\binom{n}{k}$ directed edges).

This is a finite check for any fixed $k$. The new result for 5-SAT relies on $(n,k)=(7,4)$, which took 6600 CPU hours (in parallel) after several optimizations. We describe the algorithm used to efficiently enumerate $7$-vertex $4$-PDGs and verify they were $\vec T_4$-free. The complete, documented code (along with installation and running instructions) can be found at \url{https://github.com/ThinGarfield/Density-k-PDG/tree/v10.1}.

Using this code, we computed for the following pairs $(n, k)$, the largest value of $\theta$ (which appears in the corresponding table entry) such that $\alpha(\vec{H}) + \theta \beta(\vec{H}) \le 1$ for all $n$-vertex $\vec T_k$-free $k$-PDGs $\vec{H}$.

\begin{table}[h!]
    \centering
    \begin{tabular}{@{}ccccccc@{}}
        \toprule
                  & \multicolumn{6}{c}{$k$} \\ \cmidrule{2-7}
        $n$   & 2 & 3 & 4 & 5 & 6 & 7 \\ \midrule
        2 & 1 &&&&&\\
    3 & $3/2$ & 1 &&&& \\
    4 & $3/2$ & $4/3$ & 1 &&& \\
    5 & $5/3$ & $5/3$ & $5/4$ & 1 & \\
    6 & $5/3$ & $5/3$ & $3/2$ & $6/5$ & 1 \\
    7 & $7/4$ & $7/4$ & $\mathbf{7/4}$ & $7/5$ & $7/6$ & 1  \\
    8 & $7/4$ & ? & ? & \textit{8/5} & $4/3$ & $8/7$ \\ \bottomrule
    \end{tabular}
\end{table}

In the above table, the $(n,k)=(8,5)$ entry does \textit{not} follow by computation (indeed, the projected size of such a calculation using our methods would be far beyond any reasonable computational resources). Rather, this follows by way of using the $(n,k)=(7,4)$ case as in Proposition~\ref{p:a1} to observe that $\pi(\vec{T}_5, 8/5) \le 1$ and finding a $\vec{T}_5$-free $8$ vertex graph $\vec H$ with only directed edges where $\alpha(\vec{H}) + \frac85 \beta(\vec{H}) = \frac85 \beta(\vec{H}) = 1$.


\subsection{$n = 7, k =4$}
Our computation enables us to verify Conjecture~\ref{conj:BBL} for $k = 5$.

\begin{prop}\label{p:a1}
We have that $\pi(\vec T_5, \theta) = 1$ for some $\theta > \log_2 3 $. Consequently the number of $5$-SAT functions is $(1 + o(1))2^{n + {n \choose k}}$.
\end{prop}
\begin{proof}
By computation (described in more details in the following sections), we verify that for all $7$-vertex $\vec T_4$-free graphs $\vec{H}$, $\alpha(\vec{H}) + \frac74 \beta(\vec{H}) \le 1$. This is a tight estimate for $n = 7$, achieved by the following $4$-PDG: 
\begin{align*}
\vec{H} = ([7],\{&123\wc6, 124\wc6,134\wc6, 234\wc6, 125\wc6,135\wc6, 235\wc6, 145\wc6, 245\wc6,\\
&123\wc7, 124\wc7,134\wc7, 234\wc7, 125\wc7,135\wc7, 235\wc7, 145\wc7, 245\wc7\})   
\end{align*}
By subsampling, this verifies that $\pi(\vec{T}_4, 7/4) = 1$. Applying Lemma~\ref{lem:vertex-averaging}, we find that 
$$1 \ge \pi(\vec T_4, 7/4) 
\ge \pi\left(\vec T_5, \frac{4 \cdot 7/4 + 1}{5} \right) 
= \pi(\vec T_5, 8/5).$$
Taking $\theta = 8/5 > \log_2 3$, the desired result  follows from Theorem~\ref{thm:strong-count}. 
\end{proof}

\subsection{Compute environment}
We implement this search in C++ 17 (specific details about the development environment and installation can be found in the code README).

We use the Google Cloud Platform Compute Engine to enumerate $k$-PDGs. Below we summarize the runtime for fixed $n, k$ of computing the maximal $\theta$ such that $\alpha(\vec{H}) + \theta \beta(\vec{H}) \le 1$ for all $n$-vertex, $\vec T_k$-free $k$-PDGs.
\begin{itemize}
    \item $n = 7, k = 3$ took ~880 CPU hours
    \item $n = 7, k = 4$ (the primary desired case) took ~6600 CPU hours in 30 batches, checking 29313 base graphs (described in more detail in the following subsection)
    \item $n = 7, k = 5$ took a few minutes on a regular computer
    \item $n = 8, k =5$ is not computationally tractable with the current algorithm (would require $\sim 10^{15}$ CPU hours with the current implementation and checking 4722759 base graphs). We think $n = 9, k =5$, an even more expensive computation, is the minimum $n$ that might verify Conjecture~\ref{conj:BBL} for $k = 6$.
    \item $n = 8, k = 6$ took ~200 CPU hours
\end{itemize}

\subsection{Algorithm}

We outline the algorithm below. Further documentation and associated testing can be found in the codebase.

\begin{alg}
Principally, we enumerate in a way to reduce the number of ``base $k$-PDGs'' that must be checked by growing a search tree of $k$-PDGs and eliminate unnecessary steps in computation. These optimizations are crucial for making the problem more computationally tractable.

\begin{enumerate}
\item Initialization: Let $\mcG_{k-1}$ be the set comprising the single PDG formed by taking a collection of vertices $\{1, \ldots, k-1\}$ with no edges. 
\item For $\ell \in \{k, k+1, \ldots, n-1\}$:
\begin{itemize}
    \item Initialize $\mcG_{\ell} := \left\{(V(\vec{G}) \cup \{\ell\}, E(\vec{G})) : \vec{G} \in \mcG_{\ell-1}\right\}$
    \item For each $\vec{G} \in \mcG_{\ell}$:
    \begin{itemize}
    \item Consider all possible combinations of edges (undirected or directed towards any vertex) that could be added to each $\vec{G} \in \mcG_{\ell}$ that include vertex $\ell$. Let $\mcE_{\ell}$ be the collection of possible edge sets to add.
    \item For each potential collection of edges $F \in \mcE_{\ell}$, consider candidate $\vec{H}_F = (V(\vec{G}), E(\vec{G}) \cup F)$, and check if $\vec{H}$ is $\vec T_k$-free. If not, discard the associated $\vec{H}_F$. If so, express $\vec{H}_F$ in a \textit{canonical form}, which is a representation chosen to ease in tracking non-isomorphic $k$-PDGs.
    \item For each $\vec T_k$-free $\vec{H}_F$ not isomorphic to some existing element in $\mcG_{\ell}$, add $\vec{H}_F$ to $\mcG_{\ell}$.
    \end{itemize}
    \item Sort $\mcG_{\ell}$ with respect to the canonical form.
\end{itemize}
\item Repeat the above process for $\ell = n$, but this time do not represent PDGs in canonical form or check for isomorphisms (but verify that the $k$-PDGs in $\mcG_n$ are $\vec T_k$-free)
\item Return the maximum value of $\theta$ such that for all $\vec{H} \in \mcG_n$, $\alpha(\vec{H}) + \theta \beta(\vec{H}) \le 1$.
\end{enumerate}
\end{alg}

There are several implemented optimizations (verified by a variety of tests, outlined in the code \texttt{README}) that make the above algorithm computationally tractable. We outline some of the most important ones below.
\begin{enumerate}
\item We represent $k$-PDGs with a data structure that contains edges, vertex signatures, and a graph hash without pointer indirections, representing edges with a 16-bit struct. This representation makes it faster to verify that a $k$-PDG is $\vec T_k$-free, check for isomorphisms, and copy and grow $k$-PDGs.
\item As we grow our classes $\mcG_{\ell}$, for every $\vec{G} \in \mcG_{\ell}$ we keep track of the largest $\theta$ such that $\alpha(\vec{G}) + \theta \beta(\vec{G}) \le 1$.
\item We use bit mask manipulation and a bit representation of edges to efficiently verify if a $k$-PDG is $\vec T_k$-free. We also attempt to limit the number of supersets of edge sets that contain a copy of $\vec T_k$ that are checked.
\item We adapt some graph isomorphism checking strategies~\cite{MCK81,MCK14} to $k$-PDGs by hashing degree information of vertices in a current $k$-PDG in a permutation invariant way (which is maintained via a canonical form to represent $k$-PDGs). We call this hash the \textit{vertex signature}. We also compute a whole $k$-PDG \textit{graph hash.} These steps make it easier to verify if two $k$-PDGs are isomorphic (although both of these sets of hashes could match for non-isomorphic $k$-PDGs, in which case we check for the $k$-PDGs being identical as we walk through permutations that align in vertex signature).
\end{enumerate}


\end{document}